\documentclass[reqno,12pt,letterpaper]{amsart}
\usepackage{amsmath,amssymb,amsthm,url,bookmark,hyperref}
\usepackage[pdftex]{graphicx}

\newtheorem{theo}{Theorem}
\newtheorem{prop}{Proposition}[section]
\newtheorem{defi}{Definition}[section]
\numberwithin{equation}{section}

\DeclareMathOperator{\Real}{Re}
\DeclareMathOperator{\Imag}{Im}
\DeclareMathOperator{\loc}{loc}
\DeclareMathOperator{\comp}{comp}
\DeclareMathOperator{\const}{const}
\DeclareMathOperator{\supp}{supp}
\DeclareMathOperator{\Hol}{Hol}
\DeclareMathOperator{\cl}{cl}
\DeclareMathOperator{\WFF}{WF}
\def\WF{\WFF_h}
\DeclareMathOperator{\Even}{Even}
\DeclareMathOperator{\Odd}{Odd}
\DeclareMathOperator{\Ker}{Ker}

\DeclareMathOperator{\ad}{ad}
\DeclareMathOperator{\GL}{GL}
\DeclareMathOperator{\sgn}{sgn}
\DeclareMathOperator{\Res}{Res}

\title[Asymptotic distribution of QNMs for Kerr--de Sitter]%
{Asymptotic distribution of quasi-normal modes\\
for Kerr--de Sitter black holes}
\author{Semyon Dyatlov}
\email{dyatlov@math.berkeley.edu}
\address{Department of Mathematics, Evans Hall, University of California,
Berkeley, CA 94720, USA}

\begin{document}

\begin{abstract}
We establish a Bohr--Sommerfeld type condition for quasi-normal modes
of a slowly rotating Kerr--de Sitter black hole, providing their full
asymptotic description in any strip of fixed width. In particular, we
observe a Zeeman-like splitting of the high multipicity modes at $a=0$
(Schwarzschild--de Sitter), once spherical symmetry is broken.  The
numerical results presented in Appendix~\ref{s:numerics} show that the
asymptotics are in fact accurate at very low energies and agree with
the numerical results established by other methods in the physics
literature.  We also we prove that solutions of the wave equation can
be asymptotically expanded in terms of quasi-normal modes; this
confirms the validity of the interpretation of their real
parts as frequencies of oscillations, and imaginary parts as decay
rates of gravitational waves.
\end{abstract}

\maketitle


\addcontentsline{toc}{section}{Introduction}

Quasi-normal modes (QNMs) of black holes are a topic of continued
interest in theoretical physics: from the classical interpretation as
ringdown of gravitational waves~\cite{Chan} to the recent	
investigations in the context of string theory~\cite{HoHu}.  The
ringdown\footnote{Here is an irresistible quote of Chandrasekhar
``\dots we may expect that any intial perturbation will, during its last
stages, decay in a manner characteristic of the black hole itself and
independent of the cause. In other words, we may expect that during
these last stages, the black hole emits gravitational waves with
frequencies and rates of damping that are characteristic of the black
hole itself, in the manner of a bell sounding its last dying notes.''} 
plays a role in experimental projects aimed at the detection of
gravitational waves, such as LIGO~\cite{Ligo}. See~\cite{k-s} for an
overview of the vast physics literature on the topic
and~\cite{b-c-s,k-z-2,k-z-1,y-u-f} for some more recent developments.

In this paper we consider the Kerr--de Sitter model of a rotating
black hole and assume that the speed of rotation $a$ is small; for
$a=0$, one gets the stationary Schwarzschild--de Sitter black hole.
The de Sitter model corresponds to assuming that the cosmological
constant $\Lambda$ is positive, which is consistent with the current
Lambda-CDM standard model of cosmology.

A rigorous definition of quasi-normal modes for Kerr--de Sitter black
holes was given using the scattering resolvent in~\cite{skds}.  In
Theorem 1 below we give an asymptotic description of QNMs in a band of
any fixed width, that is, for any bounded decay rate.  The result
confirms the heuristic analogy with the Zeeman effect: the high
multiplicity modes for the Schwarzschild black hole split.

Theorem 2 confirms the standard interpretation of QNMs as complex
frequencies of exponentially decaying gravitational waves; namely, we
show that the solutions of the scalar linear wave equation in the
Kerr--de Sitter background can be expanded in terms of QNMs.

In the mathematics literature quasi-normal modes of black holes were
studied by Bachelot, Motet-Bachelot, and Pravica~\cite{b-1,b-2,b-3,p}
using the methods of scattering theory. QNMs of Schwarzschild--de
Sitter metric were then investigated by S\'a
Barreto--Zworski~\cite{sb-z}, resulting in the lattice of pseudopoles
given by~\eqref{e:sb-z-q} below. For this case,
Bony--H\"afner~\cite{b-h} established polynomial cutoff resolvent
estimates and a resonance expansion, Melrose--S\'a
Barreto--Vasy~\cite{m-sb-v} obtained exponential decay for solutions
to the wave equation up to the event horizons, and
Dafermos--Rodnianski~\cite{d-r3} used physical space methods to obtain
decay of linear waves better than any power of $t$.

Quasi-normal modes for Kerr--de Sitter were rigorously defined
in~\cite{skds} and exponential decay beyond event horizons was proved
in~\cite{xpd}.  Vasy~\cite{v} has recently obtained a microlocal
description of the scattering resolvent and in particular recovered
the results of~\cite{skds,xpd} on meromorphy of the resolvent and
exponential decay; see~\cite[Appendix]{v} for how his work relates
to~\cite{skds}. The crucial component for obtaining exponential decay
was the work of Wunsch--Zworski~\cite{w-z} on resolvent estimates for
normally hyperbolic trapping.

We add that there have been many papers on decay of linear waves for
Schwarzschild and Kerr black holes~--- see
\cite{a-b,b-s,d-r,d-r2,dss,dss1,fksy,fksy:erratum,ta,t-t,to}
and references given there. In that case the cosmological constant is
0 (unlike in the de Sitter case, where it is positive), and the
methods of scattering theory are harder to apply because of an
asymptotically Euclidean infinity.
\begin{theo}\label{l:theorem-asymptotics}
Fix the mass $M_0$ of the black hole and the cosmological constant
$\Lambda$. (See Section~\ref{s:k-ds} for details.)  Then there exists
a constant $a_0>0$ such that for $|a|<a_0$ and each $\nu_0$,
there exist constants $C_\omega,C_m$%
\footnote{As in~\cite{skds}, the indices $\omega,m,\dots$ next to
constants, symbols, operators, and functions do not imply differentiation.}
such that the set of quasi-normal modes $\omega$ satisfying
\begin{equation}\label{e:region-0}
\Real\omega>C_\omega,\
\Imag\omega>-\nu_0
\end{equation}
coincides modulo $O(|\omega|^{-\infty})$ with the set of pseudopoles
\begin{equation}\label{e:our-q}
\omega=\mathcal F(m,l,k),\
m,l,k\in \mathbb Z,\
0\leq m\leq C_m,\
|k|\leq l.
\end{equation}
(Since the set of QNMs is symmetric with respect to the imaginary axis,
one also gets an asymptotic description for $\Real\omega$ negative. Also,
by~\cite[Theorem~4]{skds}, all QNMs lie in the lower half-plane.)
Here $\mathcal F$ is a complex valued classical symbol%
\footnote{Here `symbol' means a microlocal symbol
as in for example~\cite[Section~8.1]{tay}. For the proofs, however,
we will mostly use semiclassical
symbols, as defined in Section~\ref{s:prelim-pseudor}.} of order 1
in the $(l,k)$ variables, defined and smooth in the cone
$\{m\in [0,C_m],\ |k|\leq l\}\subset \mathbb R^3$. 
The principal symbol $\mathcal F_0$
of $\mathcal F$ is real-valued and
independent of $m$; moreover,
\begin{gather}\label{e:sb-z-q}
\mathcal F={\sqrt{1-9\Lambda M_0^2}\over 3\sqrt 3 M_0}
[(l+1/2)-i(m+1/2)]+O(l^{-1})\text{ for }a=0,\\
\label{e:qc-a}
(\partial_k \mathcal F_0)(m,\pm k,k)
={(2+9\Lambda M_0^2)a\over 27M_0^2}+O(a^2).
\end{gather}
\end{theo}

The pseudopoles~\eqref{e:our-q} can be computed numerically; we have
implemented this computation in a special case $l-|k|=O(1)$ and
compared the pseudopoles with the QNMs computed by the authors
of~\cite{b-c-s}.  The results are described in
Appendix~\ref{s:numerics}.  One should note that the quantization
condition of~\cite{sb-z} was stated up to~$O(l^{-1})$ error, while
Theorem~\ref{l:theorem-asymptotics} has error $O(l^{-\infty})$; we
demonstrate numerically that increasing the order of the quantization
condition leads to a substantially better approximation.

Another difference between~\eqref{e:our-q} and the quantization
condition of~\cite{sb-z} is the extra parameter $k$, resulting from
the lack of spherical symmetry of the problem.  In fact, for $a=0$
each pole in~\eqref{e:sb-z-q} has multiplicity $2l+1$; for $a\neq 0$
this pole splits into $2l+1$ distinct QNMs, each corresponding to its
own value of $k$, the angular momentum with respect to the axis of
rotation. (The resulting QNMs do not coincide for small values of $a$,
as illustrated by~\eqref{e:qc-a}).  In the physics literature this is
considered an analogue of the Zeeman effect.

Since the proof of Theorem~\ref{l:theorem-asymptotics} only uses
microlocal analysis away from the event horizons, it implies
estimates on the cutoff resolvent polynomial in $\omega$
(Proposition~\ref{l:main-lemma-2}).  Combining these with the detailed
analysis away from the trapped set (and in particular near the event
horizons) by Vasy~\cite{v}, we obtain estimates on the resolvent on
the whole space (Proposition~\ref{l:main-lemma}).  These in turn allow
a contour deformation argument leading to an expansion of waves in
terms of quasinormal modes.  Such expansions have a long tradition in
scattering theory going back to Lax--Phillips and Vainberg~---
see~\cite{t-z} for the strongly trapping case and for references.

For Schwarzschild-de Sitter black holes a full expansion involving
infinite sums over quasinormal modes was obtained in~\cite{b-h} (see
also \cite{c-z} for simpler expansions involving infinite sums over
resonances).  The next theorem presents an expansion of waves for
Kerr--de Sitter black holes in the same style as the Bony--H\"afner
expansion:
\begin{theo}\label{l:theorem-resdec}
Under the assumptions of Theorem~\ref{l:theorem-asymptotics}, take
$\nu_0>0$ such that for some $\varepsilon>0$, every QNM $\omega$ has
$|\Imag\omega+\nu_0|>\varepsilon$. (Such $\nu_0$ exists and can be
chosen arbitrarily large, as the imaginary parts of QNMs lie within
$O(|a|+l^{-1})$ of those in~\eqref{e:sb-z-q}.)  Then for $s$ large
enough depending on $\nu_0$, there exists a constant $C$ such that
every solution $u$ to the Cauchy problem on the Kerr--de Sitter space
\begin{equation}\label{e:cauchy-problem}
\Box_g u=0,\
u|_{t^*=0}=f_0\in H^{s}(X_\delta),\
\partial_{t^*}u|_{t^*=0}=f_1\in H^{s-1}(X_\delta),
\end{equation}
where $X_{\delta}=(r_--\delta,r_++\delta)\times \mathbb S^2$ is
the space slice, $t^*$ is the time variable, and $\delta>0$ is a small constant
(see Section~\ref{s:k-ds} for details), satisfies for $t^*>0$,
\begin{equation}\label{e:resonance-est}
\|u(t^*)-\Pi_{\nu_0}(f_0,f_1)(t^*)\|_{H^1(X_\delta)}
\leq Ce^{-\nu_0 t^*}(\|f_0\|_{H^s}+\|f_1\|_{H^{s-1}}).
\end{equation}
Here
\begin{equation}\label{e:resonance-exp}
\Pi_{\nu_0}(f_0,f_1)(t^*)=\sum_{\Imag\widehat\omega>-\nu_0}
e^{-it^*\widehat\omega}\sum_{0\leq j<J_{\widehat\omega}} (t^*)^j
\Pi_{\widehat\omega,j}(f_0,f_1);
\end{equation}
the outer sum is over QNMs $\widehat\omega$, $J_{\widehat\omega}$ is the
algebraic multiplicity of $\widehat\omega$ as a pole of the scattering
resolvent, and $\Pi_{\widehat\omega,j}$ are finite rank operators mapping
$H^s(X_\delta)\oplus H^{s-1}(X_\delta)\to
C^\infty(X_\delta)$. Moreover, for $|\widehat\omega|$ large enough (that
is, for all but a finite number of QNMs in the considered strip),
$J_{\widehat\omega}=1$, $\Pi_{\widehat\omega,0}$ has rank one, and
$$
\|\Pi_{\widehat\omega,0}\|_{H^s(X_\delta)\oplus H^{s-1}(X_\delta)
\to H^1(X_\delta)}\leq C|\widehat\omega|^{N-s}.
$$
Here $N$ is a constant depending on $\nu_0$, but not on $s$;
therefore, the series~\eqref{e:resonance-exp} converges in $H^1$ for
$s>N+2$.
\end{theo}
The proofs start with the Teukolsky separation of variables used
in~\cite{skds}, which reduces our problem to obtaining quantization
conditions and resolvent estimates for certain radial and angular
operators (Propositions~\ref{l:radial} and~\ref{l:angular}).  These
conditions are stated and used to obtain
Theorems~\ref{l:theorem-asymptotics} and~\ref{l:theorem-resdec} in
Section~\ref{s:separation}. Also, at the end of
Section~\ref{s:separation2} we present the separation argument in the
simpler special case $a=0$, for convenience of the reader.

In the spherically symmetric case $a=0$, the angular problem is the
eigenvalue problem for the Laplace--Beltrami operator on the round
sphere. For $a\neq 0$, the angular operator $P_\theta$
is not selfadjoint; however, in the semiclassical scaling it is an
operator of real principal type with completely integrable Hamiltonian
flow. We can then use some of the methods of~\cite{h-s} to obtain a
microlocal normal form for $h^2P_\theta$; since our perturbation is
$O(h)$, we are able to avoid using analyticity of the coefficients of
$P_\theta$.  The quantization condition we get is global, similarly
to~\cite{svn}. The proof is contained in Section~\ref{s:angular}; it
uses various tools from semiclassical analysis described in
Section~\ref{s:prelim}.

To complete the proof of the angular quantization condition, we need
to extract information about the joint spectrum of $h^2P_\theta$ and
$hD_\varphi$ from the microlocal normal form; for that, we formulate a
Grushin problem for several commuting operators.  The problem that
needs to be overcome here is that existence of joint spectrum is only
guaranteed by exact commutation of the original operators, while
semiclassical methods always give $O(h^\infty)$ errors. This
complication does not appear in~\cite{h-s,h-s-svn} as they study the
spectrum of a single operator, nor in earlier works~\cite{ch,svn} on
joint spectrum of differential operators, as they use spectral theory
of selfadjoint operators.  Since this part of the construction can be
formulated independently of the rest, we describe Grushin problems for
several operators in an abstract setting in Appendix~\ref{s:prelim-grushin}.

The radial problem is equivalent to one-dimensional semiclassical
potential scattering.  The principal part of the potential is
real-valued and has a unique quadratic maximum; the proof of the
quantization condition follows the methods developed
in~\cite{cdv-p,ra,sj2}.  In~\cite{cdv-p}, the microlocal behavior of
the principal symbol near a hyperbolic critical point is studied in
detail; however, only self-adjoint operators are considered and the
phenomenon that gives rise to resonances in our case does not
appear. The latter phenomenon is studied in~\cite{ra} and~\cite{sj2};
our radial quantization condition, proved in Section~\ref{s:radial},
can be viewed as a consequence of~\cite[Theorems~2
and~4]{ra}. However, we do not compute the scattering matrix, which
simplifies the calculations; we also avoid using analyticity of the
potential near its maximum and formulate the quantization condition by
means of real microlocal analysis instead of the action integral in
the complex plane.  As in~\cite{ra}, we use analyticity of the
potential near infinity and the exact WKB method to relate the
microlocal approximate solutions to the outgoing condition at
infinity; however, the construction is somewhat simplified compared
to~\cite[Sections~2 and~3]{ra} using the special form of the
potential.

It would be interesting to see whether our statements still hold if
one perturbs the metric, or if one drops the assumption of smallness
of $a$. Near the event horizons, we rely on~\cite[Section~6]{skds},
which uses a perturbation argument (thus smallness of $a$) and
analyticity of the metric near the event horizons.  Same applies to
Section~\ref{s:radial-wkb} of the present paper; the exact WKB
construction there requires analyticity and
Proposition~\ref{l:u-pm-wkb} uses that the values $\omega_\pm$ defined
in~\eqref{e:omega-pm} are nonzero, which might not be true for large
$a$. However, it is very possible that the construction of the
scattering resolvent of~\cite{v} can be used instead. The methods
of~\cite{v} are stable under rather general perturbations,
see~\cite[Section~2.7]{v}, and apply in particular to Kerr--de Sitter
black holes with $a$ satisfying~\cite[(6.12)]{v}.

A more serious problem is the fact that
Theorem~\ref{l:theorem-asymptotics} is a quantization condition, and
thus is expected to hold only when the geodesic flow is completely
integrable, at least on the trapped set. For large $a$, the separation
of variables of Section~\ref{s:separation2} is still valid, and it is
conceivable that the global structure of the angular integrable system
in Section~\ref{s:angular-hamiltonian} and of the radial barrier-top
Schr\"odinger operator in Section~\ref{s:radial-trapping} would be
preserved, yielding Theorem~\ref{l:theorem-asymptotics} in this case.
Even then, the proof of Theorem~\ref{l:theorem-resdec} no longer
applies as it relies on having gaps between the imaginary parts of
resonances, which might disappear for large $a$.

However, a generic smooth perturbation of the metric supported near
the trapped set will destroy complete integrability and thus any hope
of obtaining Theorem~\ref{l:theorem-asymptotics}. One way of dealing
with this is to impose the condition that the geodesic flow is
completely integrable on the trapped set. In principle, the global
analysis of~\cite{svn} together with the methods for handling $O(h)$
nonselfadjoint perturbations developed in Section~\ref{s:angular} and
Appendix~\ref{s:prelim-grushin} should provide the quantization
condition in the direction of the trapped set, while the barrier-top
resonance analysis of Section~\ref{s:radial-barrier} should handle the
transversal directions.  However, without separation of variables one
might need to merge these methods and construct a normal form at the
trapped set which is not presented here.

Another possibility is to try to establish
Theorem~\ref{l:theorem-resdec} without a quantization condition,
perhaps under the (stable under perturbations) assumption that the
trapped set is normally hyperbolic as in~\cite{w-z}.  However, this
will require to rethink the contour deformation argument, as it is not
clear which contour to deform to when there is no stratification of
resonances by depth, corresponding to the parameter $m$ in
Theorem~\ref{l:theorem-asymptotics}.

\section{Proofs of Theorems~\ref{l:theorem-asymptotics}
and~\ref{l:theorem-resdec}}\label{s:separation}

\subsection{Kerr--de Sitter metric}\label{s:k-ds}

First of all, we define Kerr--de Sitter metric and briefly review how
solutions of the wave equation are related to the scattering
resolvent; see also~\cite[Section~1]{skds} and~\cite[Section~6]{v}.
The metric is given by
$$
\begin{gathered}
g=-\rho^2\Big({dr^2\over \Delta_r}+{d\theta^2\over\Delta_\theta}\Big)
-{\Delta_\theta\sin^2\theta\over (1+\alpha)^2\rho^2}
(a\,dt-(r^2+a^2)\,d\varphi)^2\\
+{\Delta_r\over (1+\alpha)^2\rho^2}
(dt-a\sin^2\theta\,d\varphi)^2.
\end{gathered}
$$
Here $\theta\in [0,\pi]$ and $\varphi\in \mathbb R/2\pi \mathbb Z$ are
the spherical coordinates on $\mathbb S^2$ and $r,t$ take values in
$\mathbb R$; $M_0$ is the mass of the black hole, $\Lambda$ is the
cosmological constant, and $a$ is the angular momentum;
$$
\begin{gathered}
\Delta_r=(r^2+a^2)\Big(1-{\Lambda r^2\over 3}\Big)-2M_0r,\
\Delta_\theta=1+\alpha\cos^2\theta,\\
\rho^2=r^2+a^2\cos^2\theta,\
\alpha={\Lambda a^2\over 3}.
\end{gathered}
$$
The metric in the $(t,r,\theta,\varphi)$ coordinates is defined for
$\Delta_r>0$; we assume that this happens on an open interval $r\in
(r_-,r_+)$, where $r_\pm$ are two of the roots of the fourth order
polynomial equation $\Delta_r(r)=0$. The metric becomes singular at
$r=r_\pm$; however, this apparent singularity goes away if we consider
the following version of the Kerr-star coordinates
(see~\cite[Section~5.1]{d-r} and~\cite{t-t}):
\begin{equation}\label{e:kerr-star}
t^*=t-F_t(r),\
\varphi^*=\varphi-F_\varphi(r),
\end{equation}
with the functions $F_t,F_\varphi$ blowing up like $c_\pm\log
|r-r_\pm|$ as $r$ approaches $r_\pm$. One can choose $F_t,F_\varphi$
so that the metric continues smoothly across the surfaces
$\{r=r_\pm\}$, called event horizons, to
$$
M_\delta=\mathbb R_t\times X_\delta,\
X_\delta=(r_--\delta,r_++\delta)\times \mathbb S^2,
$$
with $\delta>0$ is a small constant. Moreover, the surfaces
$\{t^*=\const\}$ are spacelike, while the surfaces $\{r=\const\}$ are
timelike for $r\in (r_-,r_+)$, spacelike for $r\not\in [r_-,r_+]$, and
null for $r\in \{r_-,r_+\}$.  See~\cite[Section~1]{skds},
\cite[Section~1.1]{xpd}, or~\cite[Section~6.4]{v} for more information
on how to construct $F_t,F_\varphi$ with these properties.

Let $\Box_g$ be the d'Alembert--Beltrami operator of the Kerr--de
Sitter metric.  Take $f\in H^{s-1}(M_\delta)$ for some $s\geq 1$, and
furthermore assume that $f$ is supported in $\{0\leq t^*\leq
1\}$. Then, since the boundary of $M_\delta$ is spacelike and every
positive time oriented vector at $\partial M_\delta$ points outside of
$M_\delta$, by the theory of hyperbolic equations (see for
example~\cite[Proposition~3.1.1]{d-r}
or~\cite[Sections~2.8 and~7.7]{tay}) there exists unique solution
$u\in H^s_{\loc}(M_\delta)$ to the problem
\begin{equation}\label{e:we-rhs}
\Box_g u=f,\ \supp u\subset \{t^*\geq 0\}.
\end{equation}
We will henceforth consider the problem~\eqref{e:we-rhs}; the Cauchy
problem~\eqref{e:cauchy-problem} can be reduced to~\eqref{e:we-rhs} as
follows. Assume that $u$ solves~\eqref{e:cauchy-problem} with some
$f_0\in H^s(X_\delta)$, $f_1\in H^{s-1}(X_\delta)$.  Take a function
$\chi\in C^\infty(\mathbb R)$ such that $\supp\chi\subset \{t^*>0\}$
and $\supp (1-\chi)\subset \{t^*<1\}$; then $\chi(t^*)u$
solves~\eqref{e:we-rhs} with $f=[\Box_g,\chi] u$ supported in $\{0\leq
t^*\leq 1\}$ and the $H^{s-1}$ norm of $f$ is controlled by
$\|f_0\|_{H^s}+\|f_1\|_{H^{s-1}}$.

Since the metric is stationary, there exists a constant $C_e$ such
that every solution $u$ to~\eqref{e:we-rhs} grows slower than
$e^{(C_e-1)t^*}$; see~\cite[Proposition~1.1]{skds}. Therefore, the
Fourier--Laplace transform
$$
\hat u(\omega)=\int e^{i\omega t^*} u(t^*)\,dt^*
$$
is well-defined and holomorphic in $\{\Imag\omega\geq C_e\}$.  Here
both $u(t^*)$ and $u(\omega)$ are functions on $X_\delta$.  Moreover,
if $\hat f(\omega)$ is the Fourier--Laplace transform of $f$, then
\begin{equation}\label{e:scattering-eq}
P_g(\omega)\hat u(\omega)=\rho^2\hat f(\omega),\
\Imag\omega\geq C_e,
\end{equation}
where $P_g(\omega)$ is the stationary d'Alembert--Beltrami operator,
obtained by replacing $D_{t^*}$ with $-\omega$ in $\rho^2\Box_g$. (The
$\rho^2$ factor will prove useful in the next subsection.)  Finally,
since $f$ is supported in $\{0\leq t^*\leq 1\}$, the function $\hat
f(\omega)$ is holomorphic in the entire $\mathbb C$, and
\begin{equation}\label{e:f-hat-estimate}
\|\langle\omega\rangle^{s-1}\hat f(\omega)\|_{H^{s-1}_{\langle\omega\rangle^{-1}}(X_\delta)}
\leq C \|f\|_{H^{s-1}(M_\delta)} 
\end{equation}
for $\Imag\omega$ bounded by a fixed constant. Here $H^{s-1}_h$,
$h>0$, is the semiclassical Sobolev space, consisting of the same
functions as $H^{s-1}$, but with norm $\|\langle
hD\rangle^{s-1}f\|_{L^2}$ instead of $\|\langle
D\rangle^{s-1}f\|_{L^2}$.

If $P_g(\omega)$ was, say, an elliptic operator, then the
equation~\eqref{e:scattering-eq} would have many solutions; however,
because of the degeneracies occuring at the event horizons, the
requirement that $\hat u\in H^s$ acts as a boundary condition. This
situation was examined in detail in~\cite{v}; the following
proposition follows from~\cite[Theorem~1.2 and Lemma~3.1]{v}
(see~\cite[Proposition~1.2]{skds} for the cutoff version):
\begin{prop}\label{l:mellin-transform}
Fix $\nu_0>0$. Then for $s$ large enough depending on $\nu_0$, there
exists a family of operators (called the scattering resolvent)
$$
R(\omega):H^{s-1}(X_\delta)\to H^{s}(X_\delta),\
\Imag\omega\geq -\nu_0,
$$
meromorphic with poles of finite rank and such that for $u$
solving~\eqref{e:we-rhs}, we have
\begin{equation}\label{e:scattering-eq2}
\hat u(\omega)=R(\omega)\hat f(\omega),\
\Imag\omega\geq C_e.
\end{equation}
\end{prop}
Note that even though we originally defined the left-hand side
of~\eqref{e:scattering-eq2} for $\Imag\omega\geq C_e$, the right-hand
side of this equation makes sense in a wider region $\Imag\omega\geq
-\nu_0$, and in fact in the entire complex plane if $f$ is smooth.
The idea now is to use Fourier inversion formula
\begin{equation}\label{e:fourier-inversion}
u(t^*)={1\over 2\pi}\int_{\Imag\omega=C_e}
e^{-i\omega t^*}R(\omega)\hat f(\omega)\,d\omega
\end{equation}
and deform the contour of integration to $\{\Imag\omega=-\nu_0\}$ to
get exponential decay via the $e^{-i\omega t^*}$ factor. We pick up
residues from the poles of $R(\omega)$ when deforming the contour;
therefore, \emph{one defines quasi-normal modes as the poles of
$R(\omega)$}.

Our ability to deform the contour and estimate the resulting integral
depends on having polynomial resolvent estimates. To formulate these,
let us give the technical
\begin{defi}\label{d:asymptotics}
Let $h>0$ be a parameter and $\mathcal R(\omega;h):\mathcal H_1\to
\mathcal H_2$, $\omega\in \mathcal U(h)\subset \mathbb C$, be a meromorphic
family of operators, with $\mathcal H_j$ Hilbert spaces. Let also
$\Omega(h)\subset \mathcal U(h)$ be open and $\mathcal Z(h)\subset
\mathbb C$ be a finite subset; we allow elements of $\mathcal Z(h)$ to
have multiplicities. We say that the poles of $\mathcal R$ in
$\Omega(h)$ are simple with a polynomial resolvent estimate and given
modulo $O(h^\infty)$ by $\mathcal Z(h)$, if for $h$ small enough,
there exist maps $\mathcal Q$ and $\Pi$ from $\mathcal Z(h)$ to
$\mathbb C$ and the algebra of bounded operators $\mathcal H_1\to
\mathcal H_2$, respectively, such that:
\begin{itemize}
\item for each $\widehat\omega'\in \mathcal Z(h)$,
$\widehat\omega=\mathcal Q(\widehat\omega')$
is a pole of $\mathcal R$, $|\widehat\omega-\widehat\omega'|=O(h^\infty)$, and
$\Pi(\widehat\omega')$ is a rank one operator;
\item there exists a constant $N$ such that
$\|\Pi(\widehat\omega')\|_{\mathcal H_1\to \mathcal H_2}=O(h^{-N})$ for
each $\widehat\omega'\in \mathcal Z(h)$ and, moreover,
$$
\mathcal R(\omega;h)=\sum_{\widehat\omega'\in \mathcal Z(h)}
{\Pi(\widehat\omega')\over \omega-\mathcal Q(\widehat\omega')}
+O_{\mathcal H_1\to \mathcal H_2}(h^{-N}),\
\omega\in\Omega(h). 
$$
In particular, every pole of $\mathcal R$ in $\Omega(h)$ lies in the
image of $\mathcal Q$.
\end{itemize}
\end{defi}
The quantization condition and resolvent estimate that we need to
prove Theorems~\ref{l:theorem-asymptotics} and~\ref{l:theorem-resdec}
are contained in
\begin{prop}\label{l:main-lemma}
Fix $\nu_0>0$ and let $h>0$ be a parameter.  Then for $a$ small
enough (independently of $\nu_0$), the poles of $R(\omega)$ in the region
\begin{equation}\label{e:region-1}
|\Imag\omega|<\nu_0,\
h^{-1}<|\Real\omega|<2h^{-1},
\end{equation}
are simple with a polynomial resolvent estimate and given modulo
$O(h^\infty)$ by
\begin{equation}\label{e:main-lemma-qc}
\begin{gathered}
\omega=h^{-1}\mathcal F^\omega(m,hl,hk;h),\
m,l,k\in \mathbb Z,\\
0\leq m\leq C_m,\
C_l^{-1}\leq hl\leq C_l,\
|k|\leq l.
\end{gathered}
\end{equation}
Here $C_m$ and $C_l$ are some constants and $\mathcal
F^\omega(m,\tilde l,\tilde k;h)$ is a classical symbol:
$$
\mathcal F^\omega(m,\tilde l,\tilde k;h)\sim \sum_{j\geq 0}
h^j \mathcal F^\omega_j(m,\tilde l,\tilde k).
$$
The principal symbol $\mathcal F^\omega_0$ is real-valued and
independent of $m$; moreover,
$$
\begin{gathered}
\mathcal F^\omega(m,\tilde l,\tilde k;h)={\sqrt{1-9\Lambda M_0^2}\over 3\sqrt 3 M_0}
(\tilde l+h/2-ih(m+1/2))+O(h^2)\text{ for }a=0,\\
(\partial_{\tilde k} \mathcal F^\omega_0)(m,\pm\tilde k,\tilde k)
={(2+9\Lambda M_0^2)a\over 27M_0^2}+O(a^2).
\end{gathered}
$$
Finally, if we consider $R(\omega)$ as a family of operators between
the semiclassical Sobolev spaces $H_h^{s-1}\to H_h^s$, then the
constant $N$ in Definition~\ref{d:asymptotics} is independent of $s$.
\end{prop}
Theorem~\ref{l:theorem-asymptotics} follows from the here almost
immediately.  Indeed, since $R(\omega)$ is independent of $h$, each
$\mathcal F^\omega_j$ is homogeneous in $(\tilde l,\tilde k)$
variables of degree $1-j$; we can then extend this function
homogeneously to the cone $|\tilde k|\leq \tilde l$ and define the
(non-semiclassical) symbol
$$
\mathcal F(m,l,k)\sim\sum_{j\geq 0} \mathcal F^\omega_j(m,l,k).
$$
Note that $\mathcal F(m,l,k)=h^{-1}\mathcal F^\omega(m,hl,hk;h)+O(h^\infty)$
whenever $C_l^{-1}\leq hl\leq C_l$.  We can then
cover the region~\eqref{e:region-0} for large $C_\omega$ with the
regions~\eqref{e:region-1} for a sequence of small values of $h$ to
see that QNMs in~\eqref{e:region-0} are given by~\eqref{e:our-q}
modulo $O(|\omega|^{-\infty})$.

Now, we prove Theorem~\ref{l:theorem-resdec}. Let $u$ be a solution
to~\eqref{e:we-rhs}, with $f\in H^{s-1}$ and $s$ large enough. We
claim that one can deform the contour in~\eqref{e:fourier-inversion}
to get
\begin{equation}\label{e:contour-deformation}
u(t^*)=i\sum_{\Imag\widehat\omega>-\nu_0}\Res_{\omega=\widehat\omega}
[e^{-i\omega t^*} R(\omega)\hat f(\omega)]
+{1\over 2\pi}\int_{\Imag\omega=-\nu_0}
e^{-i\omega t^*}R(\omega)\hat f(\omega)\,d\omega.
\end{equation}
The series in~\eqref{e:contour-deformation}
is over QNMs $\widehat\omega$; all but a finite
number of them in the region $\{\Imag\omega>-\nu_0\}$ are equal to
$\mathcal Q(\widehat\omega')$ for some $\widehat\omega'$ given
by~\eqref{e:our-q} and the residue in this case is
$e^{-i\widehat\omega t^*}\Pi(\widehat\omega')\hat f(\widehat\omega)$.
Here $\mathcal Q$ and $\Pi$ are taken from
Definition~\ref{d:asymptotics}.  Now, by~\eqref{e:f-hat-estimate}, we
have
$$
\begin{gathered}
\|\Pi(\widehat\omega')\hat f(\widehat\omega)\|_{H^1}\leq
C\langle\widehat\omega\rangle^{N-s}\|f\|_{H^{s-1}};\\
\|R(\omega)\hat f(\omega)\|_{H^1}\leq
C\langle\omega\rangle^{N-s}\|f\|_{H^{s-1}},\
\Imag\omega=-\nu_0,
\end{gathered}
$$
for some constant $N$ independent of $s$; therefore, for $s$ large
enough, the series in~\eqref{e:contour-deformation} converges in $H^1$
and the $H^1$ norm of the integral in~\eqref{e:contour-deformation}
can be estimated by~$Ce^{-\nu_0 t^*}\|f\|_{H^{s-1}}$, thus proving
Theorem~\ref{l:theorem-resdec}.

To prove~\eqref{e:contour-deformation}, take small $h>0$. There are
$O(h^{-2})$ QNMs in the region~\eqref{e:region-1}; therefore, by
pigeonhole principle we can find $\omega_0(h)\in [h^{-1},2h^{-1}]$
such that there are no QNMs $h^2$-close to the segments
$$
\gamma_\pm(h)=\{\Real\omega=\pm\omega_0(h),\
-\nu_0\leq \Imag\omega\leq C_e\}.
$$
Then $\|R(\omega)\hat f(\omega)\|_{H^1}=O(h^{s-N-4})$ on
$\gamma_\pm(h)$; we can now apply the residue theorem to the rectangle
formed from $\gamma_\pm(h)$ and segments of the lines
$\{\Imag\omega=C_e\},\{\Imag\omega=-\nu_0\}$, and then let $h\to 0$.

\subsection{Separation of variables}
\label{s:separation2}

First of all, using~\cite[(A.2), (A.3)]{v}, we reduce
Proposition~\ref{l:main-lemma} to the following\footnote{One could
also try to apply the results of~\cite{d-v} here, but we use the
slightly simpler construction of~\cite[Appendix]{v}, exploiting the
fact that we have information on the exact cutoff resolvent.}
\begin{prop}\label{l:main-lemma-2}
Take $\delta>0$ and put
$$
K_\delta=(r_-+\delta,r_+-\delta)\times \mathbb S^2,\
R_g(\omega)=1_{K_\delta}R(\omega)1_{K_\delta}:
L^2(K_\delta)\to H^2(K_\delta).
$$
Then for $a$ small enough\footnote{The smallness of $a$ is implied in
all following statements.} and fixed $\nu_0$, the poles of
$R_g(\omega)$ in the region~\eqref{e:region-1} are simple with a
polynomial resolvent estimate $L^2\to L^2$ and given modulo
$O(h^\infty)$ by~\eqref{e:main-lemma-qc}.
\end{prop}
Furthermore, by~\cite[Proposition~A.1]{v} the family of operators
$R_g(\omega)$ coincides with the one constructed
in~\cite[Theorem~2]{skds}, if the functions $F_t,F_\varphi$
in~\eqref{e:kerr-star} are chosen so that
$(t,\varphi)=(t^*,\varphi^*)$ in $\mathbb R_t\times K_\delta$.  We now
review how the construction of $R_g(\omega)$ in~\cite{skds} works and
reduce Proposition~\ref{l:main-lemma-2} to two separate spectral
problems in the radial and the angular variables. For the convenience
of reader, we include the simpler separation of variables
procedure for the case $a=0$ at the end of this section.

First of all, the operator $P_g(\omega)$ is invariant under the
rotation $\varphi\mapsto\varphi+s$; therefore, the spaces $\mathcal
D'_k=\Ker(D_\varphi-k)$ of functions of angular momentum $k\in \mathbb
Z$ are invariant under both $P_g(\omega)$ and
$R_g(\omega)$. In~\cite{skds}, we construct $R_g(\omega)$ by piecing
together the restrictions $R_g(\omega,k)=R_g(\omega)|_{\mathcal D'_k}$
for all $k$. Then, Proposition~\ref{l:main-lemma-2} follows from
\begin{prop}\label{l:main-lemma-3}
Under the assumptions of Proposition~\ref{l:main-lemma-2}, there
exists a constant $C_k$ such that for each $k\in \mathbb Z$,
\begin{enumerate}
\item if $h|k|>C_k$, then $R_g(\omega,k)$ has no poles in the region~\eqref{e:region-1}
and its $L^2\to L^2$ norm is $O(|k|^{-2})$; (This is a reformulation
of~\cite[Proposition~3.3]{skds}.)
\item if $h|k|\leq C_k$, then the poles of $R_g(\omega,k)$ in the region~\eqref{e:region-1}
are simple with a polynomial resolvent estimate $L^2\to L^2$
and given modulo $O(h^\infty)$
by~\eqref{e:main-lemma-qc}, with this particular value of $k$.
\end{enumerate}
\end{prop}
Now, we recall from~\cite[Section~1]{skds} that the restriction of
$P_g(\omega)$ to $\mathcal D'_k$ has the form\footnote{The operator
$P_g(\omega)$ of~\cite{skds} differs from our operator by the
conjugation done in~\cite[Appendix]{v}; however, the two coincide
in~$K_\delta$.} $P_r(\omega,k)+P_\theta(\omega)|_{\mathcal D'_k}$,
where
\begin{equation}\label{e:p-r-theta}
\begin{gathered}
P_r(\omega,k)=D_r(\Delta_rD_r)-{(1+\alpha)^2\over\Delta_r}((r^2+a^2)\omega-ak)^2,\\
P_\theta(\omega)={1\over\sin\theta}D_\theta(\Delta_\theta\sin\theta D_\theta)
+{(1+\alpha)^2\over \Delta_\theta\sin^2\theta}(a\omega\sin^2\theta-D_\varphi)^2
\end{gathered}
\end{equation}
are differential operators in $r$ and $(\theta,\varphi)$,
respectively.  Then $R_g(\omega,k)$ is constructed in~\cite[Proof of
Theorem~1]{skds} using a certain contour integral~\cite[(2.1)]{skds}
and the radial and angular resolvents
$$
\begin{gathered}
R_r(\omega,\lambda,k):L^2_{\comp}(r_-,r_+)\to H^2_{\loc}(r_-,r_+),\\
R_\theta(\omega,\lambda):L^2(\mathbb S^2)\to H^2(\mathbb S^2),\
\lambda\in \mathbb C;
\end{gathered}
$$
$R_r$ is a certain right inverse to $P_r(\omega,k)+\lambda$, while
$R_\theta$ is the inverse to $P_\theta(\omega)-\lambda$; we write
$R_\theta(\omega,\lambda,k)=R_\theta(\omega,\lambda)|_{\mathcal
D'_k}$. Recall that both $R_r$ and $R_\theta$ are meromorphic families
of operators, as defined in~\cite[Definition~2.1]{skds}; in
particular, for a fixed value of $\omega$, these families are
meromorphic in $\lambda$ with poles of finite rank. By definition of
$R_g(\omega,k)$, a number $\omega\in \mathbb C$ is a pole of this
operator if and only if there exists $\lambda\in \mathbb C$ such that
$(\omega,\lambda,k)$ is a pole of both $R_r$ and $R_\theta$.

Now, for small $h>0$ we put
\begin{equation}\label{e:semiclassical-parameters}
\tilde\omega=h\Real\omega,\
\tilde\nu=\Imag\omega,\
\tilde\lambda=h^2\Real\lambda,\
\tilde\mu=h\Imag\lambda,\
\tilde k=hk;
\end{equation}
the assumptions of Proposition~\ref{l:main-lemma-3}(2) imply that
$1\leq\tilde\omega\leq 2$, $|\tilde\nu|\leq\nu_0$, and $|\tilde k|\leq
C_k$.  Moreover, \cite[Proposition~3.4]{skds} suggests that under
these assumptions, all values of $\lambda$ for which
$(\omega,\lambda,k)$ is a pole of both $R_r$ and $R_\theta$ have to
satisfy $|\tilde\lambda|,|\tilde\mu|\leq C_\lambda$, for some constant
$C_\lambda$.

We are now ready to state the quantization conditions and resolvent
estimates for $R_r$ and $R_\theta$; the former is proved in
Section~\ref{s:radial} and the latter, in Section~\ref{s:angular}.
\begin{prop}[Radial lemma]\label{l:radial} 

Let $C_\lambda$ be a fixed constant and put
$K_r=(r_-+\delta,r_+-\delta)$.  Then the poles of
$1_{K_r}R_r(\omega,\lambda,k)1_{K_r}$ as a function of $\lambda$, in
the region
\begin{equation}\label{e:radial-regime}
1<\tilde\omega<2,\
|\tilde\nu|<\nu_0,\
|\tilde k|<C_k,\
|\tilde\lambda|,|\tilde\mu|<C_\lambda,
\end{equation}
are simple with polynomial resolvent estimate $L^2\to L^2$ (in the
sense of Definition~\ref{d:asymptotics}) and given modulo
$O(h^\infty)$ by
\begin{equation}\label{e:radial-qc}
\tilde\lambda+ih\tilde\mu=\mathcal F^r(m,\tilde\omega,\tilde\nu,\tilde k;h),\
m\in \mathbb Z,\
0\leq m\leq C_m,
\end{equation}
for some constant $C_m$. The principal part $\mathcal F^r_0$ of the
classical symbol $\mathcal F^r$ is real-valued, independent of $m$ and
$\tilde\nu$, and
$$
\begin{gathered}
\mathcal F^r=\bigg[ih(m+1/2)+{3\sqrt 3 M_0\over\sqrt{1-9\Lambda M_0^2}}
(\tilde\omega+ih\tilde\nu)\bigg]^2+O(h^2)\text{ for }a=0,\\
\mathcal F^r_0(\tilde\omega,\tilde k)={27 M_0^2\over 1-9\Lambda M_0^2}
\tilde\omega^2-{6a\tilde k\tilde\omega\over 1-9\Lambda M_0^2}+O(a^2).
\end{gathered}
$$
In particular, for $\omega,k$ satisfying~\eqref{e:radial-regime},
every pole $\lambda$ satisfies $\tilde\lambda>\varepsilon$ for some
constant $\varepsilon>0$.
\end{prop}
\begin{prop}[Angular lemma]\label{l:angular}
Let $C_\theta$ be a fixed constant.  Then the poles of
$R_\theta(\omega,\lambda,k)$ as a function of $\lambda$ in the region
\begin{equation}\label{e:angular-regime}
1<\tilde\omega<2,\ |\tilde\nu|<\nu_0,\
|\tilde k|<C_k,\
C_\theta^{-1}<\tilde\lambda< C_\theta,\
|\tilde\mu|<C_\theta,
\end{equation}
are simple with polynomial resolvent estimate $L^2\to L^2$ and given
modulo $O(h^\infty)$ by
\begin{equation}\label{e:angular-qc}
\tilde\lambda+ih\tilde\mu=\mathcal F^\theta(hl,\tilde\omega,\tilde\nu,\tilde k;h),\
l\in \mathbb Z,\
\max(|\tilde k|,C_l^{-1})\leq hl\leq C_l,\
\end{equation}
for some constant $C_l$.  The principal part $\mathcal F^\theta_0$ of
the classical symbol $\mathcal F^\theta$ is real-valued, independent
of $\tilde\nu$, and
$$
\mathcal F^\theta=\tilde l(\tilde l+h)+O(h^\infty)\text{ for }a=0.
$$
Moreover, $\mathcal F^\theta_0(\pm\tilde k,\tilde\omega,\tilde
k)=(1+\alpha)^2(\tilde k-a\tilde\omega)^2$, $\partial_{\tilde l} \mathcal
F^\theta_0(\pm\tilde k,\tilde\omega,\tilde k)=\pm 2\tilde k+O(a^2)$,
and consequently, $\partial_{\tilde k} \mathcal F^\theta_0(\pm\tilde
k,\tilde\omega,\tilde k)=-2a\tilde\omega+O(a^2)$.
\end{prop}
Combining Propositions~\ref{l:radial} and~\ref{l:angular}
with the results of~\cite{skds}, we get
\begin{proof}[Proof of Proposition~\ref{l:main-lemma-3}]
We let $\mathcal F^\omega(m,\tilde l,\tilde k;h)$ be the solution
$\tilde\omega+ih\tilde\nu$ to the equation
\begin{equation}\label{e:final-q-eq}
\mathcal F^r(m,\tilde\omega,\tilde\nu,\tilde k;h)=
\mathcal F^\theta(\tilde l,\tilde\omega,\tilde\nu,\tilde k;h).
\end{equation}
We can see that this equation has unique solution by writing $\mathcal
F^r-\mathcal F^\theta=\mathcal F'+ih \mathcal F''$ and examining the
principal parts of the real-valued symbols $\mathcal F',\mathcal F''$ for $a=0$.

The idea now is to construct an admissible contour in the sense
of~\cite[Definition~2.3]{skds}; e.g. a contour that separates the sets
of poles (in the variable $\lambda$) of $R_r$ and $R_\theta$ from each
other; then~\cite[(2.1)]{skds} provides a formula for $R_g(\omega,k)$,
which can be used to get a resolvent estimate.
\begin{figure}
\includegraphics{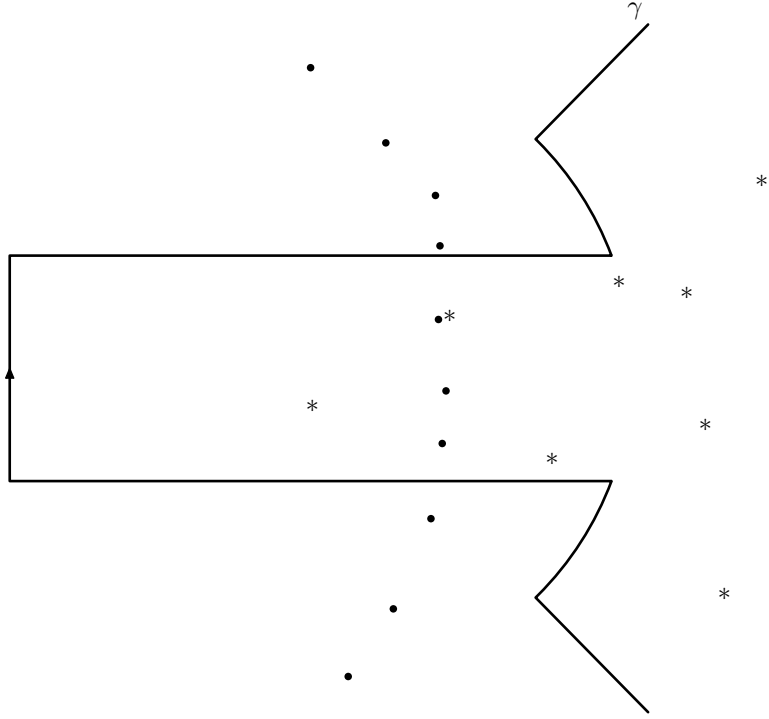}
\caption{The contour $\gamma$ and interaction between radial poles (denoted by dots)
and angular poles (denoted by asterisks).}
\end{figure}
We will use the method of proof of~\cite[Proposition~3.4]{skds}.  Take
the contour $\gamma$ introduced there, for $l_2=C_\lambda h^{-1}$,
$l_1=L=C_\lambda h^{-2}$, and $C_\lambda$ some large constant. Then we
know that all angular poles are to the right of $\gamma$ (in
$\Gamma_2$). Moreover, the only radial poles to the right of $\gamma$
lie in the domain $\{|\Imag\lambda|\leq l_2,\ |\lambda|\leq L\}$ and
they are contained in the set $\{\lambda^r_0,\dots,\lambda^r_{C_m}\}$
for some constant $C_m$, where $\lambda^r_m(\omega,k)$ is the radial
pole corresponding to $h^{-2}\mathcal
F^r(m,\tilde\omega,\tilde\nu,\tilde k;h)$.  In particular, those
radial poles are contained in
$$
U_\lambda=\{C_\theta^{-1}<\tilde\lambda<C_\theta,\ |\tilde\mu|\leq C_\theta\},
$$
for some constant $C_\theta$.

Assume that $\omega$ is not a pole of $R_g(\omega,k)$; then we can
consider the admissible contour composed of $\gamma$ and the circles
$\gamma_m$, $0\leq m\leq C_m$, enclosing $\lambda_m^r(\omega,k)$, but
none of the other poles of $R_r$ or $R_\theta$. Using the meromorphic
decomposition of $R_r$ at $\lambda_m^r$ and letting its principal part
be $\Pi^r_m/(\lambda-\lambda^r_m)$, we get
\begin{equation}\label{e:mer-dec-1}
R_g(\omega,k)=\sum_m \Pi^r_m(\omega,k) \otimes R_\theta(\omega,\lambda^r_m(\omega,k),k)
+{1\over 2\pi i}\int_\gamma
R_r(\omega,\lambda,k)\otimes R_\theta(\omega,\lambda,k)\,d\lambda.
\end{equation}
Here we only include the poles $\lambda_m^r$ lying to the right of
$\gamma$; one might need to change $l_1$ in the definition of $\gamma$
a little bit in case some $\lambda_m^r$ comes close to $\gamma$.  The
integral in~\eqref{e:mer-dec-1} is holomorphic and bounded
polynomially in $h$, by the bounds for $R_r$ given by
Proposition~\ref{l:radial}, together with the estimates in the proof
of~\cite[Proposition~3.4]{skds}.

Now, the poles of $R_\theta$ in $U_\lambda$ are given
by~\eqref{e:angular-qc}; let $\lambda^\theta_l(\omega,k)$ be the pole
corresponding to $h^{-2}\mathcal
F^\theta(hl,\tilde\omega,\tilde\nu,\tilde k;h)$ and
$\Pi^\theta_l/(\lambda-\lambda^\theta_l)$ be the principal part of the
corresponding meromorphic decomposition.  Then the resolvent estimates
on $R_\theta$ given by Proposition~\ref{l:angular} together
with~\eqref{e:mer-dec-1} imply
\begin{equation}\label{e:mer-dec-2}
R_g(\omega,k)=\Hol(\omega)+\sum_{m,l} {\Pi^r_m(\omega,k)
\otimes \Pi^\theta_l(\omega,k)
\over\lambda^r_m(\omega,k)-\lambda_l^\theta(\omega,k)}.
\end{equation}
Here $\Hol(\omega)$ is a family of operators holomorphic in $\omega$ and
bounded polynomially in $h$. Moreover,
\begin{equation}\label{e:mer-dec-3}
\lambda^r_m(\omega,k)-\lambda_l^\theta(\omega,k)
=h^{-2}(\mathcal F^r(m,\tilde\omega,\tilde\nu,\tilde k;h)
-\mathcal F^\theta(hl,\tilde\omega,\tilde\nu,\tilde k;h))
+O(h^\infty);
\end{equation}
therefore, the equation
$\lambda^r_m(\omega,k)-\lambda_l^\theta(\omega,k)=0$ is an
$O(h^\infty)$ perturbation of~\eqref{e:final-q-eq} and it has a unique
solution $\omega_{m,l}(k)$, which is $O(h^\infty)$ close to
$h^{-1}\mathcal F^\omega(m,hl,\tilde k;h)$. Finally, in the
region~\eqref{e:region-1} we can write by~\eqref{e:mer-dec-2}
$$
R_g(\omega,k)=\Hol(\omega)+\sum_{m,l}{\Pi_{m,l}(k)\over \omega-\omega_{m,l}(k)},
$$
with $\Hol(\omega)$ as above and $\Pi_{m,l}(k)$ being the product of a
coefficient polynomially bounded in $h$ with $(\Pi^r_m\otimes
\Pi^\theta_l)(\omega_{m,l}(k),k)$; this finishes the proof.
\end{proof}
Finally, let us present the simplified separation of variables for
the case $a=0$, namely the Schwarzschild--de Sitter metric:
$$
\begin{gathered}
g={\Delta_r\over r^2}dt^2-{r^2\over\Delta_r} dr^2-r^2(d\theta^2+\sin^2\theta\,d\varphi^2),\\
\Delta_r=r^2\Big(1-{\Lambda r^2\over 3}\Big)-2M_0r.
\end{gathered}
$$
Note that $d\theta^2+\sin^2\theta\,d\varphi^2$ is just the round metric on the
unit sphere. The metric decouples without the need to take Fourier series in
$\varphi$; the stationary d'Alembert--Beltrami operator has the form
$P_g(\omega)=P_r(\omega)+P_\theta$, where
$$
\begin{gathered}
P_r(\omega)=D_r(\Delta_r D_r)-{r^4\over\Delta_r}\omega^2,\\
P_\theta={1\over\sin\theta}D_\theta(\sin\theta D_\theta)+{D_\varphi^2\over\sin^2\theta}.
\end{gathered}
$$
Here $P_\theta$ is the Laplace--Beltrami operator on the round sphere;
it is self-adjoint (thus no need for the contour integral construction
of~\cite[Section~2]{skds}) and is known to have eigenvalues $l(l+1)$,
where $l\geq 0$. Each such eigenvalue has multiplicity $2l+1$,
corresponding to the values $-l,\dots,l$ of the $\varphi$-angular
momentum $k$. The angular Lemma~\ref{l:angular} follows
immediately. (We nevertheless give a more microlocal explanation
in this case at the end of Section~\ref{s:angular-outline}.)
One can now decompose $L^2$ into an orthogonal sum of the
eigenspaces of $P_\theta$; on the space $V_\lambda$ corresponding to
the eigenvalue $\lambda=l(l+1)$, we have
$$
P_g(\omega)|_{V_\lambda}=P_r(\omega)+\lambda.
$$
Therefore, the only problem is to show the radial Lemma~\ref{l:radial}
in this case, which is in fact no simpler than the general case. (Note
that we take a different path here than~\cite{sb-z} and~\cite{b-h},
using only real microlocal analysis near the trapped set, which
immediately gives polynomial resolvent bounds.)

\section{Preliminaries}\label{s:prelim}

\subsection{Pseudodifferential operators and microlocalization}\label{s:prelim-pseudor}
\relax

First of all, we review the classes of semiclassical
pseudodifferential operators on manifolds and introduce notation used
for these classes; see~\cite[Sections~9.3 and~14.2]{e-z}
or~\cite{d-s} for more information.

For $k\in \mathbb R$, we consider the symbol class $S^k(\mathbb R^n)$
consisting of functions $a(x,\xi;h)$ smooth in $(x,\xi)\in \mathbb
R^{2n}$ and satisfying the following growth conditions: for each
compact set $K\subset \mathbb R^n$ and each pair of multiindices
$\alpha,\beta$, there exists a constant $C_{\alpha\beta K}$ such that
$$
|\partial^\alpha_x \partial^\beta_\xi a(x,\xi;h)|\leq
C_{\alpha\beta K}\langle\xi\rangle^{k-|\beta|},\
x\in K,\ \xi\in \mathbb R^n,\ h>0.
$$
If we treat $\mathbb R^{2n}$ as the cotangent bundle to $\mathbb R^n$,
then the class $S^k$ is invariant under changes of variables; this
makes it possible, given a manifold $M$, to define the class $S^k(M)$
of symbols depending on $(x,\xi)\in T^*M$.

If $k\in \mathbb R$ and $a_j(x,\xi)\in S^{k-j}(M)$, $j=0,1,\dots$, is
a sequence of symbols, then there exists the asymptotic sum
\begin{equation}\label{e:asymptotic-sum}
a(x,\xi;h)\sim \sum_{j\geq 0} h^j a_j(x,\xi);
\end{equation}
i.e., a symbol $a(x,\xi;h)\in S^k(M)$ such that for every
$J=0,1,\dots$, $$ a(x,\xi;h)-\sum_{0\leq j<J} h^j a_j(x,\xi)\in
h^JS^{k-J}.
$$
The asymptotic sum $a$ is unique modulo the class $h^\infty
S^{-\infty}$ of symbols all of whose derivatives decay faster than
$h^N\langle\xi\rangle^{-N}$ for each $N$ on any compact set in $x$. If
$a$ is given by an asymptotic sum of the
form~\eqref{e:asymptotic-sum}, then we call it a classical symbol and
write $a\in S^k_{\cl}(M)$. We call $a_0(x,\xi)$ the principal part of
the symbol $a(x,\xi;h)$.

Let $\Psi^k(M)$ be the algebra of (properly supported)
semiclassical pseudodifferential operators on $M$ with symbols in $S^k(M)$.  If
$H^m_{h,\loc}(M)$, $m\in \mathbb R$, consists of functions locally
lying in the semiclassical Sobolev space, then every element of
$\Psi^k(M)$ is continuous $H^m_{h,\loc}(M)\to H^{m-k}_{h,\loc}(M)$
with every operator seminorm being $O(1)$ as $h\to 0$.  Let
$\Psi^k_{\cl}(M)$ be the algebra of operators with symbols in
$S^k_{\cl}(M)$ and $\Psi_{\cl}(M)$ be the union of $\Psi^k_{\cl}$ for
all $k$. Next, let the operator class $h^\infty\Psi^{-\infty}(M)$
correspond to the symbol class $h^\infty S^{-\infty}(M)$; it can be
characterized as follows: $A\in h^{\infty}\Psi^{-\infty}(M)$ if and
only if for each $N$, $A$ is continuous $H^{-N}_{h,\loc}(M)\to
H^N_{h,\loc}(M)$, with every operator seminorm being $O(h^N)$.  The
full symbol of an element of $\Psi^k(M)$ cannot be recovered as a
function on $T^*M$; however, if $A\in\Psi^k_{\cl}(M)$, then the
principal symbol of $A$ is an invariantly defined function on the
cotangent bundle. If $M$ is an open subset of $\mathbb R^n$, then we
can define the full symbol of a pseudodifferential operator modulo
$h^\infty S^{-\infty}$; we will always use Weyl quantization.

We now introduce microlocalization; see also~\cite[Section~8.4]{e-z}
and~\cite[Section~3]{sj-z}.  Define $U\subset T^*M$ to be conic at
infinity, if there exists a conic set $V\subset T^*M$ such that the
symmetric difference of $U$ and $V$ is bounded when restricted to
every compact subset of $M$.  For $a\in S^k(M)$ and $U\subset T^*M$
open and conic at infinity, we say that $a$ is rapidly decaying on
$U$, if for every $V\subset U$ closed in $T^*M$, conic at infinity,
and with compact projection onto $M$, every derivative of $a$ decays
on $V$ faster than $h^N
\langle\xi\rangle^{-N}$ for every $N$.  We say that $A\in\Psi^k(M)$
vanishes microlocally on $U$ if its full symbol (in any coordinate
system) is rapidly decaying on $U$.  If $A,B\in\Psi^k(M)$, then we say
that $A=B$ microlocally on $U$, if $A-B$ vanishes microlocally on $U$.

For $A\in\Psi^k(M)$, we define the semiclassical wavefront set
$\WF(A)\subset T^*M$ as follows: $(x,\xi)\not\in \WF(A)$ if and only
if $A$ vanishes microlocally on some neighborhood of $(x,\xi)$.  The
set $\WF(A)$ is closed; however, it need not be conic at infinity.
Next, we say that $A$ is compactly
microlocalized, if there exists a compact set $K\subset T^*M$ such
that $A$ vanishes microlocally on $T^*M\setminus K$. We denote by
$\Psi^{\comp}(M)$ the set of compactly microlocalized operators. Here
are some properties of microlocalization:
\begin{itemize}
\item If $A$ vanishes microlocally on $U_1$ and $U_2$, then it
vanishes microlocally on $U_1\cup U_2$.
\item The set of pseudodifferential operators vanishing microlocally
on some open and conic at infinity $U\subset T^*M$ is a two-sided
ideal; so is the set of operators with wavefront set contained in some
closed $V\subset T^*M$. In particular, $\WF(AB)\subset\WF(A)\cap\WF(B)$.
\item $A$ vanishes microlocally on the whole $T^*M$ if and only
if it lies in $h^\infty \Psi^{-\infty}$.
\item If $A$ vanishes microlocally on $U$, then $\WF(A)\cap U=\emptyset$;
the converse is true if $U$ is bounded.
However,%
\footnote{This issue can be avoided if we consider
$\WF(A)$ as a subset of the fiber compactified cotangent
bundle $\overline{T^*M}$, as in~\cite[Section~2.1]{v}.
Then an operator is compactly microlocalized if and only
if its wavefront set does not intersect the fiber infinity.}
$A$ does not necessarily vanish microlocally on the
complement of $\WF(A)$; for example, the operator $A=e^{-1/h}$ lies in
$\Psi^0$ and has an empty wavefront set, yet it does not lie in
$h^\infty\Psi^{-\infty}$.
\item The set $\Psi^{\comp}$ forms a two-sided ideal and it lies
in $\Psi^{-N}$ for every $N$.
\item Each $A\in\Psi^{\comp}$ vanishes microlocally on the complement
of $\WF(A)$.
\item Let $A\in\Psi^k_{\cl}(M)$ and let its symbol in some coordinate system
have the form~\eqref{e:asymptotic-sum}; introduce
$$
V=\bigcup_{j\geq 0} \supp a_j.
$$
Then $A$ vanishes microlocally on some open set $U$ if and only if $U\cap
V=\emptyset$; $A$ is compactly microlocalized if $V$ is bounded, and
$\WF(A)$ is the closure of $V$.
\end{itemize}
We now consider microlocally defined operators. Let $U\subset T^*M$ be
open. A local pseudodifferential operator $A$ on $U$ is, by
definition, a map
$$
B\mapsto 
[A\cdot B]\in\Psi^{\comp}(M)/h^\infty\Psi^{-\infty}(M),\ 
B\in\Psi^{\comp}(M),\ \WF(B)\subset U,
$$
such that:
\begin{itemize}
\item $\WF([A\cdot B])\subset \WF(B)$.
\item If $B_1,B_2\in\Psi^{\comp}(M)$ and $\WF(B_j)\subset U$, then
$[A\cdot (B_1+B_2)]=[A\cdot B_1]+[A\cdot B_2]$.
\item If $C\in\Psi^k(M)$, then $[A\cdot B]C=[A\cdot (BC)]$.
\end{itemize}
We denote by $\Psi^{\loc}(U)$ the set of all local operators on $U$.
Note that a local operator is only defined modulo an
$h^\infty\Psi^{-\infty}$ remainder. If $A\in\Psi^k(M)$, then the
corresponding local operator $\widetilde A$ is given by $[\widetilde A\cdot
B]=AB\mod h^\infty \Psi^{-\infty}$; we say that $A$ represents $\widetilde
A$.  For $M=\mathbb R^n$ and $U\subset T^*M$, there is a one-to-one
correspondence between local operators and their full symbols modulo
$h^\infty$; the symbols of local operators are functions $a(x,\xi;h)$
smooth in $(x,\xi)\in U$ all of whose derivatives are uniformly
bounded in $h$ on compact subsets of $U$.  In fact, for a symbol
$a(x,\xi;h)$, the corresponding local operator is defined by $[A\cdot
B]=(a\# b)^w(x,hD_x)$, where $b(x,\xi;h)$ is the full symbol of
$B\in\Psi^{\comp}$; since $b$ is compactly supported inside $U$ modulo
$O(h^\infty)$ and $a$ is defined on $U$, we can define the symbol
product $a\#b$ uniquely modulo $O(h^\infty)$. In particular, a
classical local operator $A\in\Psi^{\loc}_{\cl}(M)$ is uniquely
determined by the terms of the decomposition~\eqref{e:asymptotic-sum}
of its full symbol.  Note that $U$ is not required to be conic at
infinity, and we do not impose any conditions on the growth of $a$ as
$\xi\to\infty$.

Local operators form a sheaf of algebras; that is, one can multiply
local operators defined on the same set, restrict a local operator to
a smaller set, and reconstruct a local operator from its restrictions
to members of some finite open covering of $U$.  This makes it
possible to describe any local operator $A\in\Psi^{\loc}(U)$ on a
manifold using its full symbols in various coordinate charts.  For
$A\in\Psi^{\loc}(U)$, one can define its wavefront set $\WF(A)$ as
follows: $(x_0,\xi_0)\not\in\WF(A)$
if and only if the full symbol of $A$ is $O(h^\infty)$
in some neighborhood of $(x_0,\xi_0)$. If $A\in\Psi^k$
represents $\widetilde A\in\Psi^{\loc}$, then $\WF(A)=\WF(\widetilde A)$;
in general, wavefront sets of local operators obey
$\WF(A+B)\subset\WF(A)\cup\WF(B)$ and $\WF(AB)\subset\WF(A)\cap\WF(B)$.

Finally, we study microlocalization of arbitrary operators. Let $M_1$
and $M_2$ be two manifolds.  An $h$-dependent family of (properly
supported) operators $A(h):C^\infty(M_1)\to \mathcal D'(M_2)$ is
called tempered, or polynomially bounded, if for every compact
$K_1\subset M_1$, there exist $N$ and $C$ such that for any $u\in
C_0^\infty(K_1)$, $\|A(h)u\|_{H^{-N}_h(M_2)}\leq Ch^{-N}
\|u\|_{H^N_h}$.  Note that the composition of a tempered operator with
an element of $\Psi^k$ is still tempered. We can also treat distributions
on $M_2$ as operators from a singleton to $M_2$.

For a tempered family $A(h)$, we define its wavefront set
$\WF(A)\subset T^*(M_1\times M_2)$ as follows:
$(x,\xi;y,\eta)\not\in\WF(A)$, if and only if there exist
neighborhoods $U_1(x,\xi)$ and $U_2(y,\eta)$ such that for every
$B_j\in\Psi^{\comp}(M_j)$ with $\WF(B_j)\subset U_j$, we have
$B_2A(h)B_1\in h^\infty\Psi^{-\infty}$. We say that $A_1=A_2$
microlocally in some open and bounded $U\subset T^*M$, if
$\WF(A_1-A_2)\cap U=\emptyset$. Also, $A(h)$ is said to be compactly
microlocalized, if there exist $C_j\in\Psi^{\comp}(M_j)$ such that
$A(h)-C_2A(h)C_1\in h^\infty\Psi^{-\infty}$. In this case, all
operator norms $\|A(h)\|_{H^{N_1}_h\to H^{N_2}_h}$ are equivalent
modulo $O(h^\infty)$; if any of these norms is $O(h^r)$ for some
constant $r$, we write $\|A(h)\|=O(h^r)$. Here are some properties:
\begin{itemize}
\item If $A$ is compactly microlocalized, then $\WF(A)$ is compact.
The converse, however, need not be true.
\item If $A\in\Psi^k(M)$, then the two definitions
of compact microlocalization of $A$ (via its symbol and as a tempered
family of operators) agree; the wavefront set of $A$ as a tempered
family of operators is just $\{(x,\xi;x,\xi)\mid (x,\xi)\in\WF(A)\}$.
\item If $A_1,A_2$ are two tempered operators and at least one of them
is either compactly microlocalized or pseudodifferential, then
the product $A_2A_1$ is a tempered operator, and
$$
\begin{gathered}
\WF(A_2A_1)\subset\WF(A_1)\circ \WF(A_2)\\
=\{(x,\xi;z,\zeta)\mid \exists (y,\eta):(x,\xi;y,\eta)\in \WF(A_1),\
(y,\eta;z,\zeta)\in\WF(A_2)\}.
\end{gathered}
$$
Moreover, if both $A_1,A_2$ are compactly microlocalized, so is $A_2A_1$.
\end{itemize}
Let us quote the following microlocalization fact for
oscillatory integrals, which is the starting point for the
construction of semiclassical Fourier integral operators used in
Section~\ref{s:prelim-egorov}:
\begin{prop}\label{l:oi-wf}
Assume that $M$ is a manifold, $U\subset M\times \mathbb R^m$ is
open, $\varphi(x,\theta)$ is a smooth real-valued function on $U$,
with $x\in M$ and $\theta\in \mathbb R^m$, and $a(x,\theta)\in
C_0^\infty(U)$. Then the family of distributions
$$
u(x)=\int_{\mathbb R^m} e^{i\varphi(x,\theta)/h} a(x,\theta)\,d\theta
$$
is compactly microlocalized and
$$
\WF(u)\subset\{(x,\partial_x\varphi(x,\theta))\mid (x,\theta)\in\supp a,\
\partial_\theta\varphi(x,\theta)=0\}.
$$
\end{prop}

\subsection{Ellipticity and formal functional calculus}\label{s:prelim-functional}

Assume that $U\subset T^*M$ is open and $A\in \Psi^{\loc}_{\cl}(U)$.
We say that $A$ is (semiclassically) elliptic on $U$ if its principal
symbol does not vanish on $U$. Under this condition, there exists
unique operator $A^{-1}\in\Psi^{\loc}_{\cl}(U)$ such that
$A^{-1}A,AA^{-1}=I$ as local operators.  The next proposition provides
the form of the symbol of $A^{-1}$; it is based on the standard parametrix
construction:
\begin{prop}\label{l:elliptic-detail}
Fix a coordinate system on $M$. Assume that $A\in\Psi^{\loc}_{\cl}(U)$
is elliptic and has the full symbol $a\sim a_0+ha_1+h^2a_2+\dots$. Then
$A^{-1}$ has the full symbol $b\sim b_0+hb_1+h^2b_2+\dots$, where each
$b_j$ is a linear combination with constant coefficients of the terms
of the form
\begin{equation}\label{e:elliptic-form}
a_0^{-M-1}\prod_{m=1}^M \partial_x^{\alpha_m} \partial_\xi^{\beta_m} a_{l_m}.
\end{equation}
Here the (multi)indices $\alpha_m,\beta_m,l_m$ satisfy the condition
$$
\sum_{m=1}^M|\alpha_m|=\sum_{m=1}^M |\beta_m|=j-\sum_{m=1}^M l_m.
$$
Furthermore, we can assume that $|\alpha_m|+|\beta_m|+l_m>0$ for all $m$.
\end{prop}
\begin{proof}
We call~\eqref{e:elliptic-form} an expression of type
$(M_\alpha,M_\beta,L)$, where $M_\alpha,M_\beta,L$ are the sums of
$|\alpha_m|$, $|\beta_m|$, and $l_m$, respectively. If $f$ is an
expression of type $(M_\alpha,M_\beta,L)$, then we can prove by
induction that $\partial^\alpha_x \partial^\beta_\xi f$ is an
expression of type $(M_\alpha+|\alpha|,M_\beta+|\beta|,L)$. Now, we
write the equation $a\# b=1$; the principal term gives $b_0=a_0^{-1}$,
and the next terms give that each $b_j$ is the sum of expressions of
type $(j-L,j-L,L)$, by induction and the formula for the symbol
product $a\# b$.
\end{proof}
Let $A\in\Psi^k_{\cl}$; we say that it is elliptic on an open conic at
infinity $U\subset T^*M$ in the class $\Psi^k$ (or microlocally
elliptic), if its principal symbol $a_0$ satisfies $|a_0(x,\xi)|\geq
\langle\xi\rangle^k/C(K)$ for $(x,\xi)$ in any given closed conic at
infinity $K\subset U$ with compact projection onto $M$, and some
constant $C(K)$ depending on $K$. In this case, the full symbol of
$A^{-1}$ satisfies the decay conditions of the class $\Psi^{-k}$ in
$U$. In particular, if $A\in\Psi^k_{\cl}$ is elliptic in the class
$\Psi^k$ everywhere, then we can define $A^{-1}\in\Psi^{-k}_{\cl}$
for $h$ small enough.

We now construct functional calculus of local real principal
pseudodifferential operators.  For this, we use holomorphic functional
calculus~\cite[Section~7.3]{du-s}; another approach would be via
almost analytic continuation~\cite[Chapter~8]{d-s}.  First, assume
that $A\in\Psi^{\comp}_{\cl}(M)$ has compactly supported Schwartz
kernel.  In particular, the principal symbol $a_0$ of $A$ is compactly
supported; let $K\subset \mathbb C$ be the image of $a_0$. Let $f(z)$
be holomorphic in a neighborhood $\Omega$ of $K$, and let
$\gamma\subset\Omega$ be a contour such that $K$ lies inside of
$\gamma$. For each $h$, the operator $A$ is bounded $L^2\to L^2$; for
$h$ small enough, its spectrum lies inside of $\gamma$. Then we can
define the operator $f(A)$ by the formula
$$
f(A)={1\over 2\pi i}\oint_\gamma f(z)(z-A)^{-1}\,dz.
$$ 
For $z\in\gamma$, the operator $z-A$ is elliptic in the class
$\Psi^0$; therefore, $(z-A)^{-1}\in\Psi^0_{\cl}(M)$. It follows that
$f(A)\in\Psi^0_{\cl}(M)$. By Proposition~\ref{l:elliptic-detail}, the
full symbol of $f(A)$ (in any coordinate system) is the asymptotic sum
\begin{equation}\label{e:functional-calculus-symbol}
\sum_{j=0}^\infty h^j \sum_{M=0}^{2j} f^{(M)}(a_0(x,\xi))b_{jM}(x,\xi). 
\end{equation}
Here $a\sim a_0+ha_1+\dots$ is the full symbol of $A$; $b_{jM}$ are
the functions resulting from applying certain nonlinear differential
operators to $a_0,a_1,\dots$.

Now, assume that $U\subset T^*M$ is open and
$A\in\Psi^{\loc}_{\cl}(U)$ has real-valued principal symbol
$a_0$. Then the formula~\eqref{e:functional-calculus-symbol} can be
used to define an operator $f[A]\in\Psi^{\loc}_{\cl}(U)$ for any $f\in
C^\infty(\mathbb R)$.  Note that the principal symbol of $(z-A)^{-1}$
is $(z-a_0)^{-1}$; therefore, the principal symbol of $f[A]$ is
$f\circ a_0$.  The constructed operation posesses the following
properties of functional calculus:
\begin{prop}\label{l:formal-functional-calculus}
Assume that $U\subset T^*M$ is open, $A\in \Psi^{\loc}_{\cl}(U)$, and
$f,g\in C^\infty(\mathbb R)$. Then:

1. $\WF(f[A])\subset a_0^{-1}(\supp f)$, where $a_0$ is the principal
symbol of $A$.

2. If $f(t)=\sum_{j=0}^K f_j t^j$ is a polynomial, then $f[A]=f(A)$, where
$f(A)=\sum_j f_j A^j.$

3. $(f+g)[A]=f[A]+g[A]$ and $(fg)[A]=f[A]g[A]$.

4. If $B\in\Psi^{\loc}_{\cl}(U)$ and $[A,B]=0$, then
$[f[A],B]=0$.

The identities in parts~2---4 are equalities of local operators; in
particular, they include the $h^\infty\Psi^{-\infty}$ error. In fact,
the operator $f[A]$ is only defined uniquely modulo
$h^\infty\Psi^{-\infty}$.
\end{prop}
\begin{proof}
1. Follows immediately from~\eqref{e:functional-calculus-symbol}.

2. Take an open set $V$ compactly contained in $U$; then there exists
$\widetilde A\in\Psi^{\comp}_{\cl}(M)$ such that $A=\widetilde A$ microlocally
on $V$.  Since $f$ is entire and $\widetilde A$ is compactly
microlocalized, we can define $f(\widetilde A)$ by means of holomorphic
functional calculus; it is be a pseudodifferential operator
representing $f[\widetilde A]$. Now, $f(A)=f(\widetilde A)$ microlocally on
$V$ by properties of multiplication of pseudodifferential operators
and $f[A]=f[\widetilde A]$ microlocally on $V$
by~\eqref{e:functional-calculus-symbol}; therefore, $f(A)=f[A]$
microlocally on $V$. Since $V$ was arbitrary, we have $f(A)=f[A]$
microlocally on the whole $U$.

3. We only prove the second statement.  It suffices to show that for
every coordinate system on $M$, the full symbols of $(fg)[A]$ and
$f[A]g[A]$ are equal.  However, the terms in the asymptotic
decomposition of the full symbol of $(fg)[A]-f[A]g[A]$ at $(x,\xi)$
only depend on the derivatives of the full symbol of $A$ at $(x,\xi)$
and the derivatives of $f$ and $g$ at $a_0(x,\xi)$. Therefore, it
suffices to consider the case when $f$ and $g$ are polynomials. In
this case, we can use the previous part of the proposition and the
fact that $f(A)g(A)=(fg)(A)$.

4. This is proven similarly to the previous part, using the fact that
$[A,B]=0$ yields $[f(A),B]=0$ for every polynomial $f$.
\end{proof}
Finally, under certain conditions on the growth of $f$ and the symbol of $A$ at
infinity, $f[A]$ is a globally defined operator:
\begin{prop}\label{l:formal-functional-elliptic}
Assume that $A\in\Psi_{\cl}^k(M)$, with $k\geq 0$, and that $A$ is
elliptic in the class $\Psi^k$ outside of a compact subset of
$T^*M$. Also, assume that $f\in C^\infty(\mathbb R)$ is a symbol of
order $s$, in the sense that for each $l$, there exists a constant
$C_l$ such that
$$
|f^{(l)}(t)|\leq C_l\langle t\rangle^{s-l},\
t\in \mathbb R.
$$
Then $f[A]$ is represented by an operator in $\Psi^{sk}_{\cl}(M)$.
\end{prop}
\begin{proof}
We use~\eqref{e:functional-calculus-symbol}; by
Proposition~\ref{l:elliptic-detail}, the symbol $b_{jM}$ lies in
$S^{kM-j}$. Since $f$ is a symbol of order $s$, $f^{(M)}$ is a symbol
of order $s-M$. Then, since $a_0\in S^k$ is elliptic outside of a
compact set and $k\geq 0$, we have $f^{(M)}\circ a_0\in
S^{k(s-M)}$. It follows that each term
in~\eqref{e:functional-calculus-symbol} lies in $S^{sk-j}$; therefore,
this asymptotic sum gives an element of $\Psi^{sk}_{\cl}$.
\end{proof}

\subsection{Quantizing canonical transformations}\label{s:prelim-egorov}

Assume that $M_1$ and $M_2$ are two manifolds of the same
dimension. Recall that the symplectic form $\omega^S_j$ on $T^*M_j$ is
given by $\omega^S_j=d\sigma^S_j$, where $\sigma^S_j=\xi\,dx$ is the
canonical 1-form. We let $K_j\subset T^*M_j$ be compact and assume
that $\Phi:T^*M_1\to T^*M_2$ is a symplectomorphism defined in a
neighborhood of $K_1$ and such that $\Phi(K_1)=K_2$. Then the form
$\sigma^S_1-\Phi^*\sigma^S_2$ is closed; we say that $\Phi$ is an
exact symplectomorphism if this form is exact.  Define the classical
action over a closed curve in $T^*M_j$ as the integral of $\sigma^S_j$
over this curve; then $\Phi$ is exact if and only if for each
closed curve $\gamma$ in the domain of $\Phi$, the classical action
over $\gamma$ is equal to the classical action over $\Phi\circ\gamma$.
We can quantize exact symplectomorphisms as follows:
\begin{prop}\label{l:quantization-canonical}
Assume that $\Phi$ is an exact symplectomorphism.  Then there exist
$h$-dependent families of operators
$$
B_1: \mathcal D'(M_1)\to C_0^\infty(M_2),\
B_2: \mathcal D'(M_2)\to C_0^\infty(M_1)
$$
such that:

1. Each $B_j$ is compactly microlocalized and has operator norm
$O(1)$; moreover, $\WF(B_1)$ is contained in the graph of $\Phi$ and
$\WF(B_2)$ is contained in the graph of $\Phi^{-1}$.

2. The operators $B_1B_2$ and $B_2B_1$ are equal to the identity
microlocally near $K_2\times K_2$ and $K_1\times K_1$, respectively.

3. For each $P\in\Psi_{\cl}(M_1)$, there exists
$Q\in\Psi_{\cl}(M_2)$ that is intertwined with $P$ via $B_1$
and $B_2$:
$$
B_1P=QB_1,\
PB_2=B_2Q
$$
microlocally near $K_1\times K_2$ and $K_2\times K_1$,
respectively. Similarly, for each $Q\in\Psi_{\cl}(M_2)$ there exists
$P\in\Psi_{\cl}(M_1)$ intertwined with it. Finally, if $P$ and
$Q$ are intertwined via $B_1$ and $B_2$ and $p$ and $q$ are their
principal symbols, then $p=q\circ \Phi$ near $K_1$.

If the properties 1--3 hold, we say that the pair $(B_1,B_2)$
quantizes the canonical transformation $\Phi$ near $K_1\times K_2$.
\end{prop}
\begin{proof}
We take $B_1,B_2$ to be semiclassical Fourier integral operators
associated with $\Phi$ and $\Phi^{-1}$, respectively; their symbols
are taken compactly supported and elliptic in a neighborhood of
$K_1\times K_2$.  The existence of globally defined elliptic symbols
follows from the exactness of $\Phi$; the rest follows from calculus
of Fourier integral operators.  See~\cite[Chapter~8]{gus}
or~\cite[Chapter~2]{svn} for more details.
\end{proof}
Note that the operators $B_1$ and $B_2$ quantizing a given canonical
transformation are not unique. In fact, if
$X_j\in\Psi^{\comp}_{\cl}(M_j)$ are elliptic near $K_j$ and
$Y_j\in\Psi^{\comp}_{\cl}(M_j)$ are their inverses near $K_j$, then
$(X_2B_1X_1,Y_1B_2Y_2)$ also quantizes $\Phi$; moreover, $P$ is
intertwined with $Q$ via the new pair of operators if and only if
$X_1PY_1$ is intertwined with $Y_2QX_2$ via $(B_1,B_2)$.

We now study microlocal properties of Schr\"odinger propagators. Take
$A\in\Psi^{\comp}_{\cl}(M)$ with compactly supported Schwartz kernel
and let $a_0$ be its principal symbol; we assume that $a_0$ is
real-valued. In this case the Hamiltonian flow $\exp(tH_{a_0})$, $t\in
\mathbb R$, is a family of symplectomorphisms defined on the whole
$T^*M$; it is the identity outside of $\supp a_0$.  Moreover,
$\exp(tH_{a_0})$ is exact; indeed, if $V=H_{a_0}$, then by Cartan's
formula
$$
\mathcal L_V \sigma^S=d(i_V\sigma^S)+i_V d(\sigma^S)=
d(i_V\sigma^S-a_0) 
$$
is exact. Therefore,
$$
d_t \exp(tV)^* \sigma^S=\exp(tV)^* \mathcal L_V\sigma^S
$$
is exact and $\exp(tV)^*\sigma^S-\sigma^S$ is exact for all $t$.

For each $t$, define the operator $\exp(itA/h)$ as the solution to the
Schr\"odinger equation
$$
hD_t \exp(itA/h)=A\exp(itA/h)=\exp(itA/h)A
$$
in the algebra of bounded operators on $L^2(M)$, with the initial
condition $\exp(i0A/h)=I$. Such a family exists since $A$ is a bounded
operator on $L^2(M)$ for all $h$. Here are some of its properties
(see also~\cite[Chapter~10]{e-z}):
\begin{prop}\label{l:exp-fourior}
1. The operator $\exp(itA/h)-I$ is compactly microlocalized and has
operator norm $O(1)$.

2. If $A,B\in\Psi^{\comp}_{\cl}$ have real-valued principal symbols
and $[A,B]=O(h^\infty)$, then
$$
[\exp(itA/h),B]=O(h^\infty),\
\exp(it(A+B)/h)=\exp(itA/h)\exp(itB/h)+O(h^\infty).
$$
(We do not specify the functional spaces as the estimated families of
operators are compactly microlocalized, so all Sobolev norms are
equivalent.) In particular, if $B=O(h^\infty)$, then the propagators
of $A$ and $A+B$ are the same modulo $O(h^\infty)$.

3. Let $P\in\Psi_{\cl}$ and take
$$
P_t=\exp(itA/h)P\exp(-itA/h).
$$
Then $P_t$ is pseudodifferential and its full symbol depends smoothly
on $t$. The principal symbol of $P_t$ is $p_0\circ \exp(tH_{a_0})$,
where $p_0$ is the principal symbol of $P$; moreover,
$$
\WF(P_t)=\exp(-tH_{a_0})(\WF(P)).
$$

4. Let $K\subset T^*M$ be a compact set invariant under the
Hamiltonian flow of $a_0$. If $X\in\Psi^{\comp}_{\cl}$ is equal to the
identity microlocally near $K$, then the pair
$(X\exp(-itA/h),X\exp(itA/h))$ quantizes the canonical transformation
$\exp(tH_{a_0})$ near $K\times K$.  Moreover, if $P,Q\in\Psi^{\comp}_{\cl}$
are intertwined via these two operators, then
$Q=\exp(-itA/h)P\exp(itA/h)$ microlocally near $K$.

5. Assume that $V$ is a compactly supported vector field on $M$, and
let $\exp(tV):M\to M$ be the corresponding flow, defined for all $t$;
denote by $\exp(tV)^*$ the pull-back operator, acting on functions on
$M$. Let $K\subset T^*M$ be compact and invariant under the flow of
$V$, and $X\in\Psi^{\comp}_{\cl}$ have real-valued principal symbol
and be equal to the identity microlocally near $K$; consider
$(hV/i)X\in \Psi^{\comp}_{\cl}$.  Then for each $t$,
$$
\exp(it(hV/i)X/h)=\exp(tV)^*
$$
microlocally near $K\times K$.

The statements above are true locally uniformly in $t$.
\end{prop}
\begin{proof}
1. First, take $u\in L^2(M)$; then, since the principal symbol of $A$
is real-valued, we have $\|A-A^*\|_{L^2\to L^2}=O(h)$ and thus
$$
D_t \|\exp(itA/h) u\|_{L^2}^2
=h^{-1}((A-A^*)\exp(itA/h)u,\exp(itA/h)u)_{L^2}
=O(\|\exp(itA/h)u\|_{L^2}^2);
$$ 
therefore, $\exp(itA/h)$ is tempered:
$$
\|\exp(itA/h)\|_{L^2\to L^2}=O(e^{C|t|}).
$$
The rest follows from the identity
$$
\exp(itA/h)
=I+{it\over h}A+{i\over h}A\int_0^t (t-s)\exp(isA/ h)\,ds\cdot {i\over h}A.
$$

2. We have
$$
D_t(\exp(itA/h)B\exp(-itA/h))=h^{-1}\exp(itA/h)[A,B]\exp(-itA/h)=O(h^\infty);
$$
this proves the first identity. The second one is proved in a similar fashion:
$$
D_t(\exp(-it(A+B)/h)\exp(itA/h)\exp(itB/h))=O(h^\infty).
$$

3. We construct a family $\widetilde P_t$ of classical
pseudodifferential operators, each equal to $P$ microlocally
outside of a compact set, solving the initial-value problem
$$
D_t \widetilde P_t=h^{-1}[A,\widetilde P_t]+O(h^\infty),\
\widetilde P_0=P+O(h^\infty).
$$
For that, we can write a countable system of equations on the
components of the full symbol of $\widetilde P_t$. In particular, if
$p(t)$ is the principal symbol of $\widetilde P_t$, we get
$$
\partial_t p(t)=\{a_0,p(t)\}=H_{a_0}p(t);
$$
it follows that $p(t)=p_0\circ \exp(tH_{a_0})$. Similarly we can
recover the wavefront set of $\widetilde P_t$ from that of $P$.  Now,
$$
\partial_t (\exp(-itA/h)\widetilde P_t \exp(itA/h))=O(h^\infty);
$$
therefore, $P_t=\widetilde P_t+O(h^\infty)$.

4. Since $X$ is compactly microlocalized, so are the operators
$B_1=X\exp(-itA/h)$ and $B_2=X\exp(itA/h)$. Next, if $Y_2,Y_1\in
\Psi^{\comp}_{\cl}$, then
$$
Y_2B_1Y_1=Y_2X(\exp(-itA/h)Y_1 \exp(itA/h))\exp(-itA/h);
$$
using our knowledge of the wavefront set of the operator in brackets,
we see that this is $O(h^\infty)$ if
$$
\WF(Y_2)\cap \exp(tH_{a_0})\WF(Y_1)=\emptyset.
$$
Therefore, $\WF(B_1)$ is contained in the graph of $\exp(tH_{a_0})$;
similarly, $\WF(B_2)$ is contained in the graph of $\exp(-tH_{a_0})$.
Next,
$$
B_1B_2=X(\exp(-itA/h)X\exp(itA/h));
$$
however, the operator in brackets is the identity microlocally near
$K$, as $X$ is the identity microlocally near $K$ and $K$ is invariant
under $\exp(tH_{a_0})$. Therefore, $B_1B_2$ is the identity
microlocally near $K$.  The intertwining property is proved in a
similar fashion.

5. We have
$$
\partial_t(\exp(tVX)\exp(-tV)^*)=\exp(tVX)V(X-I)\exp(-tV)^*=O(h^\infty)
$$
microlocally near $K\times K$.
\end{proof}

Finally, we consider the special case $a_0=0$; in other words, we
study $\exp(itA)$, where $A\in\Psi^{\comp}_{\cl}$.  Since the
associated canonical transformation is the identity, it is not
unexpected that $\exp(itA)$ is a pseudodifferential operator:
\begin{prop}\label{l:exp-pseudor}
Let $a_1$ be the principal symbol of $A$. Then:

1. $\exp(itA)-I\in\Psi^{\comp}_{\cl}$ and the principal symbol of
$\exp(itA)$ equals $e^{ita_1}$. Moreover, if $A_1=A_2$ microlocally in
some open set, then $\exp(itA_1)=\exp(itA_2)$ microlocally in the same
set.

2. For any $P\in\Psi^{\comp}_{\cl}$, we have the following asymptotic
sum:
$$
\exp(itA)P\exp(-itA)\sim\sum_{j\geq 0} {(it\ad_A)^jP\over j!},
$$
where $\ad_A Q=[A,Q]$ for every $Q$.

3. If $U\subset T^*M$ is connected and $\exp(iA)=I$ microlocally in
$U$, then $A=2\pi l$ microlocally in $U$, where $l$ is an integer
constant.
\end{prop}
\begin{proof}
1. We can find a family of pseudodifferential operators $B_t$ solving
$$
\partial_t B_t=iAB_t+O(h^\infty),\
B_0=I,
$$
by subsequently finding each member of the asymptotic decomposition of
the full symbol of $B_t$.  Then
$$
\partial_t (\exp(-itA)B_t)=O(h^\infty);
$$
therefore, $\exp(itA)=B_t+O(h^\infty)$. The properties of $B_t$ can be
verified directly.

2. Follows directly from the equation
$$
\partial_t (\exp(itA)P\exp(-itA))=i\ad_A(\exp(itA)P\exp(-itA)).
$$

3. By calculating the principal symbol of $\exp(iA)$, we see that
$a_1$ has to be equal to $2\pi l$ in $U$ for some constant $l\in
\mathbb Z$. Subtracting this constant, we reduce to the case when
$A=O(h)$. However, if $A=O(h^N)$ for some $N\geq 1$, then
$\exp(iA)=I+iA+O(h^{N+1})$; by induction, we get $A=O(h^N)$
microlocally in $U$ for all $N$.
\end{proof}

\subsection{Integrable systems}\label{s:prelim-integrable}

Assume that $M$ is a two-dimensional manifold and $p_1,p_2$ are two
real-valued functions defined on an open set $U\subset T^*M$ such that:
\begin{itemize}
\item $\{p_1,p_2\}=0$;
\item for $\mathbf p=(p_1,p_2):U\to \mathbb R^2$
and each $\rho\in \mathbf p(U)$, the set
$\mathbf p^{-1}(\rho)$ is compact and connected.
\end{itemize}
We call such $\mathbf p$ an integrable system. Note that if $V\subset
\mathbb R^2$ is open and intersects $\mathbf p(U)$, and $F:V\to
\mathbb R^2$ is a diffeomorphism onto its image, then $F(\mathbf p)$
is an integrable system on $\mathbf p^{-1}(V)$.

We say that an integrable system $\mathbf p:U\to \mathbb R^2$ is
nondegenerate on $U$, if the differentials of $p_1$ and $p_2$ are
linearly independent everywhere on $U$. The following two propositions
describe the normal form for nondegenerate integrable systems:
\begin{prop}\label{l:arnold-liouville-1} Assume that the integrable system $\mathbf p$ is
nondegenerate on $U$. Then:

1. For each $\rho\in \mathbf p(U)$, the set $\mathbf
p^{-1}(\rho)\subset T^*M$ is a Lagrangian torus. Moreover, the family
of diffeomorphisms $$
\phi_t=\exp(t_1 H_{p_1}+t_2H_{p_2}),\ t=(t_1,t_2)\in \mathbb R^2,
$$
defines a transitive action of $\mathbb R^2$ on $\mathbf p^{-1}(\rho)$.
The kernel of this action is a rank two lattice depending
smoothly on $\rho$; we call it the periodicity lattice (at $\rho$).

2. For each $\rho_0\in\mathbf p(U)$, there exists a neighborhood $V(\rho_0)$
and a diffeomorphism $F:V\to \mathbb R^2$ onto its image such that
the nondegenerate integrable system $F(\mathbf p)$ has periodicity lattice
$2\pi \mathbb Z^2$ at every point. Moreover, if the Hamiltonian
flow of $p_2$ is periodic with minimal period $2\pi$, we can take
the second component of $F(\mathbf p)$ to be $p_2$.

3. Assume that $V\subset\mathbf p(U)$ is open and connected
and $F_1,F_2:V\to \mathbb R^2$ are two maps satisfying the conditions of part 2.
Then there exist $A\in\GL(2,\mathbb Z)$ and $b\in \mathbb R^2$ such that
$F_2=A\cdot F_1+b$.
\end{prop}
\begin{proof}
This is a version of Arnold--Liouville theorem; see~\cite[Section~1]{du} for the proof.
\end{proof}

\begin{prop}\label{l:arnold-liouville-2}
Assume that $\mathbf p:U\to \mathbb R^2$, $\mathbf p':U'\to \mathbb
R^2$, are nondegenerate integrable systems with periodicity lattices
$2\pi \mathbb Z^2$ at every point; here $U\subset T^*M$, $U'\subset
T^*M'$.  Take $\rho_0\in \mathbf p(U)\cap \mathbf p'(U')$. Then:

1. There exists a symplectomorphism $\Phi$ from a neighborhood of
$\mathbf p^{-1}(\rho_0)$ in $T^*M$ onto a neighborhood of $(\mathbf
p')^{-1}(\rho_0)$ in $T^*M'$ such that $\mathbf p=\mathbf
p'\circ\Phi$.

2. $\Phi$ is exact, as defined in Section~\ref{s:prelim-egorov}, if and only if
$$
\int_{\gamma_j}\sigma^S=\int_{\gamma'_j}\sigma^{\prime S},\ j=1,2,
$$
where $\gamma_j$ and $\gamma'_j$ are some fixed ($2\pi$-periodic)
Hamiltonian trajectories of $p_j$ on $\mathbf p^{-1}(\rho_0)$ and
$p'_j$ on $(\mathbf p')^{-1}(\rho_0)$, respectively.
\end{prop}
\begin{proof}
Part~1 again follows from Arnold--Liouville theorem.  For part~2, we
use that the closed 1-form $\sigma^S-\Phi^*\sigma^{\prime S}$ on a
tubular neighborhood of $\mathbf p^{-1}(\rho_0)$ is exact if and only
if its integral over each $\gamma_j$ is zero. Since $\gamma_j$ lie in
$\mathbf p^{-1}(\rho_0)$ and the restriction of $d\sigma^S=\omega^S$
to $\mathbf p^{-1}(\rho_0)$ is zero, we may shift $\gamma_j$ to make
both of them start at a fixed point $(x_0,\xi_0)\in \mathbf
p^{-1}(\rho_0)$. Similarly, we may assume that both $\gamma'_j$ start
at $\Phi(x_0,\xi_0)$. But in this case $\gamma'_j=\Phi\circ\gamma_j$
and
$$
\int_{\gamma_j}\sigma^S-\Phi^*\sigma^{\prime S}
=\int_{\gamma_j}\sigma^S-\int_{\gamma'_j}\sigma^{\prime S},
$$  
which finishes the proof.
\end{proof}
Next, we establish normal form for one-dimensional Hamiltonian systems
with one degenerate point. For that, consider $\mathbb R^2_{x,\xi}$
with the standard symplectic form $d\xi\wedge dx$, and define
$\zeta=(x^2+\xi^2)/2$; then $\zeta$ has unique critical point at zero
and its Hamiltonian flow is $2\pi$-periodic.
\begin{prop}\label{l:arnold-liouville-degenerate}
Assume that $p(x,\xi)$ is a real-valued function defined on an open subset of
$\mathbb R^2$ and for some $A\in \mathbb R$,
\begin{itemize}
\item the set $K_A=\{p\leq A\}$ is compact;
\item $p$ has exactly one critical point $(x_0,\xi_0)$
in $K_A$, $p(x_0,\xi_0)<A$, and the Hessian of $p$ at $(x_0,\xi_0)$
is positive definite.
\end{itemize}
Then there exists a smooth function $F$ on the segment
$[p(x_0,\xi_0),A]$, with $F'>0$ everywhere and $F(p(x_0,\xi_0))=0$,
and a symplectomorphism $\Psi$ from $K_A$ onto the disc $\{\zeta\leq
F(A)\}\subset \mathbb R^2$ such that $F(p)=\zeta\circ\Psi$. Moreover,
$F'(p(x_0,\xi_0))=(\det\nabla^2 p(x_0,\xi_0))^{-1/2}$.
If $p$ depends smoothly on some parameter $Z$, then $F$ and $\Psi$
can be chosen locally to depend smoothly on this parameter as well.
\end{prop}
\begin{proof} 
Without loss of generality, we may assume that $p(x_0,\xi_0)=0$.
Recall that in one dimension, symplectomorphisms are diffeomorphisms
that preserve both area and orientation.  By Morse lemma, there
exists an orientation preserving diffeomorphism $\Theta$ from a
neighborhood of $(x_0,\xi_0)$ onto a neighborhood of the origin such
that $p=\zeta\circ\Theta$.  Using the gradient flow of $p$, we can
extend $\Theta$ to a diffeomorphism from $K_A$ to the disc
$\{\zeta\leq A\}$ such that $p=\zeta\circ\Theta$. Let $J$ be the
Jacobian of $\Theta^{-1}$; then the integral of $J$ inside the disc
$\{\zeta\leq a\}$ is a smooth function of $a$.  Therefore, there
exists unique function $F$ smooth on $[0,A]$ such that $F'>0$ everywhere,
$F(0)=0$, and the integral of $J$ inside the disc $\{F(\zeta)\leq
a\}$, that is, the area of $\Theta^{-1}(\{F(\zeta)\leq a\})=\{F(p)\leq a\}\subset K_A$,
is equal to $2\pi a$.

Let $\widetilde\Theta$ be a diffeomorphism from $K_A$ onto $\{\zeta\leq
F(A)\}$ such that $F(p)=\zeta\circ\widetilde\Theta$ (constructed as
in the previous paragraph, taking $F(p)$ in place of $p$) and let $\tilde J$ be
the Jacobian of $\widetilde\Theta^{-1}$. We know that for $0\leq a\leq
F(A)$, the integral of $\tilde J-1$ over $\{\zeta\leq a\}$ is equal to
$0$.  Introduce polar coordinates $(r,\varphi)$; then there exists a
smooth function $\psi$ such that $\tilde J=1+\partial_\varphi \psi$
(see Proposition~\ref{l:moeser-lemma}). The transformation
$$
\widetilde\Psi:(r,\varphi)\mapsto (r,\varphi+\psi)
$$
is a diffeomorphism from $\{\zeta\leq F(A)\}$ to itself and has
Jacobian $\tilde J$; it remains to put
$\Psi=\widetilde\Psi\circ\widetilde\Theta$. To compute $F'(p_0(x_0,\xi_0))$,
we can compare the Hessians of $F(p)$ and $\zeta\circ\Phi$ at
$(x_0,\xi_0)$.

The function $F$ is uniquely determined by $p$ and thus will depend
smoothly on $Z$. As for $\Psi$, we first note that
$\widetilde\Theta$ was constructed using Morse lemma
and thus can be chosen locally to depend smoothly on $Z$
(see for example~\cite[Proof of Theorem~3.15]{e-z}).
Next, we can fix $\psi$ by requiring
that it integrates to zero over each circle centered at the origin
(see~Proposition~\ref{l:moeser-lemma}); then $\psi$, and thus $\widetilde\Psi$, will depend
smoothly on $Z$.  
\end{proof}

\section{Angular problem}\label{s:angular}

\subsection{Outline of the proof}\label{s:angular-outline}

Consider the semiclassical differential operators
(using the notation of~\eqref{e:semiclassical-parameters})
$$
\begin{gathered}
P_1(\tilde\omega,\tilde\nu;h)=h^2P_\theta(\omega)=
{1\over\sin\theta}(hD_\theta)(\Delta_\theta\sin\theta\cdot hD_\theta)\\
+{(1+\alpha)^2\over\Delta_\theta\sin^2\theta}
(a(\tilde\omega+ih\tilde\nu)\sin^2\theta-hD_\varphi)^2,\\
P_2(h)=hD_\varphi
\end{gathered}
$$
on the sphere $\mathbb S^2$.  Then $(\omega,\lambda,k)$ is a pole of
$R_\theta$ if and only if $(\tilde\lambda+ih\tilde\mu,\tilde k)$ lies
in the joint spectrum of the operators $(P_1,P_2)$ (see
Definition~\ref{d:joint-spectrum}). For $a=0$, $P_1$ is the
Laplace--Beltrami operator on the round sphere (multiplied by $-h^2$);
therefore, the joint spectrum of $(P_1,P_2)$ is given by the spherical
harmonics $(\tilde l(\tilde l+h),\tilde k)$, $\tilde k,\tilde l\in h
\mathbb Z$, $|\tilde k|\leq \tilde l$
(see for example~\cite[Section~8.4]{tay}). In the end of this subsection,
we give a short description of which parts of the angular problem are
simplified for $a=0$. For general small $a$, we will prove that the
joint spectrum is characterized by the following
\begin{prop}\label{l:angular-quantization-inner}
Let $\tilde\omega,\tilde\nu$ satisfy~\eqref{e:angular-regime}; we
suppress dependence of the operators and symbols on these parameters.
Consider
$$
\begin{gathered}
\widetilde K=\{(\tilde\lambda,\tilde k)\mid C_\theta^{-1}
\leq\tilde\lambda\leq C_\theta,\
\tilde\lambda\geq (1+\alpha)^2(\tilde k-a\tilde\omega)^2\}
\subset \mathbb R^2,\\
\widetilde K_\pm=\{(\tilde\lambda,\tilde k)\in \widetilde K\mid 
(1+\alpha)(\tilde k-a\tilde\omega)=\pm\sqrt{\tilde\lambda}\}.
\end{gathered}
$$
Then there exist functions $G_\pm(\tilde\lambda,\tilde k;h)$ such that:

1. $G_\pm$ is a complex valued classical symbol in $h$,
smooth in a fixed neighborhood of $\widetilde K$. For $(\tilde
\lambda,\tilde k)$ near $\widetilde K$ and $|\tilde\mu|\leq C_\theta$, we
can define $G_\pm(\tilde\lambda+ih\tilde\mu,\tilde k)$ by means of an
asymptotic (analytic) Taylor series for $G_\pm$ at $(\tilde\lambda,\tilde k)$.

2. For $a=0$, $G_\pm(\tilde\lambda,\tilde
k;h)=-h/2+\sqrt{\tilde\lambda+h^2/4}\mp\tilde k$.

3. $G_-(\tilde\lambda,\tilde k;h)-G_+(\tilde\lambda,\tilde
k;h)=2\tilde k$.

4. Let $F_\pm$ be the principal symbol of $G_\pm$. Then $F_\pm$ is
real-valued, $\partial_{\tilde\lambda} F_\pm>0$ and
$\mp\partial_{\tilde k} F_\pm>0$ on $\widetilde K$, and $F_\pm|_{\widetilde
K_\pm}=0$.

5. For $h$ small enough, the set of elements
$(\tilde\lambda+ih\tilde\mu,\tilde k)$ of the joint spectrum of
$(P_1,P_2)$ satisfying~\eqref{e:angular-regime} lies within $O(h)$ of
$\widetilde K$ and coincides modulo $O(h^\infty)$ with the set of
solutions to the quantization conditions
$$
\tilde k\in h \mathbb Z,\
G_\pm(\tilde\lambda+ih\tilde\mu,\tilde k)\in h \mathbb N;
$$
here $\mathbb N$ is the set of nonnegative integers.  Note that the
conditions $G_+\in \mathbb Z$ and $G_-\in \mathbb Z$ are equivalent;
however, we also require that both $G_+$ and $G_-$ be nonnegative.
Moreover, the corresponding joint eigenspaces are one-dimensional.
\end{prop}
\begin{figure}
\includegraphics{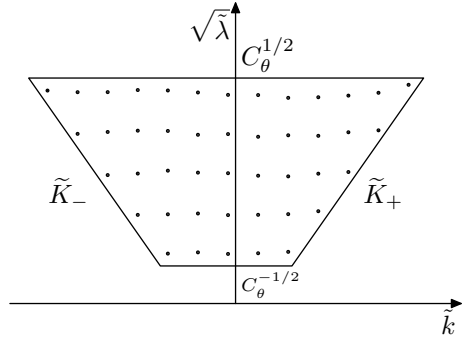}
\caption{The joint spectrum and the set $\widetilde K$.}
\end{figure}

Proposition~\ref{l:angular} follows from the proof of
Proposition~\ref{l:angular-quantization-inner}. In fact, the symbol
$\mathcal F^\theta(\tilde l,\tilde\omega,\tilde\nu,\tilde k;h)$ is
defined as the solution $\tilde\lambda+ih\tilde\mu$ to the equation
$$
G_+(\tilde\lambda+ih\tilde\mu,\tilde k,\tilde\omega+ih\tilde\nu;h)=\tilde l-\tilde k;
$$
this proves part (1) of Definition~\ref{d:asymptotics}.  The resolvent
estimates are an immediate corollary of the ones stated in
Proposition~\ref{l:angular-quantization-local} below. The
decomposition of $\mathcal F^\theta_0$ at $a=0$ follows from
Proposition~\ref{l:f-pm}.

We now give the schema of the proof of Proposition~\ref{l:angular-quantization-inner}.
Let $p_{j0}$ be the principal symbol of $P_j$; note that both
$p_{10}$ and $p_{20}$ are real-valued; also, define $\mathbf
p=(p_{10},p_{20}):T^* \mathbb S^2\to \mathbb R^2$.  In
Section~\ref{s:angular-hamiltonian}, we construct the principal parts
$F_\pm$ of the quantization symbols globally in $\widetilde K$, and
show that the intersection of the image of $\mathbf p$ with
$\{C_\theta^{-1}\leq\tilde\lambda\leq C_\theta\}$ is exactly
$\widetilde K$.  Using the theory of integrable systems described in
Section~\ref{s:prelim-integrable}, we then construct local
symplectomorphisms conjugating $(F_\pm(\mathbf p),p_{20})$ away from
$\widetilde K_\mp$ to the system $(\zeta,\eta)$ on $T^*\mathcal M$,
where $\mathcal M=\mathbb R_x\times \mathbb S^1_y$ is called the model
space, $(\xi,\eta)$ are the momenta corresponding to $(x,y)$, and
$$
\zeta={x^2+\xi^2\over 2}.
$$
Note that the integrable system $(\zeta,\eta)$ is nondegenerate on
$\{\zeta>0\}$ with periodicity lattice $2\pi \mathbb Z^2$, and
$d\zeta=0$ on $\{\zeta=0\}$.

Next, we take $(\tilde\lambda_0,\tilde k_0)\in \widetilde K$ and show
that joint eigenvalues in a certain $h$-independent neighborhood of
this point are given by a quantization condition. For this, we first
use Egorov's theorem and the symplectomorphisms constructed in
Section~\ref{s:angular-hamiltonian} to conjugate $P_1,P_2$
microlocally near $\mathbf p^{-1}(\tilde\lambda_0,\tilde k_0)$ to some
pseudodifferential operators $Q_1,Q_2$ on $\mathcal M$. The principal
symbols of $Q_j$ are real-valued functions of $(\zeta,\eta)$ only; in
Section~\ref{s:angular-moser}, we use Moser averaging to further
conjugate $Q_1,Q_2$ by elliptic pseudodifferential operators so that
the full symbols of $Q_j$ depend only on $(\zeta,\eta)$. In
Section~\ref{s:angular-grushin}, we use spectral theory to construct a
local Grushin problem for $(Q_1,Q_2)$, which we can conjugate back to
a local Grushin problem for $(P_1,P_2)$; then, we can apply the
results of Appendix~\ref{s:prelim-grushin} to obtain local
quantization conditions
(Proposition~\ref{l:angular-quantization-local}). To pass from these
local conditions to the global one, we use
\begin{prop}\label{l:angular-switch}
Assume that $G_j(\tilde\lambda,\tilde k;h)$ are two complex-valued
classical symbols in $h$ defined in some open set $U\subset \mathbb
R^2$, their principal symbols are both equal to some real-valued
$F(\tilde\lambda,\tilde k)$, with $\partial_{\tilde\lambda} F\neq 0$
everywhere and $\{F\geq 0\}$ convex, and solution sets to quantization
conditions
$$
\tilde k\in h\mathbb Z,\
G_j(\tilde\lambda+ih\tilde\mu,\tilde k)\in h \mathbb N
$$
in the region $(\tilde\lambda,\tilde k)\in U$, $\tilde\mu=O(1)$
coincide modulo $O(h^\infty)$. Then $G_1-G_2=hl+O(h^\infty)$ on
$\{F\geq 0\}$ for some constant $l\in \mathbb Z$.  Moreover, if
$\{F=0\}\cap U\neq\emptyset$, then $l=0$.
\end{prop}
\begin{proof}
Assume that $(\tilde\lambda_1,\tilde k_1)\in \{F\geq 0\}$. Then for
every $h$, there is a solution $(\tilde
\lambda(h)+ih\tilde\mu(h),\tilde k(h))$ to the quantization conditions
within $O(h)$ of $(\tilde\lambda_1,\tilde k_1)$; we know that
$$
G_j(\tilde\lambda(h)+ih\tilde\mu(h),\tilde k(h))\in h \mathbb Z+O(h^\infty),\ j=1,2,
$$
and thus $(G_1-G_2)(\tilde\lambda(h)+ih\tilde\mu(h),\tilde
k(h))=hl(h)+O(h^\infty)$, for some $l(h)\in \mathbb Z$. Since
$G_1-G_2=O(h)$ in particular in $C^1$, we have
$$
|(G_1-G_2)(\tilde \lambda_1,\tilde k_1)-hl(h)|=O(h^2).
$$
Therefore, $l(h)$ is constant for $h$ small enough and it is equal to
the difference of subprincipal symbols of $G_1$ and $G_2$ at
$(\tilde\lambda_1,\tilde k_1)$. It follows that $l(h)$ is independent
of $(\tilde\lambda_1,\tilde k_1)$; we can subtract it from one of the
symbols to reduce to the case when $G_1-G_2=O(h^2)$. The analysis
in the beginning of this proof then shows that
$$
\|G_1-G_2\|_{C(F\geq 0)}=O(h\|G_1-G_2\|_{C^1(F\geq 0)}+h^\infty).
$$
Arguing by induction, we get $G_1-G_2=O(h^N)$ for all $N$.  The last
statement follows directly by taking solutions to the quantization
conditions with $G_j=0$ and requiring that they satisfy the
quantization conditions $G_{3-j}\geq 0$.
\end{proof}
We can now cover $\widetilde K$ by a finite family of open sets, on
each of which there exists a local quantization condition. Using
Proposition~\ref{l:angular-switch} and starting from $\widetilde
K_\pm$, we can modify the local quantization conditions and piece them
together to get unique (modulo $h^\infty$) global $G_\pm$. The joint spectrum of
$(P_1,P_2)$ in a neighborhood of $\widetilde K$ is then given by the
global quantization condition; the joint spectrum outside of this
neighborhood, but satisfying~\eqref{e:angular-regime}, is empty by
part 2 of Proposition~\ref{l:angular-quantization-local}.

Also, the principal symbol of $G_--G_+$ is equal to $2\tilde k$;
therefore, $G_--G_+-2\tilde k$ is equal to $l\tilde h$ for some fixed
$l\in \mathbb Z$. However, $G_\pm$ depend smoothly on $a$ and thus it
is enough to prove that $l=0$ for $a=0$; in the latter case, the
symbols $G_\pm$ are computed explicitly from the spectrum of Laplacian
on the round sphere. (Without such a reference point, one would need
to analyse the subprincipal symbols of $G_\pm$ using the Maslov index.)
This finishes the proof of Proposition~\ref{l:angular-quantization-inner}.

Finally, let us outline the argument in the special case $a=0$
and indicate which parts of the construction are simplified. The
formulas below are not used in the general argument; we provide
them for the reader's convenience. The
principal symbol $p_{10}$ of $P_1$ is just the square of the
norm on $T^*\mathbb S^2$ generated by the round metric:
$$
p_{10}=\xi_\theta^2+{\xi_\varphi^2\over\sin^2\theta}.
$$
The set $\mathbf p^{-1}(\tilde\lambda,\tilde k)$ consists of all cotangent
vectors with length $\sqrt{\tilde\lambda}$ and momentum $\tilde k$;
therefore
\begin{enumerate}
\item for $\tilde\lambda\leq\tilde k^2$ (corresponding to the complement
of $\widetilde K$), the set $\mathbf p^{-1}(\tilde\lambda,\tilde k)$
is empty;
\item for $\tilde k=\pm\sqrt{\tilde\lambda}$ (corresponding to $\widetilde K_\pm$),
the set $\mathbf p^{-1}(\tilde\lambda,\tilde k)$ is a circle, consisting
of covectors tangent to the equator with length $\sqrt{\tilde\lambda}$ and direction
determined by the choice of sign;
\item for $\tilde \lambda > \tilde k^2$ (corresponding to the interior
of $\widetilde K$), the set $\mathbf p^{-1}(\tilde\lambda,\tilde k)$
is a Liouville torus.
\end{enumerate} 
The principal parts $F_\pm$ of the quantization symbols, constructed
in Proposition~\ref{l:f-pm}, can be computed explicitly:
$F_\pm=\sqrt{\tilde\lambda}\mp\tilde k$ (see the proof of part~2 of
this Proposition). Then $F_\pm^{-1}(\zeta,\eta)=(\zeta\pm\eta)^2$.
For $\pm \tilde k>0$, the canonical transformation~$\Phi_\pm$ from
Proposition~\ref{l:angular-conjugation} can be taken in the form
\begin{equation}\label{e:phi-pm-special}
\begin{gathered}
(\theta,\varphi,\xi_\theta,\xi_\varphi)\mapsto
(x,y,\xi,\eta)\\
=((2p_{10})^{1/2}(\sqrt {p_{10}}\pm\xi_\varphi)^{-1/2}\cos\theta,
\varphi+G,-2^{1/2}(\sqrt {p_{10}}\pm\xi_\varphi)^{-1/2}\sin\theta\xi_\theta
,\xi_\varphi);
\end{gathered}
\end{equation}
here $(x,y,\xi,\eta)$ are coordinates on $T^* \mathcal M$, with
$\mathcal M=\mathbb R_x\times \mathbb S^1_y$ the model space.
The function $G:T^* \mathbb S^2\to \mathbb S^1$ here is given by
$$
(\sqrt{p_{10}}\pm\xi_\varphi)\cos G=p_{10}^{1/2}\sin\theta\pm{\xi_\varphi\over\sin\theta},\
(\sqrt{p_{10}}\pm\xi_\varphi)\sin G=\mp\cos\theta\xi_\theta.
$$
In fact, the maps $\Phi_\pm$ defined in~\eqref{e:phi-pm-special}
extend smoothly to the poles $\{\sin\theta=0\}$ of the sphere and
satisfy the conditions of Proposition~\ref{l:angular-conjugation} on
the complement of the opposite equator $\{\theta=\pi/2,\
\xi_\theta=0,\ \xi_\varphi=\mp\sqrt{\tilde\lambda}\}$.

One can then conjugate the operators $P_1,P_2$ to some model operators
$Q_1,Q_2$ as in Proposition~\ref{l:moeser}.  To bring the subprincipal
terms in $Q_j$ to normal form, one still needs Moser averaging. Once the
normal form of Proposition~\ref{l:moeser} is obtained, it is possible
to use the ellipticity of $\mathbf p-(\tilde\lambda,\tilde k)$ away from
$\mathbf p^{-1}(\tilde\lambda,\tilde k)$ (as in Proposition~\ref{l:local-grushin})
and spectral theory to obtain the quantization condition. The Grushin problem
construction of Section~\ref{s:angular-grushin} and Appendix~\ref{s:prelim-grushin-1}
is not needed, as the operator $P_1$ is self-adjoint.

\subsection{Hamiltonian flow}\label{s:angular-hamiltonian}

Let $(\theta,\varphi)$ be the spherical coordinates on $\mathbb S^2$
and let $(\xi_\theta,\xi_\varphi)$ be the corresponding momenta.  Note
that $\xi_\theta$ is defined away from the poles $\{\sin\theta=0\}$,
while $\xi_\varphi$ is well-defined and smooth on the whole $T^*
\mathbb S^2$.  In the $(\theta,\varphi,\xi_\theta,\xi_\varphi)$
coordinates, the principal symbols of $P_2$ and $P_1$ are
$p_{20}=\xi_\varphi$ and
$$
p_{10}(\theta,\varphi,\xi_\theta,\xi_\varphi)=\Delta_\theta\xi_\theta^2
+{(1+\alpha)^2\over \Delta_\theta\sin^2\theta}(\xi_\varphi-a\tilde\omega\sin^2\theta)^2.
$$
Since $p_{10}$ does not depend on $\varphi$, we have
$$
\{p_{10},p_{20}\}=0.
$$
We would like to apply the results of
Section~\ref{s:prelim-integrable} on integrable Hamiltonian systems to
establish a normal form for $\mathbf p=(p_{10},p_{20})$.  First of
all, we study the points where the integrable system $\mathbf p$ is
degenerate:
\begin{prop}\label{l:hamiltonian-flow} For $a$ small enough,

1. For $C_\theta^{-1}\leq \tilde\lambda\leq C_\theta$, the set
$\mathbf p^{-1}(\tilde\lambda,\tilde k)$ is nonempty if and only if
$(\tilde\lambda,\tilde k)\in \widetilde K$.

2. The integrable system $\mathbf p$ is nondegenerate on $\mathbf
p^{-1}(\widetilde K)$, except at the equators
$$
E_\pm(\tilde\lambda)=\{\theta=\pi/2,\ \xi_\theta=0,\
(1+\alpha)(\xi_\varphi-a\tilde\omega)=\pm\sqrt{\tilde\lambda}\}
\subset T^* \mathbb S^2,\
C_\theta^{-1}\leq \tilde\lambda\leq C_\theta.
$$
Moreover, $p_{10}=\tilde\lambda$ on $E_\pm(\tilde\lambda)$ and the
union of all $E_\pm(\tilde\lambda)$ is equal to $\mathbf p^{-1}(\tilde
K_\pm)$.  Also,
\begin{equation}\label{e:angular-degeneracy}
dp_{10}=\pm 2(1+\alpha)\sqrt{\tilde\lambda}\,dp_{20}
\text{ on }E_\pm(\tilde\lambda).
\end{equation}
\end{prop}
\begin{proof}
We can verify directly the statements above for $a=0$, and
also~\eqref{e:angular-degeneracy} for all $a$. Then part~2 follows for
small $a$ by a perturbation argument; part~1 follows from part~2 by
studying the extremum problem for $\xi_\varphi$ restricted to
$\{p_{10}=\tilde\lambda\}$.
\end{proof}
Next, we construct the principal parts $F_\pm$ of the quantization symbols globally:
\begin{prop}\label{l:f-pm} For $a$ small enough,

1. There exist unique smooth real-valued functions
$F_\pm(\tilde\lambda,\tilde k)$ on $\widetilde K$ such that
$F_\pm|_{\widetilde K_\pm}=0$ and $(F_\pm(\mathbf p),p_{20})$ is a
nondegenerate completely integrable system on $\mathbf p^{-1}(\widetilde
K\setminus (\widetilde K_+\cup \widetilde K_-))$ with periodicity lattice
$2\pi \mathbb Z^2$.

2. $\partial_{\tilde\lambda} F_\pm>0$, $\mp \partial_{\tilde
k}F_\pm>0$, and $F_-(\tilde\lambda,\tilde k)-F_+(\tilde\lambda,\tilde
k)=2\tilde k$ on $\widetilde K$.  In particular, one can define the
inverse $F_\pm^{-1}(\zeta,\tilde k)$ of $F_\pm$ in the $\tilde\lambda$
variable, with $\tilde k$ as a parameter. Also,
$F_\pm=\sqrt{\tilde\lambda}\mp\tilde k$ for $a=0$ and
$\partial_{\tilde\lambda} F_\pm=\pm(2\tilde k)^{-1}+O(a^2)$ on
$\widetilde K_\pm$.

3. If $(\tilde\lambda,\tilde k)\in \widetilde K\setminus (\widetilde
K_+\cup\widetilde K_-)$ and $\gamma_\pm$ are some ($2\pi$-periodic)
trajectories of $F_\pm(\mathbf p)$ on $\mathbf
p^{-1}(\tilde\lambda,\tilde k)$, then
\begin{equation}\label{e:classical-action-F}
\int_{\gamma_\pm} \sigma^S=2\pi F_\pm(\tilde\lambda,\tilde k).
\end{equation}
\end{prop}
\begin{proof}
1. We first construct $F_+$ in a neighborhood of $\widetilde K_+$.  In
fact, we take small $\varepsilon_k>0$ and define $F_+$ on the set
$$
\widetilde K_\varepsilon=\{(\tilde\lambda,\tilde k)\mid \tilde k\geq \varepsilon_k,\
(1+\alpha)^2(\tilde k-a\tilde\omega)^2\leq\tilde\lambda\leq C_\theta\}.
$$
We will pick $\varepsilon_k$ small enough so that $\widetilde K_+\subset
\widetilde K_\varepsilon$; note, however, that $\widetilde K_\varepsilon$ does
not lie in $\widetilde K$.  Moreover, we will construct a
symplectomorphism $\Phi$ from $\mathbf p^{-1}(\widetilde K_\varepsilon)$
onto a subset of $T^* \mathcal M$ such that $F_+(\mathbf
p)\circ\Phi^{-1}=\zeta$ and $p_{20}\circ\Phi^{-1}=\eta$.

Note that $\xi_\varphi=0$ on the poles of the sphere
$\{\sin\theta=0\}$; therefore,
$(\theta,\varphi,\xi_\theta,\xi_\varphi)$ is a symplectic system of
coordinates near $\mathbf p^{-1}(\widetilde K_\varepsilon)$. Next, fix
$\xi_\varphi\geq \varepsilon_k$ and consider $p_{10}$ as a function of
$(\theta,\xi_\theta)$; then for $a$ small enough, this function has a
unique critical point $(0,0)$ on the compact set
$\{p_{10}(\cdot,\cdot,\xi_\varphi)\leq C_\theta\}$; the Hessian at
this point is positive definite. Indeed, it is enough to verify these
statements for $a=0$ and check that $\partial_\theta
p_{10}=\partial_{\xi_\theta}p_{10}=0$ for
$(\theta,\xi_\theta)=(\pi/2,0)$ and small $a$. Now, we may apply
Proposition~\ref{l:arnold-liouville-degenerate} to the function
$\{p_{10}(\cdot,\cdot,\xi_\varphi)\}$ and obtain a function
$F_+(\tilde\lambda;\tilde k)$ on $\widetilde K_\varepsilon$ such that
$F_+|_{\widetilde K_+}=0$ and $\partial_{\tilde\lambda}F_+>0$ and a
mapping
$$
\Psi:(\theta,\xi_\theta,\xi_\varphi)\mapsto
(\Psi_x(\theta,\xi_\theta,\xi_\varphi),
\Psi_\xi(\theta,\xi_\theta,\xi_\varphi))
$$
that defines a family of symplectomorphisms
$(\theta,\xi_\theta)\mapsto (\Psi_x,\Psi_\xi)$, depending smoothly on
the parameter $\xi_\varphi$, and
$$
F_+(p_{10}(\theta,\xi_\theta,\xi_\varphi),\xi_\varphi)
={1\over 2}(\Psi_x(\theta,\xi_\theta,\xi_\varphi)^2
+\Psi_\xi(\theta,\xi_\theta,\xi_\varphi)^2),\
(\theta,\xi_\theta,\xi_\varphi)\in\mathbf p^{-1}(\widetilde K_\varepsilon).
$$
Now, define $\Phi:(\theta,\varphi,\xi_\theta,\xi_\varphi)\mapsto
(\Phi_x,\Phi_y,\Phi_\xi,\Phi_\eta)\in T^* \mathcal M$ by
$$
\Phi_x=\Psi_x,\
\Phi_y=\varphi+G(\Psi_x,\Psi_\xi,\xi_\varphi),\
\Phi_\xi=\Psi_\xi,\
\Phi_\eta=\xi_\varphi.
$$
Here $G(x,\xi;\xi_\varphi)$ is some smooth function.  For $\Phi$ to be
a symplectomorphism, $G$ should satisfy
$$
\partial_\xi G(\Psi_x,\Psi_\xi,\xi_\varphi)
=\partial_{\xi_{\varphi}} \Psi_x,\
\partial_x G(\Psi_x,\Psi_\xi,\xi_\varphi)
=-\partial_{\xi_{\varphi}}\Psi_\xi.
$$
Since $(x,\xi)$ vary in a disc, this system has a solution
if and only if
$$
0=\{\Psi_\xi,\partial_{\xi_\varphi}\Psi_x\}+\{\partial_{\xi_\varphi}\Psi_\xi,\Psi_x\}
=\partial_{\xi_\varphi}\{\Psi_\xi,\Psi_x\};
$$
this is true since $\{\Psi_\xi,\Psi_x\}=1$. The defining properties of
$F_+$ now follow from the corresponding properties of the integrable
system $(\zeta,\eta)$; uniqueness follows from part~3 of
Proposition~\ref{l:arnold-liouville-1} and the condition $F_+|_{\widetilde
K_+}=0$.

Now, by part~2 of Proposition~\ref{l:arnold-liouville-1} and
Proposition~\ref{l:hamiltonian-flow}, for each
$(\tilde\lambda_0,\tilde k_0)\in \widetilde K\setminus (\widetilde K_+\cup
\widetilde K_-)$, there exists a smooth function $F(\tilde\lambda,\tilde
k)$ defined in a neighborhood of $(\tilde\lambda_0,\tilde k_0)$ such
that $\partial_\lambda F\neq 0$ and $(F(\mathbf p),p_{20})$ has
periodicity lattice $2\pi \mathbb Z^2$; moreover, part~3 of
Proposition~\ref{l:arnold-liouville-1} describes all possible $F$.
Then we can cover $\widetilde K\setminus \widetilde K_\varepsilon$ by a finite
set of the neighborhoods above and modify the resulting functions $F$
and piece them together, to uniquely extend the function $F_+$
constructed above from $\widetilde K_\varepsilon$ to $\widetilde
K\setminus\widetilde K_-$. (Here we use that $\widetilde K\setminus\widetilde K_-$
is simply connected.)
Similarly, we construct $F_-$ on $\widetilde
K\setminus\widetilde K_+$.  (The fact that $F_\pm$ is smooth at $\widetilde
K_\mp$ will follow from smoothness of $F_\mp$ at $\widetilde K_\mp$ and
the identity $F_--F_+=2\tilde k$.)

2. We can verify the formulas for $F_\pm$ for $a=0$ explicitly, using
the fact that the Hamiltonian flow of $\sqrt{\tilde\lambda}$ is
$2\pi$-periodic in this case.  The first two identities now follow
immediately. As for the third one, we know by part~3 of
Proposition~\ref{l:arnold-liouville-1} and the case $a=0$ that
$F_--F_+=2\tilde k+c$ for some constant $c$; we can then show that
$c=0$ using part 3 of this proposition. Finally,
$\partial_{\tilde\lambda} F_\pm|_{\widetilde K_\pm}$ can be computed using
Proposition~\ref{l:arnold-liouville-degenerate}.

3. First, assume that $(\tilde\lambda,\tilde k)\in \tilde
K_\varepsilon$ and let $\Phi$ be the symplectomorphism constructed in
part 1.  Then
$$
\Phi\circ\gamma_+=\{\zeta=F_+(\tilde\lambda,\tilde k),\ \eta=\tilde k,\ \varphi=\const\}
$$
is a circle. Let $D_+$ be the preimage under $\Phi$ of the disc with
boundary $\Phi\circ\gamma_+$; then
$$
\int_{\gamma_+} \sigma^S=\int_{D_+} \omega^S
=\int_{\Phi\circ D_+}\omega^S_{\mathcal M}=2\pi F_+(\tilde\lambda,\tilde k).
$$

We see that~\eqref{e:classical-action-F} holds for $F_+$ near $\widetilde
K_+$; similarly, it holds for $F_-$ near $\widetilde K_-$.  It now
suffices to show that for each $(\tilde\lambda_0,\tilde k_0)\in \tilde
K\setminus (\widetilde K_+\cup \widetilde K_-)$, there exists a neighborhood
$V(\tilde\lambda_0,\tilde k_0)$ such that if $(\tilde\lambda_j,\tilde
k_j)\in V$, $j=1,2$, and $\gamma^\pm_j$ are some ($2\pi$-periodic)
Hamiltonian trajectories of $F_\pm(\mathbf p)$ on $\mathbf
p^{-1}(\tilde\lambda_j,\tilde k_j)$, then
\begin{equation}\label{e:classical-action-d}
\int_{\gamma^\pm_2}\sigma^S-\int_{\gamma^\pm_1}\sigma^S
=2\pi(F_\pm(\tilde\lambda_2,\tilde k_2)-F_\pm(\tilde\lambda_1,\tilde k_1)).
\end{equation}
In particular, if~\eqref{e:classical-action-F} holds for one point of
$V$, it holds on the whole $V$. One way to
prove~\eqref{e:classical-action-d} is to use part~1 of
Proposition~\ref{l:arnold-liouville-2} to conjugate $(F_\pm(\mathbf
p),p_{20})$ to the system $(\xi_x,\xi_y)$ on the torus $\mathbb
T_{x,y}$ and note that the left-hand side
of~\eqref{e:classical-action-d} is the integral of the symplectic form
over a certain submanifold bounded by $\gamma_1,\gamma_2$; therefore,
it is the same for the conjugated system, where it can be computed
explicitly.
\end{proof}
Finally, we construct local symplectomorphisms conjugating
$(F_\pm(\mathbf p),p_{20})$ to $(\zeta,\eta)$:
\begin{prop}\label{l:angular-conjugation}
For each $(\tilde\lambda_0,\tilde k_0)\in \widetilde K\setminus \widetilde K_\mp$,
there exists an exact symplectomorphism $\Phi_\pm$ from a neighborhood
of $\mathbf p^{-1}(\tilde\lambda_0,\tilde k_0)$ in $T^*\mathbb S^2$
onto a neighborhood of
$$
\Lambda_{\mathcal M}=\{\zeta=F_\pm(\tilde\lambda_0,\tilde k_0),\
\eta=\tilde k_0\}
$$
in $T^* \mathcal M$ such that
$$
p_{10}\circ\Phi_\pm^{-1}=F_\pm^{-1}(\zeta,\eta),\
p_{20}\circ\Phi_\pm^{-1}=\eta.
$$
\end{prop}
\begin{proof}
The existence of $\Phi_\pm$ away from $\widetilde K_\pm$ follows from
part~1 of Proposition~\ref{l:arnold-liouville-2}, applied to the
systems $(F_\pm(\mathbf p),p_{20})$ and $(\zeta,\eta)$; near $\widetilde
K_\pm$, these symplectomorphisms have been constructed in the proof of
part 1 of Proposition~\ref{l:f-pm}.  Exactness follows by part 2 of
Proposition~\ref{l:arnold-liouville-2} (which still applies in the
degenerate case); the equality of classical actions over the flows of
$F_\pm(\mathbf p)$ and $\zeta$ follows from part~3 of
Proposition~\ref{l:f-pm}, while the classical actions over the flows
of both $p_{20}$ on $\mathbf p^{-1}(\tilde\lambda_0,\tilde k_0)$ and
$\eta$ on $\Lambda_{\mathcal M}$ are both equal to $2\pi\tilde k_0$.
\end{proof}

\subsection{Moser averaging}\label{s:angular-moser}

Fix $(\tilde\lambda_0,\tilde k_0)\in \widetilde K\setminus \widetilde K_\mp$,
take small $\varepsilon>0$, and define (suppressing the dependence on
the choice of the sign)
$$
\begin{gathered}
\Lambda^0=\mathbf p^{-1}(\tilde\lambda_0,\tilde k_0),\
\zeta_0=F_\pm(\tilde\lambda_0,\tilde k_0),\
\Lambda^0_{\mathcal M}=\{\zeta=\zeta_0,\ \eta=\tilde k_0\}
\subset T^*\mathcal M,\\
V^\varepsilon=\{(\tilde\lambda,\tilde k)\mid
|F_\pm(\tilde\lambda,\tilde k)-\zeta_0|\leq\varepsilon,\
|\tilde k-\tilde k_0|\leq\varepsilon\}\subset \mathbb R^2,\\
V^\varepsilon_{\mathcal M}=\{|\zeta-\zeta_0|\leq\varepsilon,\
|\eta-\tilde k_0|\leq\varepsilon\}\subset T^* \mathcal M;
\end{gathered}
$$
then $V^\varepsilon$ and $V^\varepsilon_{\mathcal M}$ are compact
neighborhoods of $(\tilde\lambda_0,\tilde k_0)$ and
$\Lambda^0_{\mathcal M}$, respectively. Here the functions $F_\pm$ are
as in Proposition~\ref{l:f-pm}.  Let $\Phi_\pm$ be the
symplectomorphism constructed in
Proposition~\ref{l:angular-conjugation}; we know that for
$\varepsilon$ small enough, $\Phi_\pm(\mathbf
p^{-1}(V^\varepsilon))=V^\varepsilon_{\mathcal M}$.  In this
subsection, we prove
\begin{prop}\label{l:moeser}
For $(\tilde\lambda_0,\tilde k_0)\in\widetilde K\setminus\widetilde
K_\pm$ and $\varepsilon>0$ small enough, there exists a pair of
operators $(B_1,B_2)$ quantizing $\Phi_\pm$ near $\mathbf
p^{-1}(V^\varepsilon)\times V^\varepsilon_{\mathcal M}$ in the sense
of Proposition~\ref{l:quantization-canonical} and operators
$Q_1,Q_2\in\Psi^{\comp}_{\cl}(\mathcal M)$ such that:

1. $P_1$ and $P_2$ are intertwined with $Q_1$ and $Q_2$, respectively,
via $(B_1,B_2)$, near $\mathbf p^{-1}(V^\varepsilon)\times
V^\varepsilon_{\mathcal M}$.  It follows immediately that the
principal symbols of $Q_1$ and $Q_2$ are $F_\pm^{-1}(\zeta,\eta)$ and
$\eta$, respectively, near $V^\varepsilon_{\mathcal M}$.

2. $Q_2=hD_y$ and the full symbol of $Q_1$ is a function of
$(\zeta,\eta)$, microlocally near $V^\varepsilon_{\mathcal M}$.  Here
we use Weyl quantization on $\mathcal M$, inherited from the covering
space $\mathbb R^2$.
\end{prop}
First of all, we use Proposition~\ref{l:quantization-canonical} to
find some $(B_1,B_2)$ quantizing $\Phi_\pm$ and $Q_1,Q_2$ intertwined
with $P_1$ and $P_2$ by $(B_1,B_2)$. Then we will find a couple of
operators $X,Y\in\Psi^{\comp}_{\cl}(\mathcal M)$ such that $Y=X^{-1}$
near $V^\varepsilon_{\mathcal M}$ and the operators
$Q'_1=XQ_1Y,Q'_2=XQ_2Y$ satisfy part 2 of
Proposition~\ref{l:moeser}. This is the content of this subsection and
will be done in several conjugations by pseudodifferential operators
using Moser averaging technique. We can then change $B_1,B_2$
following the remark after Proposition~\ref{l:quantization-canonical}
so that $P_1$ and $P_2$ are intertwined with $Q'_1$ and $Q'_2$, which
finishes the proof.

The averaging construction is based on the following
\begin{prop}\label{l:moeser-lemma}
Assume that the functions $p_0,f_0,g\in
C^\infty(V^\varepsilon_{\mathcal M})$ are given by one of the
following:
\begin{enumerate}
\item $p_0=f_0=\eta$ and $g$ is arbitrary;
\item $p_0=\zeta$ and $f_0=f_0(\zeta,\eta)$ is smooth in $V^\varepsilon_{\mathcal M}$,
with $\partial_\zeta f_0\neq 0$ everywhere, and $g$ is independent of
$y$.
\end{enumerate}
Define
$$
\langle g\rangle={1\over 2\pi}\int_0^{2\pi} g\circ \exp(tH_{p_0})\,dt.
$$
Then there exists unique $b\in C^\infty(V^\varepsilon_{\mathcal M})$ such that
$\langle b\rangle=0$ and
$$
g=\langle g\rangle+\{f_0,b\}.
$$
Moreover, in case (2) $b$ is independent of~$y$.
\end{prop}
\begin{proof}
We only consider case (2); case (1) is proven in a similar
fashion. First of all, if $b$ is $y$-independent, then
$\{f_0,b\}=\partial_\zeta f_0\cdot\{\zeta,b\} =\{\zeta,\partial_\zeta
f_0\cdot b\}$; therefore, without loss of generality we may assume
that $f_0=\zeta$.  The existence and uniqueness of $b$ now follows
immediately if we treat $y,\eta$ as parameters and consider polar
coordinates in the $(x,\xi)$ variables. To show that $b$ is smooth at
$\zeta=0$ (in case $\zeta_0\leq\varepsilon$), let $z=x+i\xi$ and
decompose $g-\langle g\rangle$ into an asymptotic sum of the terms
$z^j\bar z^k$ with $j,k\geq 0$, $j\neq k$, and coefficients smooth in
$(y,\eta)$; the term in $b$ corresponding to $z^j\bar z^k$ is $z^j\bar
z^k/(i(k-j))$.
\end{proof}
Henceforth in this subsection we will work with the operators $Q_j$ on
the level of their full symbols, microlocally in a neighborhood of
$V^\varepsilon_{\mathcal M}$.  (The operators $X$ and $Y$ will then be
given by the product of all operators used in conjugations below,
multiplied by an appropriate cutoff.)  Denote by $q_j$ the full symbol
of $Q_j$. We argue in three steps, following in
part~\cite[Section~3]{h-s}.

\noindent\textbf{Step 1:} Use Moser averaging to make $q_2$ independent of $y$.

Assume that $q_2$ is independent of $y$ modulo $O(h^{n+1})$ for some $n\geq 0$;
more precisely,
$$
q_2=\sum_{j=0}^n h^j q_{2,j}(x,\xi,\eta)+h^{n+1}r_n(x,y,\xi,\eta)+O(h^{n+2}).
$$
Take some $B\in\Psi^{\loc}_{\cl}$
with principal symbol $b$ and consider the conjugated operator
$$
Q'_2=\exp(ih^n B)Q_2\exp(-ih^n B).
$$
Here $\exp(\pm ih^n B)\in\Psi^{\loc}_{\cl}$ are well-defined by
Proposition~\ref{l:exp-pseudor} and inverse to each other; using the
same proposition, we see that the full symbol of $Q'_2$ is
$$
\sum_{j=0}^n h^j q_{2,j}(x,\xi,\eta)+h^{n+1}(r_n-\{\eta,b\})+O(h^{n+2}).
$$
If we choose $b$ as in Proposition~\ref{l:moeser-lemma}(1), then
$r_n-\{\eta,b\}=\langle r_n\rangle$ is a function of $(x,\xi,\eta)$
only; thus, the full symbol of $Q'_2$ is independent of $y$ modulo
$O(h^{n+2})$. Arguing by induction and taking the asymptotic product
of the resulting sequence of exponentials, we make the full symbol of
$Q_2$ independent of $y$.

\noindent\textbf{Step 2:} Use our knowledge of the spectrum of $P_2$ to make $q_2=\eta$.

First of all, we claim that
\begin{equation}\label{e:exp-1}
\exp(2\pi iQ_2/h)=I
\end{equation}
microlocally near $V^\varepsilon_{\mathcal M}\times
V^\varepsilon_{\mathcal M}$.  For that, we will use
Proposition~\ref{l:exp-fourior}. Let $X\in\Psi^{\comp}_{\cl}(\mathbb
S^2)$ have real-valued principal symbol, be microlocalized in a small
neighborhood of $\mathbf p^{-1}(V^\varepsilon)$, but equal to the
identity microlocally near this set.  Consider
$$
\mathcal P_t=\exp(itQ_2/h)B_1\exp(-it(P_2X)/h)B_2.
$$
We see that
$$
hD_t \mathcal P_t=\exp(itQ_2/h)(Q_2B_1-B_1P_2X)\exp(-it(P_2X)/h)B_2
$$
vanishes microlocally near $V^\varepsilon_{\mathcal M}\times
V^\varepsilon_{\mathcal M}$; integrating between~0 and~$2\pi$ and
using part~5 of Proposition~\ref{l:exp-fourior} to show that
$\exp(-2\pi i(P_2X)/h)=I$ microlocally near $\mathbf
p^{-1}(V^\varepsilon)
\times\mathbf p^{-1}(V^\varepsilon)$, we get~\eqref{e:exp-1}.

Now, let $X_{\mathcal M}\in\Psi^{\comp}_{\cl}(\mathcal M)$ be equal to
the identity microlocally near $\WF(Q_2)$; since the full symbol of
$Q_2$ is independent of $y$, we have $[Q_2,(hD_y)X_{\mathcal
M}]=O(h^\infty)$.  Therefore, by parts~2 and~5 of
Proposition~\ref{l:exp-fourior}
$$
\exp(2\pi i(Q_2-(hD_y)X_{\mathcal M})/h)=\exp(-2\pi iD_yX_{\mathcal M})\exp(2\pi iQ_2/h)=I
$$
microlocally near $V^\varepsilon_{\mathcal M}
\times V^\varepsilon_{\mathcal M}$. However,
$R=h^{-1}(Q_2-(hD_y)X_{\mathcal M})\in\Psi^{\loc}_{\cl}$ near $V^\varepsilon_{\mathcal M}$
and thus the left-hand side $\exp(2\pi i R)$ is pseudodifferential;
by part~3 of Proposition~\ref{l:exp-pseudor}, we get $R=l$ for some constant $l\in \mathbb Z$
and therefore
$$
Q_2=hD_y+hl
$$
microlocally near $V^\varepsilon_{\mathcal M}$.  It remains to
conjugate $Q_2$ by $e^{ily}$ to get $q_2=\eta$.

\noindent\textbf{Step 3:} Use Moser averaging again to make $q_1$ a function of $(\zeta,\eta)$, while
preserving $q_2=\eta$.

Recall that $[P_1,P_2]=0$; therefore, $[Q_1,Q_2]=0$ (microlocally near
$V^\varepsilon_{\mathcal M}$). Since $q_2=\eta$, this means that $q_1$
is independent of $y$. We now repeat the argument of Step 1, using
Proposition~\ref{l:moeser-lemma}(2) with $f_0=F_\pm^{-1}(\zeta,\eta)$.
The function $b$ at each step is independent of $y$; thus, we can take
$[B,hD_y]=O(h^\infty)$.  But in that case, conjugation by
$\exp(ih^nB)$ does not change $Q_2$; the symbol of the conjugated
$Q_1$ is still independent of $y$.  Finally, $\langle r_n\rangle$ is a
function of $(\zeta,\eta)$; therefore, $q_1$ after conjugation will
also be a function of $(\zeta,\eta)$.

\subsection{Construction of the Grushin problem}\label{s:angular-grushin}

In this subsection, we establish a local quantization condition:
\begin{prop}\label{l:angular-quantization-local} 
1. Assume that $(\tilde\lambda_0,\tilde k_0)\in\widetilde
K\setminus\widetilde K_\mp$ and $V_\varepsilon$ is the neighborhood of
$(\tilde\lambda_0,\tilde k_0)$ introduced in the beginning of
Section~\ref{s:angular-moser}. Then for $\varepsilon>0$ small
enough, there exists a classical symbol $\widetilde
G_\pm(\tilde\lambda,\tilde k;h)$ on $V_\varepsilon$ with principal
symbol $F_\pm$ and such that for $k\in \mathbb Z$, the poles
$\tilde\lambda+ih\tilde\mu$ of $R_\theta(\omega,\lambda,k)$ with
$(\tilde\lambda,\tilde k)\in V_\varepsilon$ and $|\tilde\mu|\leq
C_\theta$ are simple with polynomial resolvent estimate $L^2\to
L^2$, in the sense of Definition~\ref{d:asymptotics}, and coincide
modulo $O(h^\infty)$ with the solution set of the quantization
condition
\begin{equation}\label{e:local-qc}
\widetilde G_\pm(\tilde\lambda+ih\tilde\mu,\tilde k;h)\in h\mathbb N.
\end{equation}

2. Assume that $(\tilde\lambda_0,\tilde k_0)$
satisfies~\eqref{e:angular-regime}, but does not lie in $\widetilde K$.
Then there exists a neighborhood $V(\tilde\lambda_0,\tilde k_0)$ such
that for $h$ small enough, there are no elements
$(\tilde\lambda+ih\tilde\mu,\tilde k)$ of the joint spectrum of
$P_1,P_2$ with $(\tilde\lambda,\tilde k)\in V$ and $|\tilde\mu|\leq
C_\theta$, and $R_\theta(\omega,\lambda,k)$ is bounded $L^2\to L^2$
by $O(h^2)$.
\end{prop}
To prove part~1, we will use the microlocal conjugation constructed above. Let
$(\tilde\lambda_0,\tilde k_0)\in\widetilde K\setminus\widetilde K_\mp$
and $\varepsilon>0,B_1,B_2,Q_1,Q_2$ be given by Proposition~\ref{l:moeser}.
Consider the operators
$$
T_1={1\over 2}((hD_x)^2+x^2)-{h\over 2},\
T_2=hD_y
$$
on $\mathcal M$;
their full symbols are $\zeta-h/2$ and $\eta$, respectively. 
We know that $T_1$ and $T_2$ commute; the joint spectrum of
$T_1,T_2$ is $h(\mathbb N\times \mathbb Z)$.
Therefore, for any bounded function $f$ on $\mathbb R^2$,
we can define $f(T_1,T_2)$ by means of spectral theory; this is
a bounded operator on $L^2(\mathcal M)$.
\begin{prop}\label{l:spectral-microlocal}
1. For $f\in C_0^\infty(\mathbb R^2)$, the operator $f(T_1,T_2)$ is
pseudodifferential; moreover,
$f(T_1,T_2)\in\Psi^{\comp}_{\cl}(\mathcal M)$ and
$\WF(f(T_1,T_2))\subset\{(\zeta,\eta)\in\supp f\}$.  The full symbol
of $f(T_1,T_2)$ in the Weyl quantization is a function of $\zeta$ and
$\eta$ only; the principal symbol is $f(\zeta,\eta)$.

2. Assume that $\zeta_1\in h\mathbb N$, $\eta_1\in h \mathbb Z$.
Let $u$ be the $L^2$ normalized joint eigenfunction of $(T_1,T_2)$ with
eigenvalue $(\zeta_1,\eta_1)$. Then $u$ is compactly microlocalized
and
$$
\WF(u)\subset \{\zeta=\zeta_1,\ \eta=\eta_1\}.
$$

3. Assume that the function $f(\zeta,\eta;h)$ is Borel measurable,
has support contained in
a compact $h$-independent subset $K_f$ of $\mathbb R^2$, and
$$
\max\{|f(\zeta,\eta;h)|\mid \zeta\in h\mathbb N,\ \eta\in h \mathbb Z\}
\leq Ch^{-r}
$$
for some $r\geq 0$. 
Then the operator $f(T_1,T_2;h)$ is compactly microlocalized, its
wavefront set is contained in the square of $\{(\zeta,\eta)\in K_f\}$,
and the operator norm of $f(T_1,T_2;h)$ is $O(h^{-r})$.
\end{prop}
\begin{proof}
For part~1, we can show that the operator $f(T_1,T_2)$ is
pseudodifferential by means of Helffer--Sj\"ostrand formula in
calculus of several commuting pseudodifferential operators; see for
example~\cite[Chapter~8]{d-s}. This also gives information on the
principal symbol and the wavefront set of this operator.  To show
that the full symbol of $f(T_1,T_2)$ depends only on $(\zeta,\eta)$,
note that if $A\in\Psi^{\loc}(\mathcal M)$ and $a$ is its full
symbol in the Weyl quantization, then the full symbol of $[A,T_1]$
in the Weyl quantization is $-ih\{a,\zeta\}$; similarly, the full
symbol of $[A,T_2]$ in the Weyl quantization is $-ih\{a,\eta\}$ (see
for example~\cite[discussion before~(1.11)]{sj}). Since
$[f(T_1,T_2),T_j]=0$, the full symbol of $f(T_1,T_2)$ Poisson
commutes with $\zeta$ and $\eta$.

To show part~2, we take $\chi(\zeta,\eta)\in C_0^\infty(\mathbb R^2)$
equal to 1 near $(\zeta_1,\eta_1)$; then $u=\chi(T_1,T_2)u$.
Similarly, to show part~3, we take $\chi$ equal to 1 near $K_f$; then
the $L^2$ operator norm of $f(T_1,T_2)$ can be estimated easily and
$f(T_1,T_2)=\chi(T_1,T_2)f(T_1,T_2)\chi(T_1,T_2)$.
\end{proof}
Now, recall that by Proposition~\ref{l:moeser},
the full symbol of $Q_1$ in the Weyl quantization is
a function of $(\zeta,\eta)$ near $V^{\varepsilon}_{\mathcal M}$;
therefore, we can find a compactly supported symbol $\widetilde
G_\pm(\tilde\lambda,\tilde k;h)$ such that the principal symbol of
$\widetilde G_\pm$ near $V^\varepsilon$ is $F_\pm$ and
$$
Q_1=\widetilde G_\pm^{-1}(T_1,T_2;h)
$$
microlocally near $V^\varepsilon_{\mathcal M}$, where
$\widetilde G_\pm^{-1}(\zeta,\tilde k;h)$ is the inverse of $\widetilde G_\pm$ in the
$\tilde\lambda$ variable. Recall also that $Q_2=T_2$ microlocally near
$V^\varepsilon_{\mathcal M}$. Multiplying $Q_1,Q_2$ by an appropriate
cutoff, which is a function of $T_1,T_2$, we can assume that $Q_1,Q_2$
are functions of $T_1,T_2$ modulo $h^\infty\Psi^{-\infty}$.  We can
now construct a local Grushin problem for $Q_1,Q_2$:
\begin{prop}\label{l:grushin-q} Let $(\tilde\lambda_1,\tilde k_1)\in V^\varepsilon$
and $|\tilde\mu_1|\leq C_\theta$.

1. Assume that $(\tilde\lambda_1+ih\tilde\mu_1,\tilde k_1)$
satisfies~\eqref{e:local-qc}, with $\zeta_1=\widetilde G_\pm
(\tilde\lambda_1+ih\tilde\mu_1,\tilde k_1)\in h \mathbb N$.  Then
there exist operators $A_1,A_2,S_1,S_2$ such that
conditions~(L1)--(L5) of Appendix~\ref{s:prelim-grushin-2} are
satisfied, with $r=1$, $(P_1,P_2)$ replaced by
$(Q_1-\tilde\lambda_1-ih\tilde\mu_1,Q_2-\tilde k_1)$,
$K=\{\zeta=\zeta_1,\ \eta=\tilde k_1\}$, and
\begin{equation}\label{e:lg-identity-1}
A_1(Q_1-\tilde\lambda_1-ih\tilde\mu_1)+A_2(Q_2-\tilde k_1)=I-S_1S_2
\end{equation}
microlocally near $V^\varepsilon_{\mathcal M}\times
V^\varepsilon_{\mathcal M}$.

2. Fix $\delta>0$ and assume that
$$
|(\tilde\lambda_1+ih\tilde\mu_1,\tilde k_1)-(G_\pm^{-1}(\zeta,\eta),\eta)|\geq\delta h,\
\zeta\in h \mathbb N,\
\eta\in h \mathbb Z.
$$
Then there exist operators $A_1,A_2$ such that the conditions
(L1)--(L2) of Appendix~\ref{s:prelim-grushin-2} are satisfied, with
$r=1$, $(P_1,P_2)$ replaced by
$(Q_1-\tilde\lambda_1-ih\tilde\mu_1,Q_2-\tilde k_1)$,
$K=\{\zeta=\zeta_1,\ \eta=\tilde k_1\}$, and
$$
A_1(Q_1-\tilde\lambda_1-ih\tilde\mu_1)+A_2(Q_2-\tilde k_1)=I
$$
microlocally near $V^\varepsilon_{\mathcal M}\times V^\varepsilon_{\mathcal M}$.
\end{prop}
\begin{proof}
1. Let $S_1:\mathbb C\to L^2(\mathcal M)$ and
$S_2:L^2(\mathcal M)\to \mathbb C$ be the inclusion and
the orthogonal projection onto, respectively, the unit
joint eigenfunction of $(T_1,T_2)$ with eigenvalue $(\zeta_1,\tilde k_1)$.
The properties (L3) and (L4) now follow from part 2 of
Proposition~\ref{l:spectral-microlocal}.

Next, we use a partition of unity on the circle to construct
the functions $\chi_1,\chi_2$ with the following properties:
\begin{itemize}
\item $\chi_j\in C^\infty(\mathbb R^2\setminus 0)$ is positively homogeneous
of degree 0;
\item $\chi_j\geq 0$ and $\chi_1+\chi_2=1$ everywhere on $\mathbb R^2\setminus 0$;
\item $\chi_j(s_1,s_2)=0$ for $|s_j|<|s_{3-j}|/2$.
\end{itemize}
It follows
that
\begin{equation}\label{e:chi-estimate}
|s_j^{-1}\chi_j(|s_1|^2,|s_2|^2)|\leq C(|s_1|+|s_2|)^{-1}.
\end{equation}
Take $\chi(\zeta,\eta)\in C_0^\infty$
supported in a small neighborhood of $V^\varepsilon_{\mathcal M}$,
while equal to 1 near $V^\varepsilon_{\mathcal M}$;
define the functions $f_1,f_2$ as follows:
$$
\begin{gathered}
f_1(\zeta,\eta;h)={\chi(\zeta,\eta)\chi_1(|\widetilde G_\pm^{-1}(\zeta,\eta;h)-\tilde\lambda_1-ih\tilde\mu_1|^2,
|\eta-\tilde k_1|^2)\over \widetilde G_\pm^{-1}(\zeta,\eta;h)-\tilde\lambda_1-ih\tilde\mu_1},\\
f_2(\zeta,\eta;h)={\chi(\zeta,\eta)\chi_2(|\widetilde G_\pm^{-1}(\zeta,\eta;h)-\tilde\lambda_1-ih\tilde\mu_1|^2,
|\eta-\tilde k_1|^2)\over \eta-\tilde k_1},
\end{gathered}
$$
for $(\zeta,\eta)\neq (\zeta_1,\tilde k_1)$; we put $f_j(\zeta_1,\tilde k_1)=0$.
We now take $A_j=f_j(T_1,T_2;h)$.
Noticing that
$$
|\widetilde G_\pm^{-1}(\zeta,\eta;h)-\tilde\lambda_1|+|\eta-\tilde k_1|\geq h/C,\
(\zeta,\eta)\in h(\mathbb N\times \mathbb Z)\cap\supp\chi\setminus (\zeta_1,\tilde k_1),
$$
and using Proposition~\ref{l:spectral-microlocal}
and~\eqref{e:chi-estimate}, we get that $A_j$ are compactly
microlocalized and $\|A_j\|=O(h^{-1})$. Moreover, if
$\tilde\chi(\zeta,\eta)$ is equal to $1$ near $(\zeta_1,\tilde k_1)$,
then $(1-\tilde\chi)f_j$ are smooth symbols; then,
$A''_j=(1-\tilde\chi)(T_1,T_2)A_j$ belongs to $\Psi^{\comp}_{\cl}$ by
part~1 of Proposition~\ref{l:spectral-microlocal} and
$A'_j=\tilde\chi(T_1,T_2)A_j$ is microlocalized in the Cartesian square of
$\{(\zeta,\eta)\in\supp\tilde\chi\}$; we have established
property~(L1), with $r=1$.  The properties~(L2), (L5),
and~\eqref{e:lg-identity-1} are easy to verify, given that all the
operators of interest are functions of $T_1,T_2$.

2. This is proved similarly to part 1.
\end{proof}
Finally, we conjugate the operators of the previous proposition by
$B_1,B_2$ to get a local Grushin problem for $P_1,P_2$ and obtain
information about the joint spectrum:
\begin{proof}[Proof of Propositon~\ref{l:angular-quantization-local}]
1. Assume first that $\tilde\lambda_1,\tilde k_1,\tilde\mu_1$ satisfy
the conditions of part~1 of Proposition~\ref{l:grushin-q}; let
$A_1,A_2,S_1,S_2$ be the operators constructed there. Recall that
$A_j$ are microlocalized in a small neighborhood of
$V^\varepsilon_{\mathcal M}$.  Then the operators
$$
\widetilde A_j=B_2A_jB_1,\
\widetilde S_1=B_2S_1,\
\widetilde S_2=S_2B_1,
$$
together with $P_1-\tilde\lambda_1-ih\tilde\mu_1,P_2-\tilde k_1$ in
place of $P_1,P_2$ satisfy the conditions~(L1)--(L5) of
Appendix~\ref{s:prelim-grushin-2} with $K=\mathbf
p^{-1}(\tilde\lambda_1,\tilde k_1)$ and
$$
\widetilde A_1(P_1-\tilde\lambda_1-ih\tilde\mu_1)+\widetilde A_2(P_2-\tilde k_1)=I-\widetilde S_1\widetilde S_2
$$
microlocally near $\mathbf p^{-1}(V^\varepsilon)$. Moreover,
$P_1-\tilde\lambda_1-ih\tilde\mu_1,P_2-\tilde k_1$ satisfy
conditions~(E1)--(E2) of Appendix~\ref{s:prelim-grushin-2} and the set
where both their principal symbols vanish is exactly $K$.  We can now
apply part~2 of Proposition~\ref{l:local-grushin} to show that for $h$
small enough and some $\delta>0$, independent of
$h,\tilde\lambda_1,\tilde k_1$, there is exactly one element of the
joint spectrum of $(P_1,P_2)$ in the ball of radius $\delta h$
centered at $(\tilde\lambda_1+ih\tilde\mu_1,\tilde k_1)$, and this
point is within $O(h^\infty)$ of the center of the ball.

Now, we assume that $(\tilde\lambda+ih\tilde\mu,\tilde k)$ satisfies
the conditions of part~2 of Proposition~\ref{l:grushin-q}, with
$\delta$ specified in the previous paragraph. Then we can argue as
above, using part~1 of Proposition~\ref{l:local-grushin}, to show that
this point does not lie in the joint spectrum for $h$ small enough.

Since every point $(\tilde\lambda+ih\tilde\mu,\tilde k)$ such that
$(\tilde\lambda,\tilde k)\in V^\varepsilon$ and $|\tilde\mu|\leq
C_\theta$ is covered by one of the two cases above, we have
established that the angular poles in the indicated region coincide
modulo $O(h^\infty)$ with the set of solutions to the quantization
condition. Moreover, Proposition~\ref{l:global-grushin-special}
together with the construction of a global Grushin problem from a
local one carried out in the proof of
Proposition~\ref{l:local-grushin} provides the resolvent estimates
required in Definition~\ref{d:asymptotics}.

2. The set $\mathbf p^{-1}(\tilde\lambda_0,\tilde k_0)$ is empty by
Proposition~\ref{l:hamiltonian-flow}; therefore, the operator
$$
T=(P_1-\tilde\lambda-ih\tilde\mu)^*(P_1-\tilde \lambda-ih\tilde\mu)+(P_2-\tilde k)^2
$$
is elliptic in the class $\Psi^2(\mathbb S^2)$ for
$(\tilde\lambda,\tilde k)$ close to $(\tilde\lambda_0,\tilde k_0)$ and
$\tilde \mu$ bounded; therefore, for $h$ small enough,
$\|T^{-1}\|_{L^2\to L^2}=O(1)$. The absense of joint spectrum and
resolvent estimate follow immediately if we notice that
the restriction of $T$ to $\mathcal D'_k$ is $h^4
(P_\theta-\lambda)^*(P_\theta-\lambda)$.
\end{proof}

\section{Radial problem}\label{s:radial}

\subsection{Trapping}\label{s:radial-trapping}

In~\cite[Section~4]{skds}, we use a Regge--Wheeler change of variables
$r\to x$, under which and after an appropriate rescaling the radial
operator becomes (using the notation of~\eqref{e:semiclassical-parameters})
$$
\begin{gathered}
P_x(h)=h^2D_x^2+V(x,\tilde\omega,\tilde\nu,\tilde\lambda,\tilde\mu,\tilde k;h),\\
V(x;h)=(\tilde\lambda+ih\tilde\mu)\Delta_r
-(1+\alpha)^2((r^2+a^2)(\tilde\omega+ih\tilde\nu)-a\tilde k)^2
\end{gathered}
$$
(note the difference in notation with~\cite[Section~7]{skds}).  Let
$V(x;h)=V_0(x)+hV_1(x)+h^2V_2(x)$, where
$$
V_0(x)=\tilde\lambda\Delta_r-(1+\alpha)^2((r^2+a^2)\tilde\omega-a\tilde k)^2
$$
is the semiclassical principal part of $V(x)$; note that $V_0$ is
real-valued and for $1\leq\tilde\omega\leq 2$ and $a$ small enough,
$V_0(\pm\infty)<0$. Now, \cite[Proposition~7.4]{skds} establishes an
arbitrarily large strip free of radial poles in the nontrapping cases;
therefore, the only radial poles in the region~\eqref{e:radial-regime}
appear in case (3) of~\cite[Proposition~7.3]{skds}. Using the proof
of the latter proposition, we may assume that:
\begin{itemize}
\item $|\tilde\lambda-\tilde\lambda_0(\tilde\omega,\tilde k)|<\varepsilon_r$, where
$\tilde\lambda_0^{-1}$ is the value of the function
$$
F_V(r;\tilde\omega,\tilde k)={\Delta_r\over (1+\alpha)^2((r^2+a^2)\tilde\omega-a\tilde k)^2} 
$$
at its only maximum point. Under the assumptions~\eqref{e:radial-regime},
$1/C\leq\tilde\lambda_0\leq C$ for some constant $C$;
\item $V_0$, as a function of $x$, has unique global maximum $x_0$,
$|V_0(x_0)|<\varepsilon_r^3$ and $V''_0(x)<0$ for $|x-x_0|\leq\varepsilon_r$;
\item $V_0(x)<-\varepsilon_r^3$ for $|x-x_0|\geq\varepsilon_r$.
\end{itemize}
Here $\varepsilon_r>0$ is a small constant we will choose later.
We can also compute
\begin{equation}\label{e:pot-0}
\begin{gathered}
\tilde\lambda_0={27M_0^2\tilde\omega^2\over 1-9\Lambda M_0^2}
\text{ for }a=0;\\
V''_0(x_0)=-18M_0^4(1-9\Lambda M_0^2)^2\tilde\lambda
\text{ for }a=0,\ \tilde\lambda=\tilde\lambda_0.
\end{gathered}
\end{equation}
Letting
$$
p_0(x,\xi)=\xi^2+V_0(x)
$$
be the principal symbol of $P_x$, we see that
$p_0$ has a nondegenerate hyperbolic critical point
at $(x_0,0)$ and this is the only critical point
in the set $\{p_0\geq -\varepsilon_r^3\}$.
\begin{figure}
\includegraphics{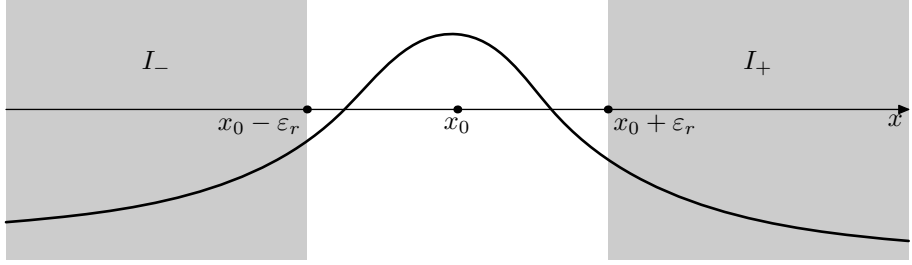}
\caption{The potential $V_0$ and the intervals $I_\pm$.}
\end{figure}

\subsection{WKB solutions and the outgoing condition}\label{s:radial-wkb}

Firstly, we obtain certain approximate solutions to the equation
$P_xu=0$ in the region $|x-x_0|>\varepsilon_r$, where $V_0$ is known
to be negative. (Compare with~\cite[Sections~2 and~3]{ra}.)  Define the
intervals
\begin{equation}\label{e:intervals}
I_+=(x_0+\varepsilon_r,+\infty),\
I_-=(-\infty,x_0-\varepsilon_r),\
I_0=(x_0-2\varepsilon_r,x_0+2\varepsilon_r).
\end{equation}
Let $\psi_0(x)$ be a smooth function on $I_+\cup I_-$ solving the
eikonal equation
$$
\psi'_0(x)=\sgn(x-x_0)\sqrt{-V_0(x)}.
$$
(We will specify a normalization condition for $\psi_0$ later.)  Then
we can construct approximate WKB solutions
\begin{equation}\label{e:radial-wkb-form}
u_\pm^+(x;h)=e^{i\psi_0(x)/h}a_\pm^+(x;h),\
u_\pm^-(x;h)=e^{-i\psi_0(x)/h}a_\pm^-(x;h),\
x\in I_\pm,
\end{equation}
such that $P_x u_\pm^\delta=O(h^\infty)$ in $C^\infty(I_\pm)$%
\footnote{Henceforth we say that $u=O(h^\infty)$ in
$C^\infty(I)$ for some open set $I$, if for every compact $K\subset I$
and every $N$, $\|u\|_{C^N(K)}=O(h^N)$. In particular, this does not
provide any information on the growth of $u$ at the ends of
$I$. Similarly, we say that $u$ is polynomially bounded in
$C^\infty(I)$ if for every $K$ and $N$, there exists $M$ such that
$\|u\|_{C^N(K)}=O(h^{-M})$.} and
$a_\gamma^\delta$ are smooth classical symbols in $h$, for
$\gamma,\delta\in\{+,-\}$. Indeed, if
$$
a_\gamma^\delta(x;h)\sim\sum_{j\geq 0}h^j a_\gamma^{\delta(j)}(x),
$$
then the functions $a_\gamma^{\delta(j)}$ have to solve the transport
equations
\begin{equation}\label{e:wkb-transport}
\begin{gathered}
(2\psi'_0(x)\partial_x+\psi''_0(x)\pm i V_1(x))a_\gamma^{\pm (0)}=0,\\
(2\psi'_0(x)\partial_x+\psi''_0(x)\pm i V_1(x))a_\gamma^{\pm(j+1)}=
\pm i(\partial_x^2-V_2(x))a_\gamma^{\pm (j)},\ j\geq 0;
\end{gathered}
\end{equation}
the latter can be solved inductively in $j$. We will fix the
normalization of $a_\gamma^{\delta(0)}$ later; right now, we only
require that for $x$ in a compact set, $a_\gamma^{\delta(0)}\sim 1$ in
the sense that $C^{-1}\leq |a_\gamma^{\delta(0)}|\leq C$ for some
$h$-independent constant $C$. Put
\begin{equation}\label{e:radial-gamma-pm}
\Gamma^\pm_\gamma=\{(x,\pm\psi'_0(x))\mid x\in I_\gamma\}\subset T^* I_\gamma,\
\gamma\in\{+,-\};
\end{equation}
then by Proposition~\ref{l:oi-wf} (with $m=0$),
\begin{equation}\label{e:wkb-wf}
\WF(u_\gamma^\delta)\subset \Gamma_\gamma^\delta,\
\gamma,\delta\in \{+,-\}.
\end{equation}

Now, we show that the Cauchy problem for the equation $P_xu=0$ is
well-posed semiclassically in $I_\pm$.  For two smooth functions
$v_1,v_2$ on some interval, define their semiclassical Wronskian by
$$
W(v_1,v_2)=v_1\cdot h \partial_xv_2-v_2\cdot h \partial_x v_1;
$$
then
\begin{equation}\label{e:w-diff}
h\partial_x W(v_1,v_2)=v_2\cdot P_x v_1-v_1\cdot P_x v_2.
\end{equation}
Also, if $W(v_1,v_2)\neq 0$ and $u$ is some smooth function, then
\begin{equation}\label{e:w-sol}
u={W(u,v_1)v_2-W(u,v_2)v_1\over W(v_2,v_1)}.
\end{equation}
We have $W(u_\pm^+,u_\pm^-)\sim 1$; therefore, the following fact applies:
\begin{prop}\label{l:w-cauchy}
Assume that $I\subset \mathbb R$ is an interval and $U\subset I$ is a
nonempty open set. Let $v_1(x;h),v_2(x;h)\in C^\infty(I)$ be two
polynomially bounded functions such that $P_xv_j(x;h)=O(h^\infty)$ in
$C^\infty(I)$ and $W(v_1,v_2)^{-1}$ is polynomially bounded. (Note
that by~\eqref{e:w-diff}, $d_x W(v_1,v_2)=O(h^\infty)$.)  Let
$u(x;h)\in C^\infty(I)$ be polynomially bounded in $C^\infty(U)$ and
$P_x u=O(h^\infty)$ in $C^\infty(I)$. Then
$u=c_1v_1+c_2v_2+O(h^\infty)$ in $C^\infty(I)$, where the constants
$c_1,c_2$ are polynomially bounded.  Moreover,
$c_j=W(u,v_{3-j})/W(v_j,v_{3-j})+O(h^\infty)$.
\end{prop}
\begin{proof}
Let $W_j=W(u,v_j)$. Combining~\eqref{e:w-diff} and~\eqref{e:w-sol}, we
get $|d_x W_j|=O(h^\infty)(|W_1|+|W_2|)$. Also, $W_j$ are polynomially
bounded on $U$.  By Gronwall's inequality, we see that $W_j$ are
polynomially bounded on $I$ and constant modulo $O(h^\infty)$; it
remains to use~\eqref{e:w-sol}.
\end{proof}
Now, recall~\cite[Section~4]{skds} that for $X_0$ large enough, we have
$V(x)=V_\pm(e^{\mp A_\pm x})$ for $\pm x>X_0$, where $A_\pm>0$ are some
constants and $V_\pm(w)$ are holomorphic functions in the discs
$\{|w|<e^{-A_\pm X_0}\}$, and $V_\pm(0)=-\omega_\pm^2$, where
\begin{equation}\label{e:omega-pm}
\omega_\pm=(1+\alpha)((r_\pm^2+a^2)(\tilde\omega+ih\tilde\nu)-a\tilde k).
\end{equation}
For $a$ and $h$ small enough, we have $\Real\omega_\pm>0$.
In~\cite[Section~4]{skds}, we constructed exact solutions $u_\pm(x)$
to the equation $P_x u_\pm=0$ such that
$$u_\pm(x)=e^{\pm i\omega_\pm x/h}v_\pm(e^{\mp A_\pm x})\text{ for }\pm x>X_0,
$$
with $v_\pm(w)$ holomorphic in the discs $\{|w|<e^{-A_\pm X_0}\}$ and
$v_\pm(0)=1$.  Note that we can use a different normalization
condition than~\cite[Proposition~4.2]{skds}, as $\Imag\omega_\pm=O(h)$
under the assumptions~\eqref{e:radial-regime}.
\begin{prop}\label{l:u-pm-wkb}
For a certain normalization of the functions $\psi_0$ and $a_\pm^{+(0)}$,
\begin{equation}\label{e:u-pm-wkb}
u_\pm(x)=u_\pm^+(x)+O(h^\infty)\text{ in }C^\infty(I_\pm).
\end{equation}
In particular, by~\eqref{e:wkb-wf}
\begin{equation}\label{e:u-pm-wf}
\WF(u_\pm|_{I_\pm})\subset \Gamma_\pm^+.
\end{equation}
\end{prop}
\begin{proof}
We will consider the case of $u_+$.
By Proposition~\ref{l:w-cauchy}, it is enough to show~\eqref{e:u-pm-wkb}
for $\pm x>X_0$, where $X_0$ is large, but fixed.
We choose $X_0$ large enough so that $\Real V_\pm(w)<0$ for
$|w|\leq e^{-A_\pm X_0}$. Then there exists a function $\psi(x)$
such that
$$
\begin{gathered}
(\partial_x\psi(x))^2+V(x)=0,\
x>X_0;\\
\psi(x)=\omega_+x+\tilde\psi(e^{-A_+x}),
\end{gathered}
$$
with $\tilde\psi$ holomorphic in $\{|w|<e^{-A_+ X_0}\}$.
We can fix $\psi$ by requiring that $\tilde\psi(0)=0$.
Take
$$
u_+(x)=e^{i\psi(x)/h}a(e^{-A_+x};h);
$$
then $P_xu_+=0$ if and only if
$$
([hA_+wD_w+A_+w\tilde\psi'(w)-\omega_+]^2+V)a=0.
$$
This can be rewritten as
$$
-A_+(w\tilde\psi'(w))'a+(2\omega_++ihA_+-2A_+w\tilde\psi'(w))\partial_w a+ihA_+w\partial_w^2 a=0.
$$
We will solve this equation by a power series in $w$ and estimate the terms of this series
uniformly in $h$. Let us write
$$
\tilde\psi'(w;h)=\sum_{l\geq 0}\psi_l(h) w^l,\
a(w;h)=\sum_{j\geq 0}a_j(h) w^j
$$
and solve for $a_j$ with the initial condition $a_0=1$, obtaining
$$
a_{j+1}(h)={A_+\over (j+1)(2\omega_++ihA_+(j+1))}\sum_{0\leq l\leq j} \psi_l(h)(1+2j-l)a_{j-l}(h).
$$
We claim that for some $R$, all $j$, and small $h$, $|a_j(h)|\leq R^j$. Indeed, we
have $|2\omega_++ihA_+(j+1)|\geq \varepsilon>0$; combining this with an estimate
on $\psi_l$, we get
$$
|a_{j+1}|\leq {C\over j+1}\sum_{0\leq l\leq j}S^l(1+2j-l)|a_{j-l}|
\leq 2C\sum_{0\leq l\leq j}S^l|a_{j-l}|
$$
for some constants $C$ and $S$. We can then conclude by induction if
$R\geq 2C+S$. In a similar way, we can estimate the derivatives of
$a_j$ in $h$; therefore, $a(w;h)$ is a classical symbol for
$|w|<R^{-1}$.

Now, we take $X_0$ large enough so that $e^{-A_+X_0}<R^{-1}$
and restrict ourselves to real $x>X_0$. We can normalize $\psi_0$
so that $\psi(x)=\psi_0(x)+h\psi_1(x;h)$ for some classical symbol
$\psi_1$; then
$$
u_+(x)=e^{i\psi_0(x)/h}[e^{i\psi_1(x;h)}a(e^{-A_+x};h)].
$$
The expression in square brackets is a classical symbol; therefore,
this expression solves the transport
equations~\eqref{e:wkb-transport}; it is then equal to a constant
times $a_+^+$, modulo $O(h^\infty)$ errors.
\end{proof}

\subsection{Transmission through the barrier}\label{s:radial-barrier}

First of all, we establish a microlocal normal form for $P_x$ near the
potential maximum.  Let $\varepsilon_0>0$ be small; define
$$
\begin{gathered}
K_0=\{|x-x_0|\leq \varepsilon_0,\
|\xi|\leq \varepsilon_0\}\subset T^* \mathbb R.
\end{gathered}
$$
We pick $\varepsilon_r$ small enough, depending on $\varepsilon_0$,
such that $\varepsilon_r<\varepsilon_0/2$ and
$$
\{p_0=0\}\subset K_0\cup \bigcup_{\gamma,\delta}\Gamma_\gamma^\delta,
$$
with $\Gamma_\gamma^\delta$ defined in~\eqref{e:radial-gamma-pm}.
(Recall from Section~\ref{s:radial-trapping} that $\varepsilon_r$
controls how close we are to the trapping region.)
We also assume that $\varepsilon_0$ is small enough so that $(x_0,0)$
is the only critical point of $p_0$ in $K_0$.
\begin{figure}
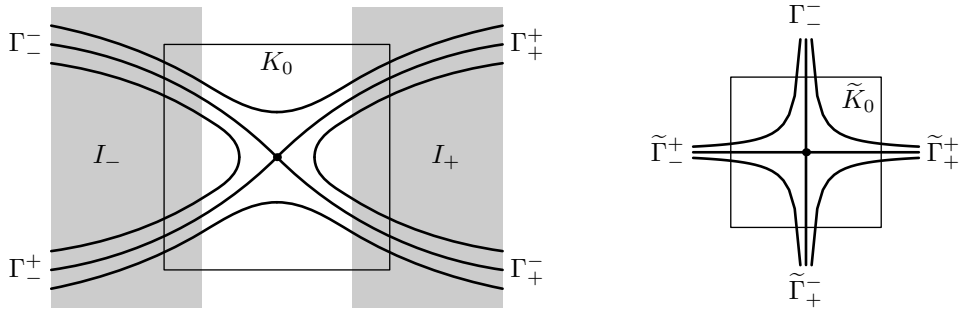

\includegraphics{zeeman.1}
\hskip.5in
\includegraphics{zeeman.2}
\caption{Level lines of $p_0$ before and after conjugation by $\Phi$.}
\end{figure}
\begin{prop}\label{l:radial-conj}
For $\varepsilon_0$ small enough and $\varepsilon_r$ small enough
depending on $\varepsilon_0$, there exists a symplectomorphism $\Phi$
from a neighborhood of $K_0$ onto a neighborhood of the origin in $T^*
\mathbb R$ and operators $B_1,B_2$ quantizing $\Phi$ near
$K_0\times\Phi(K_0)$ in the sense of
Proposition~\ref{l:quantization-canonical}, such that $P_x$ is
intertwined via $(B_1,B_2)$ with the operator $SQ(\beta)$ microlocally
near $K_0\times \Phi(K_0)$, with $S\in\Psi^0_{\cl}$ elliptic in
the class $\Psi^0(\mathbb R)$,
$$
Q(\beta)=hxD_x-\beta,
$$
and $\beta=\beta(\tilde\omega,\tilde\nu,\tilde\lambda,\tilde\mu,\tilde k;h)$ a classical
symbol. Moreover, the principal part $\beta_0$ of $\beta$ is
real-valued, independent of $\tilde\nu,\tilde\mu$, and vanishes if and only
$\tilde\lambda=\tilde\lambda_0(\tilde\omega,\tilde k)$. Also,
\begin{gather}
\label{e:radial-beta-1}
\beta_0=-{V_0(x_0)\over \sqrt{-2V''_0(x_0)}}+O(V_0(x_0)^2),\\
\label{e:radial-beta-2}
\beta_0={27M_0^2\tilde\omega^2-\tilde\lambda(1-9\Lambda M_0^2)\over
2\sqrt{\tilde\lambda}(1-9\Lambda M_0^2)}+O(V_0(x_0)^2)\text{ for }a=0.
\end{gather}
Finally, $\Phi(K_0)\supset\widetilde K_0=\{|x|\leq\tilde\varepsilon_0,\
|\xi|\leq\tilde\varepsilon_0\}$ for some $\tilde\varepsilon_0$
depending on $\varepsilon_0$, and, with $I_0$ defined in~\eqref{e:intervals},
\begin{gather}
\label{e:radial-mapping+}
\Phi(\Gamma_\pm^+)\subset \widetilde\Gamma_\pm^+=
\{\pm x>0,\ \xi=\beta_0/x,\ |\xi|\leq \tilde\varepsilon_0/2\},\\
\label{e:radial-mapping-}
\Phi(\Gamma_\pm^-)\subset \widetilde\Gamma_\pm^-=
\{\mp \xi>0,\ x=\beta_0/\xi,\ |x|\leq\tilde\varepsilon_0/2\};\\
\label{e:radial-mapping0}
\{p_0=0\}\cap \{x\in I_0\}\subset \Phi^{-1}(\widetilde K_0)\subset K_0.
\end{gather}
\end{prop}
\begin{proof}
First of all, we use~\cite[Theorem~12]{cdv-p} to construct
$\Phi,B_1,B_2$ conjugating $P_x$ microlocally near the critical point
$(x_0,0)$ to an operator of the
form $f[hxD_x]$, for some symbol $f(s;h)$, where the latter employs
the formal functional calculus of Section~\ref{s:prelim-functional}.
The techniques in the proof are similar to those of
Section~\ref{s:angular-moser} of the present paper, with appropriate
replacements for Propositions~\ref{l:arnold-liouville-degenerate}
and~\ref{l:moeser-lemma}; therefore, the proof goes through for
complex valued symbols with real principal part.

Let $f_0$ be the principal part of $f$; then
$p_0\circ\Phi^{-1}=f_0(x\xi)$. Note that $\Phi(x_0,0)=(0,0)$.  The
level set $\{p_0=V_0(x_0)\}$ at the trapped energy contains in
particular the outgoing trajectory $\{x>x_0,\
\xi=\sqrt{V_0(x_0)-V_0(x)}\}$; we can choose $\Phi$ mapping this
trajectory into $\{x>0,\ \xi=0\}$. Since the latter is also outgoing
for the Hamiltonian flow of $x\xi$, we have $\partial_s f_0(0)>0$; it
follows that $\partial_s f_0(s)>0$ for all $s$ (if $\partial_s f_0$
vanishes, then $p_0$ has a critical point other than $(x_0,0)$). The
function $f(s;h)$ not uniquely defined; however, its Taylor
decomposition at $s=0$, $h=0$ is and we can compute in particular
\begin{equation}\label{e:radial-beta-3}
f_0(s)=V_0(x_0)+s\sqrt{-2V''_0(x_0)}
+O(s^2).
\end{equation}
Therefore, for $\varepsilon_r$ small enough, we can solve the equation
$f(s;h)=0$ for $s$; let $\beta$ be the solution.  We now write
$f(s;h)=f_1(s;h)(s-\beta)$ for some nonvanishing $f_1$ and get
$f[hxD_x]=SQ(\beta)$ microlocally in $\Phi(K_0)$, with $S=f_1[hxD_x]$
in $\Phi(K_0)$ and extended to be globally elliptic outside of this
set. The equation~\eqref{e:radial-beta-1} follows from~\eqref{e:radial-beta-3},
while~\eqref{e:radial-beta-2} follows from~\eqref{e:radial-beta-1}
and~\eqref{e:pot-0}.

Finally, $p_0=0$ on each $\Gamma_\gamma^\delta$ and thus $x\xi=\beta$
on $\Phi(\Gamma_\gamma^\delta)$. By analysing the properties of $\Phi$
near $(x_0,0)$, we can deduce which part of the sets $\{x\xi=\beta\}$
each $\Gamma_\gamma^\delta$ maps into;
\eqref{e:radial-mapping+}--\eqref{e:radial-mapping0}
follow for $\varepsilon_r$ small enough.
\end{proof}
We now describe the radial quantization condition and provide a
non-rigorous explanation for it.  Recall from~\cite[Section~4]{skds}
that $(\omega,\lambda,k)$ is a pole of $R_r$ if and only if the
functions $u_\pm$, studied in the previous subsection, are multiples
of each other. Assume that this is true and $u=u_+\sim u_-$.  However,
by~\eqref{e:u-pm-wf} the function $B_1u$ is microlocalized on the
union of $\widetilde\Gamma_\pm^+$, but away from $\widetilde\Gamma_\pm^-$; it
also solves $Q(\beta)B_1u=0$ microlocally.  By propagation of
singularities, this can happen only if the characteristic set of
$Q(\beta)$ is $\{x\xi=0\}$, and in this case, $B_1u$ is smooth near
$x=0$. Then $B_1u$ must be given by $x^{i\beta/h}$, with $\beta\in
-ih \mathbb N$ and $\mathbb N$ denoting the set of nonnegative integers.
Therefore, we define the radial quantization symbol $\mathcal
F^r(m,\tilde\omega,\tilde\nu,\tilde k;h)$ as the solution
$\tilde\lambda+ih\tilde\mu$ to the equation
$$
\beta(\tilde\omega,\tilde\nu,\tilde\lambda,\tilde\mu,\tilde k;h)=-ihm,\
m\in \mathbb Z,\
0\leq m\leq C_m.
$$
The expansions for $\mathcal F^r$ near $a=0$ described in Proposition~\ref{l:radial}
follow from~\eqref{e:radial-beta-2} and~\eqref{e:numerics-radial-stuff}.

We now prove the rest of Proposition~\ref{l:radial}. We start with quantifying
the statement that in order for the equation $Q(\beta)u=0$ to
have a nontrivial solution smooth near $x=0$, the quantization condition
must be satisfied:
\begin{prop}\label{l:radial-q-technical}
Assume that $\beta\in \mathbb C$ satisfies
\begin{equation}\label{e:radial-technical-beta}
|\beta|\leq C_\beta,\
|\Imag\beta|\leq C_\beta h,\
\min_{m\in \mathbb N} |\beta+ihm|\geq C_\beta^{-1}h,
\end{equation}
for some constant $C_\beta$. Let $U\subset \mathbb R$ be a bounded
open interval, $I\subset U$ be a compact interval centered at zero,
and $X\in\Psi^{\comp}_{\cl}(\mathbb R)$. Then there exist
constants $C$ and $N$ such that for each $u\in L^2(\mathbb R)$,
\begin{equation}\label{e:radial-technical-0}
\|Xu\|_{L^2(I)}\leq Ch^{-N}\|Q(\beta)Xu\|_{L^2(U)}+O(h^\infty)\|u\|_{L^2(\mathbb R)}.
\end{equation}
\end{prop}
\begin{proof}
First, assume that $\Imag\beta\geq h$. Let $I'$ be an interval
compactly contained in $I$ and centered at zero. We will use the fact
that every $C^\infty(I')$ seminorm of $Xu$ is bounded by
$Ch^{-N}\|Xu\|_{L^2(I)}+O(h^\infty)\|u\|_{L^2(\mathbb R)}$ for some
constants $C$ and $N$, depending on the seminorm chosen. (Henceforth
$C$ and $N$ will be constants whose actual values may depend on the
context.)  Since $Xu\in C^\infty$, we can write
$$
hD_x(x^{-i\beta/h}Xu)=x^{-1-i\beta/h}Q(\beta)Xu.
$$
However, $\Real(-i\beta/h)\geq 1$; therefore, $x^{-i\beta/h}Xu$
vanishes at $x=0$ and we can integrate to get
\begin{equation}\label{e:radial-technical-1}
\|x^{-i\beta/h}Xu\|_{L^\infty(I)}\leq C\|Q(\beta)Xu\|_{L^2(I)}.
\end{equation}
On the other hand, 
\begin{equation}\label{e:radial-technical-2}
\|Xu\|_{L^\infty(I')}\leq Ch^{-N_0}\|Xu\|_{L^2(I)}+O(h^\infty)\|u\|_{L^2(\mathbb R)}.
\end{equation}
for some constants $C$ and $N_0$.  Now, take a large constant
$\varkappa$; using~\eqref{e:radial-technical-2} in
$\{|x|<h^\varkappa\}$ and~\eqref{e:radial-technical-1} elsewhere, we
get
$$
\|Xu\|_{L^2(I)}\leq Ch^{\varkappa/2-N_0}\|Xu\|_{L^2(I)}+h^{-\varkappa C_\beta}\|Q(\beta)Xu\|_{L^2(I)}
+O(h^\infty)\|u\|_{L^2(\mathbb R)};
$$
taking $\varkappa$ large enough, we get~\eqref{e:radial-technical-0}.

For the general case, we choose an integer $M$ large enough so that
$\Imag(\beta+ihM)\geq h$. Let $v$ be $M$-th Taylor polynomial of $Xu$
at zero. Since
$$
\|Q(\beta+ihM)D_x^M Xu\|_{L^2(I)}=\|D_x^MQ(\beta)Xu\|_{L^2(I)}
\leq Ch^{-N}\|Q(\beta)Xu\|_{L^2(U)}+O(h^\infty)\|u\|_{L^2(\mathbb R)},
$$
we apply the current proposition for the case $\Imag\beta\geq h$
considered above to get
$$
\|Xu-v\|_{L^2(I)}\leq C\|D_x^MXu\|_{L^2(I)}
\leq Ch^{-N}\|Q(\beta)Xu\|_{L^2(U)}+O(h^\infty)\|u\|_{L^2(\mathbb R)};
$$
therefore, $\|Q(\beta)v\|_{L^2(I)}$ is bounded by the same expression.
However, one can verify directly that if $\beta$ is $C_\beta^{-1}h$
away from $-ih \mathbb N$, then
$$
\|v\|_{L^2(I)}\leq Ch^{-N}\|Q(\beta)v\|_{L^2(I)};
$$
this completes the proof.
\end{proof}
Now, we show that each radial pole lies within $o(h)$ of a pseudopole:
\begin{prop}
Assume that $\beta(h)$ satisfies~\eqref{e:radial-technical-beta}.
Then for $h$ small enough, $(\omega,\lambda,k)$ is not a radial pole,
and for each compact interval $I\subset \mathbb R$, there exist
constants $C$ and $N$ such that
$$
\|1_I R_r(\omega,\lambda,k) 1_I\|_{L^2\to L^2}\leq Ch^{-N}.
$$
\end{prop}
\begin{proof}
Let $u\in H^2_{\loc}(\mathbb R)$ be an outgoing solution to the
equation $P_xu=f$, with $f\in L^2$ supported in a fixed compact subset
inside the open interval $I$.  Then
$$
u(x)=c_\pm u_\pm(x),\ \pm x\gg 0,
$$
for some constants $c_\pm$. Clearly, $|c_\pm|\leq C\|u\|_{L^2(I)}$.
Using the method of proof of Proposition~\ref{l:w-cauchy}, we get
\begin{equation}\label{e:radial-propagation}
\|u-c_\pm u_\pm\|_{L^2(I\cap I_\pm)}\leq Ch^{-1}\|f\|_{L^2}.
\end{equation}
Next, let $\chi\in C_0^\infty(I_0)$ be equal to 1 near the complement of
$I_+\cup I_-$ and $B_1$ be the operator introduced in
Proposition~\ref{l:radial-conj}; consider the compactly microlocalized
operator
$$
T=SQ(\beta)B_1-B_1P_x.
$$
Then by Proposition~\ref{l:quantization-canonical}, we can write
$T=TX+O(h^\infty)$, where $X\in\Psi^{\comp}_{\cl}$ is a certain
operator vanishing microlocally on $K_0$.
By~\eqref{e:radial-mapping0}, we can further write $X=X_1+X_2$, where
$X_j\in\Psi^{\comp}_{\cl}$, $\WF(X_j)\cap K_0=\emptyset$,
$\WF(X_1)\cap \{p_0=0\}=\emptyset$ and $\WF(X_2)\subset \{x\in I_+\cup
I_-\}$. By ellipticity,
$$
\|X_1 u\|_{L^2}\leq C\|f\|_{L^2}+O(h^\infty)\|u\|_{L^2(I)}.
$$
Take $\chi_\pm\in C_0^\infty(I_\pm)$ such that $\chi_\pm=1$
near $I_\pm\cap \pi(\WF(X_2))$ (here $\pi:\mathbb T^*\mathbb R\to \mathbb R$
is the projection map onto the base variable and $\pi(\WF(X_2))$ is a compact
subset of $I_+\cup I_-$); then by~\eqref{e:radial-propagation},
$$
\begin{gathered}
\|X_2 (u-u_1)\|_{L^2}\leq Ch^{-1}\|f\|_{L^2}+O(h^\infty)\|u\|_{L^2(I)},\\
u_1=c_+\chi_+(x)u_++c_-\chi_-(x)u_-.
\end{gathered}
$$
It follows that
$$
\|SQ(\beta)B_1u-TX_2u_1\|_{L^2}\leq Ch^{-1}\|f\|_{L^2}+O(h^\infty)\|u\|_{L^2(I)}.
$$
Combining~\eqref{e:u-pm-wf} with~\eqref{e:radial-mapping+}
and the fact that $\WF(X_2)\cap K_0=\emptyset$, we get
$$
\WF(TX_2\chi_\pm(x)u_\pm)\subset \widetilde\Gamma_\pm^+\setminus \widetilde K_0.
$$
The projections of the latter sets onto the $x$ variable do not
intersect $\tilde I_0=\{|x|\leq\tilde\varepsilon_0\}$; therefore, for
some open $\widetilde U_0$ containing $\tilde I_0$,
$$
\|TX_2u_1\|_{L^2(\widetilde U_0)}=O(h^\infty)\|u\|_{L^2(I)}.
$$
Using ellipticity of $S$, we then get
$$
\|Q(\beta)B_1u\|_{L^2(\widetilde U_0)}\leq Ch^{-1}\|f\|_{L^2}+O(h^\infty)\|u\|_{L^2(I)}. 
$$
Applying Proposition~\ref{l:radial-q-technical} to $B_1u$ on $I_0$ and
using that $B_1$ is compactly microlocalized, we get
$$
\|\widetilde XB_1u\|_{L^2}\leq Ch^{-N}\|f\|_{L^2}+O(h^\infty)\|u\|_{L^2(I)},
$$
for any $\widetilde X\in\Psi^{\comp}_{\cl}$ microlocalized in $\widetilde
K_0$.  Using the elliptic estimate and~\eqref{e:radial-mapping0}, we
get
$$
\|u\|_{L^2(I_0)}\leq Ch^{-N}\|f\|_{L^2}+O(h^\infty)\|u\|_{L^2(I)}.
$$
From here by~\eqref{e:radial-propagation},
$$
|c_\pm|\leq C\|c_\pm u_\pm\|_{L^2(I_\pm\cap I_0)}
\leq Ch^{-N}\|f\|_{L^2}+O(h^\infty)\|u\|_{L^2(I)};
$$
combining the last two estimates with~\eqref{e:radial-propagation}, we
get the required estimate:
$$
\|u\|_{L^2(I)}\leq Ch^{-N}\|f\|_{L^2}+O(h^\infty)\|u\|_{L^2(I)}.\qedhere
$$
\end{proof}
To finish the proof of Proposition~\ref{l:radial}, it remains to show
\begin{prop}\label{l:radial-last}
Fix $\tilde\omega,\tilde\nu,\tilde k$ satisfying~\eqref{e:radial-regime},
$m\in\mathbb N$ bounded by a large constant $C_m$, and let
$V$ be the set of all $\lambda$ such that
$$
|\beta(\tilde\omega,\tilde\nu,\tilde k,\tilde\lambda,\tilde\mu;h)+ihm|<h/3.
$$
Then for $h$ small enough, $R_r(\omega,\lambda,k)$ has a unique pole
$\lambda_0$ in $V$, and $\lambda_0$ is within $O(h^\infty)$
of $\mathcal F^r(m,\tilde\omega,\tilde\nu,\tilde k;h)$.
Moreover, we can write
$$
R_r(\omega,\lambda,k)={S(\lambda)\over \lambda-\lambda_0},\
\lambda\in V,
$$
where the family of operators $S(\lambda):L^2_{\comp}(\mathbb R)
\to L^2_{\loc}(\mathbb R)$ is bounded polynomially in $h$
and $S(\lambda_0)$ is a rank one operator.
\end{prop}
\begin{proof}
We will use Proposition~\ref{l:radial-conj} and the fact that $\beta=O(h)$ to
extend the WKB solutions $u_\pm^+$ from $I_\pm$ to the whole $\mathbb
R$. Consider the locally integrable functions
$$
\tilde u_\pm^+=(x\pm i0)^{i\beta/h}.
$$
solving the equation $Q(\beta)\tilde u_\pm^+=0$. (See for
example~\cite[Section~3.2]{ho} for the definition and basic properties
of $(x\pm i0)^b$.) We have
$$
\WF(\tilde u_\pm^+)\subset \{\xi=0\}\cup \{x=0,\ \pm\xi>0\}; 
$$
$\tilde u_\pm^+(x)=x^{i\beta/h}$ microlocally near $\{x>0,\ \xi=0\}$,
$\tilde u_\pm^+(x)=e^{\mp \pi\beta/h}(-x)^{i\beta/h}$ microlocally near
$\{x<0,\ \xi=0\}$. Using the formulas for the Fourier transform of $\tilde u_\pm^+$
\cite[Example~7.1.17]{ho}, we get
$$
\tilde u_\pm^+(x)={h^{i\beta/h}e^{\mp \beta\pi/(2h)}\over \Gamma(-i\beta/h)}
\int_0^\infty\chi(\xi) \xi^{-1-i\beta/h}e^{\pm ix\xi/h}\,d\xi
$$
microlocally near $\{x=0,\ \pm\xi\in K_\xi\}$, for every $\chi\in
C_0^\infty(0,\infty)$ such that $\chi=1$ near $K_\xi\subset
(0,\infty)$. Let $B_2$ be the operator constructed in
Proposition~\ref{l:radial-conj}. By~\eqref{e:radial-mapping0}
$P_xB_2\tilde u_\pm^+=O(h^\infty)$ in $C^\infty(I_0)$, and $B_2\tilde
u^+_\pm=\tilde c_\pm u_\pm^++O(h^\infty)$ in $C^\infty(I_\pm\cap I_0)$
for some constants $\tilde c_\pm\sim 1$. To prove the latter, we can
use the theory of Fourier integral operators and Lagrangian
distributions to represent $B_2\tilde u^+_\pm$ in the
form~\eqref{e:radial-wkb-form} microlocally near $\Gamma_\pm^+$; the
symbols in these WKB expressions will have to solve the transport
equations. Then we can use $B_2\tilde u^+_\pm$ to extend $u_\pm^+$ to
$I_\pm\cup I_0$ so that $P_xu_\pm^+=O(h^\infty)$ there. We claim that
\begin{equation}\label{e:radial-transition}
u_\pm^+=c_{\pm 1}u_\mp^++c_{\pm 2}\Gamma(-i\beta/h)^{-1}u_\mp^-+O(h^\infty) 
\end{equation}
in $C^\infty(I_\mp\cap I_0)$, with $c_{\pm j}$ constants such that
$c_{\pm j}$ and $c_{\pm j}^{-1}$ are polynomially bounded in $h$.  To
show~\eqref{e:radial-transition}, we can apply the theory of Lagrangian
distributions to $B_2\tilde u_\pm^+$ one more time; alternatively, we
know that this function is an $O(h^\infty)$ approximate solution to
the equation $P_xu=0$ on $I_0\cap I_\mp$, and we have control on its
$L^2$ norm when microlocalized to $\Gamma^+_\mp$ and $\Gamma^-_\mp$.
Thus, we can extend $u_\pm^+$ to the whole $\mathbb R$ as a
polynomially bounded family with $P_x u_\pm^+=O(h^\infty)$ in
$C^\infty(\mathbb R)$ and~\eqref{e:radial-transition} holding on
$I_\mp$. Similarly we can extend $u_\pm^-$; using either of the
families $(u_\pm^+,u_\pm^-)$ in Proposition~\ref{l:w-cauchy} together
with Proposition~\ref{l:u-pm-wkb}, we get $u_\pm=u_\pm^++O(h^\infty)$
in $C^\infty(\mathbb R)$. It now follows
from~\eqref{e:radial-transition} that
$$
W(u_+,u_-)=c(\Gamma(-i\beta/h)^{-1}+O(h^\infty)),
$$
with $c$ and $c^{-1}$ bounded polynomially in $h$. By~\cite[(4.8)]{skds}, we get
$$
R_r(\omega,\lambda,k)={\widetilde S(\omega,\lambda,k)\over W(u_+,u_-)}, 
$$
with the family $\widetilde S$ holomorphic and bounded polynomially in $h$.
Moreover, for $W(u_+,u_-)=0$, $\widetilde S$ is proportional to $u_+\otimes u_+$
and thus has rank one. We are now done if we let
$\lambda_0$ be the unique solution to the equation
$W(u_+,u_-)=0$ in $V$.
\end{proof}

\appendix
\section{Grushin problems for several commuting operators}\label{s:prelim-grushin}

\subsection{Global Grushin problem}\label{s:prelim-grushin-1}

Assume that $P_1,\dots,P_n$ are pseudodifferential operators on a
compact manifold $M$, with $P_j\in\Psi^{k_j}(M)$ and $k_j\geq 0$.
\begin{defi}\label{d:joint-spectrum}
We say that $\lambda=(\lambda_1,\dots,\lambda_n)\in \mathbb C^n$
belongs to the joint spectrum%
\footnote{Strictly speaking, this is the definition of the joint \emph{point} spectrum.
However, the operators we study in Section~\ref{s:angular} are joint elliptic near
the fiber infinity,
as in Proposition~\ref{l:local-grushin}, thus all joint spectrum is given by eigenvalues.}
of $P_1,\dots,P_n$, if the joint
eigenspace
$$
\{u\in C^\infty(M)\mid P_ju=\lambda_ju,\ j=1,\dots,n\}
$$
is nontrivial. (In our situation, one of the operators $P_j$ will be
elliptic outside of a compact set, so all joint eigenfunctions will be
smooth.)
\end{defi}
The goal of this appendix is to extract information about the joint
spectrum of $P_1,\dots,P_n$ from certain microlocal information.
Essentially, we will construct exact joint eigenfunctions based on
approximate eigenfunctions and certain invertibility conditions.  The
latter will be given in the form of operators $A_1,\dots,A_n$, with
the following properties:
\begin{enumerate}
\item[(G1)] Each $A_j$ can be represented as $A'_j+A''_j$, where
$A'_j$ is compactly microlocalized and has operator norm $O(h^{-r})$;
 $A''_j\in h^{-r}\Psi^{-k_j}(M)$.
Here $r>0$ is a constant.
\item[(G2)] The commutator of any two of the operators $P_1,\dots,P_n,A_1,\dots,A_n$
lies in $h^\infty\Psi^{-\infty}(M)$.
\end{enumerate}
We would like to describe the joint spectrum of $P_1,\dots,P_n$ in a
ball of radius $o(h^r)$ centered at zero. First, we consider a
situation when there is no joint spectrum:
\begin{prop}\label{l:global-grushin-empty}
Assume that conditions (G1) and (G2) hold and additionally,
$$
\sum_{j=1}^n A_jP_j=I\mod h^\infty\Psi^{-\infty}(M).
$$
Then there exists $\delta>0$ such that for $h$ small enough, the
ball of radius $\delta h^r$ centered at zero contains no joint
eigenvalues of $P_1,\dots,P_n$.
\end{prop}
\begin{proof}
Assume that $u\in L^2(M)$ and $P_j u=\lambda_j u$, where
$|\lambda_j|\leq \delta h^r$.  Then
$$
0=\sum_{j=1}^n A_j(P_j-\lambda_j)u=(I+h^\infty\Psi^{-\infty})u-\sum_{j=1}^n \lambda_jA_j u.
$$
It follows from condition (G1) that $\|A_j\|_{L^2\to L^2}=O(h^{-r})$; therefore,
$$
\|u\|_{L^2}=O(\delta+h^\infty)\|u\|_{L^2}
$$
and we must have $u=0$ for $\delta$ and $h$ small enough.
\end{proof}
Now, we study the case when the joint spectrum is nonempty. Assume
that $S_1:\mathbb C\to C^\infty(M)$ and $S_2:\mathcal D'(M)\to \mathbb
C$ are operators with the following properties:
\begin{enumerate}
\item[(G3)] Each $S_j$ is compactly microlocalized with operator norm $O(1)$.
\item[(G4)] $S_2S_1=1+O(h^\infty)$.
\item[(G5)] If $Q$ is any of the operators $P_1,\dots,P_n,A_1,\dots,A_n$,
then $QS_1\in h^\infty\Psi^{-\infty}$ and
$S_2Q\in h^\infty\Psi^{-\infty}$.
\item[(G6)] We have
$$
\sum_{j=1}^n A_jP_j=I-S_1S_2\mod h^\infty\Psi^{-\infty}(M).
$$ 
\end{enumerate}
Note that (G5) implies that the image of $S_1$ consists of
$O(h^\infty)$-approximate joint eigenfunctions. For $n=1$, one
recovers existence of exact eigenfunctions from approximate ones using
Grushin problems, based on Schur complement formula; see for
example~\cite[Section~6]{h-s}. The proposition below constructs an
analogue of these Grushin problems for the case of several operators.
This construction is more involved, since we need to combine the fact
that $P_j$ commute exactly, needed for the existence of joint
spectrum, with microlocal assumptions (G1)--(G6) having $O(h^\infty)$
error. Note also that condition (G2) does not appear in the case
$n=1$.
\begin{prop}\label{l:global-grushin}
Assume that the conditions (G1)--(G6) hold and the operators
$P_1,\dots,P_n$ commute exactly; that is, $[P_j,P_k]=0$ for all $j,k$.
Then there exists $\delta>0$ such that for $h$ small enough, the
ball of radius $\delta h^r$ contains exactly one joint eigenvalue
of $P_1,\dots,P_n$. Moreover, this eigenvalue is $O(h^\infty)$ and the
corresponding eigenspace is one dimensional.
\end{prop}
\begin{proof}
We prove the proposition in the case $n=2$ (which is the case we will
need in the present paper); the proof in the general case can be found
in Appendix~\ref{s:prelim-grushin-3}.

For $\lambda=(\lambda_1,\lambda_2)\in \mathbb C^2$, consider the
operator
$$
\begin{gathered}
T(\lambda)=\begin{pmatrix}
P_1-\lambda_1&-A_2&S_1&0\\
P_2-\lambda_2&A_1&0&S_1\\
S_2&0&0&0\\
0&S_2&0&0
\end{pmatrix}:\mathcal H_1\to \mathcal H_2;\\
\mathcal H_1=L^2(M)\oplus H_h^{-k_1-k_2}(M)\oplus \mathbb C^2,\
\mathcal H_2=H_h^{-k_1}(M)\oplus H_h^{-k_2}(M)\oplus \mathbb C^2.
\end{gathered}
$$
The conditions (G1)--(G6) imply that for
$$
Q=\begin{pmatrix}
A_1&A_2&S_1&0\\
-P_2&P_1&0&S_1\\
S_2&0&0&0\\
0&S_2&0&0
\end{pmatrix}:\mathcal H_2\to \mathcal H_1,
$$
we have $T(0)Q=I+O_{\mathcal H_2\to \mathcal H_2}(h^\infty)$,
$QT(0)=I+O_{\mathcal H_1\to \mathcal H_1}(h^\infty)$.  By (G1) and
(G3), we have $\|Q\|_{\mathcal H_2\to \mathcal H_1}=O(h^{-r})$;
therefore, if $\delta>0$ and $h$ are small enough and
$|\lambda|\leq \delta h^r$, then $T(\lambda)$ is invertible and
$$
\|T(\lambda)^{-1}\|_{\mathcal H_2\to \mathcal H_1}=O(h^{-r}).
$$
Now, let $|\lambda|\leq\delta h^r$ and put
$$
(u(\lambda),u_2(\lambda),f(\lambda))=T(\lambda)^{-1}(0,0,1,0),\
f(\lambda)=(f_1(\lambda),f_2(\lambda))\in \mathbb C^2.
$$
This is the only solution to the following system of equations, which
we call global Grushin problem:
\begin{align}
\label{e:magic-system-1}
(P_1-\lambda_1)u(\lambda)-A_2u_2(\lambda)+S_1f_1(\lambda)&=0,\\
\label{e:magic-system-2}
(P_2-\lambda_2)u(\lambda)+A_1u_2(\lambda)+S_1f_2(\lambda)&=0,\\
\label{e:magic-system-3}
S_2u(\lambda)&=1,\\
\label{e:magic-system-4}
S_2u_2(\lambda)&=0.
\end{align}
We claim that $\lambda$ is an element of the joint spectrum if and
only if $f(\lambda)=0$, and in that case, the joint eigenspace is one
dimensional and spanned by $u(\lambda)$. First, assume that $u$ is a
joint eigenfunction with the eigenvalue $\lambda$. Then
$T(\lambda)(u,0,0,0)=(0,0,s,0)$, where $s$ is some nonzero number; it
follows immediately that $f(\lambda)=0$ and $u$ is a multiple of
$u(\lambda)$.

Now, assume that $f(\lambda)=0$; we need to prove that $u(\lambda)$ is
a joint eigenfunction for the eigenvalue
$\lambda$. By~\eqref{e:magic-system-1} and~\eqref{e:magic-system-2},
it suffices to show that $u_2(\lambda)=0$. For that, we
multiply~\eqref{e:magic-system-2} by $P_1-\lambda_1$ and
subtract~\eqref{e:magic-system-1} multiplied by $P_2-\lambda_2$; since
$f(\lambda)=0$ and $[P_1,P_2]=0$, we get
$$
((P_1-\lambda_1)A_1+(P_2-\lambda_2)A_2)u_2(\lambda)=0.
$$ 
Recalling (G6), we get
$$
(I-S_1S_2+O_{H_h^{-k_1-k_2}\to H_h^{-k_1-k_2}}(\delta+h^\infty))u_2(\lambda)=0.
$$
By~\eqref{e:magic-system-4},
$(I+O(\delta+h^\infty))u_2(\lambda)=0$ and thus
$u_2(\lambda)=0$. The claim is proven.

It remains to show that the equation $f(\lambda)=0$ has exactly one
root in the disc of radius $\delta h^r$ centered at zero, and
this root is $O(h^\infty)$.  For that, let $QT(\lambda)=I-R(\lambda)$;
we have
$$
\begin{gathered}
R(\lambda)=\begin{pmatrix}
\lambda_1A_1+\lambda_2A_2&0&0&0\\
-\lambda_1P_2+\lambda_2P_1&0&0&0\\
\lambda_1S_2&0&0&0\\
\lambda_2S_2&0&0&0
\end{pmatrix}
+O_{\mathcal H_1\to \mathcal H_1}(h^\infty);\\
T(\lambda)^{-1}=(I+R(\lambda)+(I-R(\lambda))^{-1}R(\lambda)^2)Q.
\end{gathered}
$$
One can verify that $R(\lambda)^2Q(0,0,1,0)=O_{\mathcal
H_1}(h^\infty)$ and then
$$
f(\lambda)=\lambda-g(\lambda;h),
$$
where $g(\lambda;h)=O(h^\infty)$ uniformly in $\lambda$. It remains to
apply the contraction mapping principle.
\end{proof} 
Finally, we establish a connection between global Grushin problem and
meromorphic resolvent expansions, using some more information about
our particular application:
\begin{prop}\label{l:global-grushin-special}
Assume that $n=2$, $P_1,P_2$ satisfy the properties
stated in the beginning of this subsection, $[P_1,P_2]=0$, $k_1>0$, and
$P_1-\lambda$ is elliptic in the class $\Psi^{k_1}$ for some
$\lambda\in \mathbb C$.  If $V$ is the kernel of $P_2$, then by
analytic Fredholm theory (see for example~\cite[Theorem~D.4]{e-z}),
the resolvent
$$
R(\lambda)=(P_1-\lambda)^{-1}|_V:H^{-k_1}_h(M)\cap V\to L^2(M)\cap V
$$
is a meromorphic family of operators in $\lambda\in \mathbb C$ with
poles of finite rank. Then:
\begin{enumerate}
\item Assume that the conditions of Proposition~\ref{l:global-grushin-empty} hold
and let $\delta>0$ be given by this proposition. Then for $h$
small enough, $R(\lambda)$ is holomorphic in $\{|\lambda|<\delta
h^r\}$ and $\|R(\lambda)\|_{L^2\cap V\to L^2}=O(h^{-r})$ in this
region.
\item Assume that the conditions of Proposition~\ref{l:global-grushin} hold
and let $(\lambda_0,\lambda_2^0)$ be the joint eigenvalue and
$\delta>0$ the constant given by this proposition. Suppose that
$\lambda_2^0=0$. Then for $h$ small enough,
$$
R(\lambda)=S(\lambda)+{\Pi\over\lambda-\lambda_0},\
|\lambda|<\delta h^r,
$$
where $S(\lambda)$ is holomorphic, $\Pi$ is a rank one operator, and
the $L^2\cap V\to L^2$ norms of $S(\lambda)$ and $\Pi$ are $O(h^{-N})$
for some constant $N$.
\end{enumerate}
\end{prop}
\begin{proof}
1. We have
$$
A_1(P_1-\lambda)=I-\lambda A_1+h^\infty\Psi^{-\infty}(M)\text{ on }V;
$$
the right-hand side is invertible for $\delta$ small enough. 
Therefore, $R$ has norm $O(h^{-r})$.

2. We know that $R(\lambda)$ has a pole at $\lambda$ if and only if
there exists nonzero $u\in L^2(M)\cap V$ such that
$(P_1-\lambda)u=0$; that is, a joint eigenfunction of $(P_1,P_2)$ with
joint eigenvalue $(\lambda, 0)$.  Therefore, $\lambda_0$ is the only
pole of $R(\lambda)$ in $\{|\lambda|<\delta h^r\}$.

Now, take $\lambda\neq\lambda_0$, $|\lambda|<\delta h^r$, and
assume that $v\in H^{-k_1}_h(M)\cap V$ and $u=R(\lambda)v\in
L^2(M)\cap V$. Let $T(\lambda)$ be the family of operators introduced
in the proof of Proposition~\ref{l:global-grushin}, with
$\lambda_1=\lambda$ and $\lambda_2=0$; we know that $T(\lambda)$ is
invertible. We represent $T(\lambda)^{-1}$ as a $4\times 4$
operator-valued matrix; let $T_{ij}^{-1}(\lambda)$ be its entries.  We
have $T(\lambda)(u,0,0,0)=(v,0,c,0)$ for some number $c$. However,
then $(u,0,0,0)=T(\lambda)^{-1}(v,0,c,0)$; taking the third entry of
this equality, we get
$T_{31}^{-1}(\lambda)v+T_{33}^{-1}(\lambda)c=0$. Now,
$T_{33}^{-1}(\lambda)=f_1(\lambda)$, with the latter introduced in the
proof of Proposition~\ref{l:global-grushin}.  Therefore, we can
compute $c$ in terms of $v$; substituting this into the expression for
$u$, we get the following version of the Schur complement formula:
\begin{equation}\label{e:schur-complement}
R(\lambda)=\bigg(T_{11}^{-1}(\lambda)-
{T_{13}^{-1}(\lambda)T_{31}^{-1}(\lambda)\over f_1(\lambda)}\bigg)\bigg|_{V}.
\end{equation}
Next, by the proof of Proposition~\ref{l:global-grushin},
$f_1(\lambda_0)=0$ and $f_1(\lambda)=\lambda+O(h^\infty)$. Therefore,
we may write $f_1(\lambda)=(\lambda-\lambda_0)/g(\lambda)$, with $g$
holomorphic and bounded by $O(1)$. Let $u_0$ be the joint
eigenfunction of $(P_1,P_2)$ with eigenvalue $(\lambda_0,0)$; then
$\Pi=-g(\lambda_0)T_{13}^{-1}(\lambda_0)T_{31}^{-1}(\lambda_0)$ is a
rank one operator, as $T_{13}^{-1}(\lambda_0)$ acts $\mathbb C\to V$
and $\Pi u_0=-(1+O(h^\infty))u_0$.  Since the operators $T_{ij}^{-1}$
are polynomially bounded in $h$, we are done.
\end{proof}

\subsection{Local Grushin problem}\label{s:prelim-grushin-2}

In this subsection, we show how to obtain information about the joint
spectrum of two operators $P_1,P_2$ based only on their behavior
microlocally near the set where neither of them is elliptic. For that,
we use global Grushin problems discussed in the previous subsection.
Assume that $P_1\in\Psi^{k_1}_{\cl}(M)$, $P_2\in\Psi^{k_2}_{\cl}(M)$
satisfy
\begin{enumerate}
\item[(E1)] The principal symbol $p_{j0}$ of $P_j$ is real-valued.
\item[(E2)] The symbol $p_{10}$ is elliptic in the class $S^{k_1}(M)$
outside of some compact set. As a corollary, the set
$$
K=\{(x,\xi)\in T^*M\mid p_{10}(x,\xi)=p_{20}(x,\xi)=0\}
$$
is compact.
\end{enumerate}
Next, assume that $A_1,A_2$ are compactly microlocalized operators on $M$ such that:
\begin{enumerate}
\item[(L1)] For each $j$ and every bounded neighborhood $U$ of $K$,
$A_j$ can be represented as $A'_j+A''_j$, where both $A'_j$ and
$A''_j$ are compactly microlocalized, $\|A'_j\|=O(h^{-r})$,
$\WF(A'_j)\subset U\times U$, and $A''_j\in
h^{-r}\Psi^{\comp}_{\cl}(M)$.  Here $r\geq 0$ is some constant.
\item[(L2)] The commutator of any two of the
operators $P_1,P_2,A_1,A_2$ lies in $h^\infty\Psi^{-\infty}(M)$.
\end{enumerate}
Finally, let $S_1:\mathbb C\to C^\infty(M)$,
$S_2:\mathcal D'(M)\to \mathbb C$ be compactly microlocalized operators such that:
\begin{enumerate}
\item[(L3)] $\|S_j\|=O(1)$ and $\WF(S_j)\subset K$.
\item[(L4)] $S_2S_1=1+O(h^\infty)$.
\item[(L5)] If $Q$ is any of the operators $P_1,P_2,A_1,A_2$, then
$QS_1\in h^\infty\Psi^{-\infty}$ and $S_2Q\in h^\infty\Psi^{-\infty}$.
\end{enumerate}
\begin{prop}\label{l:local-grushin}
1. If the conditions (E1)--(E2) and (L1)--(L2) hold, and
$$
A_1P_1+A_2P_2=I
$$
microlocally near $K\times K$, then there exists $\delta>0$ such
that for $h$ small enough, there are no joint eigenvalues of $P_1,P_2$
in the ball of radius $\delta h^r$ centered at zero.

2. If the conditions (E1)--(E2) and (L1)--(L5) hold, $[P_1,P_2]=0$, and
\begin{equation}\label{e:local-grushin-2}
A_1P_1+A_2P_2=I-S_1S_2
\end{equation}
microlocally near $K\times K$, then there exists $\delta>0$ such
that for $h$ small enough, the ball of radius $\delta h^r$
centered at zero contains exactly one joint eigenvalue $\lambda$ of
$P_1,P_2$. Moreover, $\lambda=O(h^\infty)$ and the corresponding joint
eigenspace is one dimensional.
\end{prop}
\begin{proof}
We will prove part 2; part 1 is handled similarly. Take small
$\varepsilon>0$ and let $\chi_\varepsilon\in C_0^\infty(\mathbb R)$
be supported in $(-\varepsilon,\varepsilon)$ and equal to 1 on
$[-\varepsilon/2,\varepsilon/2]$.  Also, let $\psi_\varepsilon\in
C^\infty(\mathbb R)$ satisfy
$t\psi_\varepsilon(t)=1-\chi_\varepsilon(t)$ for all $t$; then
$\psi_\varepsilon(t)=0$ for $|t|\leq \varepsilon/2$.  The function
$\psi_\varepsilon$ is a symbol of order $-1$, as it is equal to
$t^{-1}$ for $|t|\geq\varepsilon$.

By~(E1), we can define the operators
$\chi_\varepsilon[P_j],\psi_\varepsilon[P_j]\in \Psi^{\loc}_{\cl}(T^*M)$ using
the formal functional calculus introduced in
Section~\ref{s:prelim-functional}. By~(E2) and
Proposition~\ref{l:formal-functional-elliptic}
$\psi_\varepsilon[P_1]\in\Psi^{-k_1}_{\cl}(M)$, and
$\chi_\varepsilon[P_1]\in\Psi^{\comp}_{\cl}$. Therefore, we can define
uniquely up to $h^\infty\Psi^{-\infty}$ the operators
\begin{equation}\label{e:local-grushin-op}
X_\varepsilon=\chi_\varepsilon[P_1]\chi_\varepsilon[P_2]\in \Psi^{\comp}_{\cl}(M),\
\psi_\varepsilon[P_1]\in\Psi^{-k_1}_{\cl}(M),\
\chi_\varepsilon[P_1]\psi_\varepsilon[P_2]\in\Psi^{\comp}_{\cl}(M).
\end{equation}
By Proposition~\ref{l:formal-functional-calculus}, these operators
commute with each other and with $P_1,P_2$ modulo
$h^\infty\Psi^{-\infty}$. Let $Y$ be any of the operators
in~\eqref{e:local-grushin-op}; we will show that it commutes with each
$A_j$ modulo $h^\infty\Psi^{-\infty}$.  Take a neighborhood $U$ of $K$
so small that $|p_{10}|+|p_{20}|\leq \varepsilon/4$ on $U$; then $Y$ is
either zero or the identity operator microlocally on $U$. By~(L1),
decompose $A_j=A'_j+A''_j$, where $\WF(A'_j)\subset U\times U$ and
$A''_j\in\Psi^{\comp}_{\cl}$.  We have $A_j=A''_j$ microlocally away
from $U\times U$; therefore, $[A''_j,P_k]=0$ microlocally near
$T^*M\setminus U$.  By Proposition~\ref{l:formal-functional-calculus},
$[A''_j,Y]=0$ microlocally near $T^*M\setminus U$; therefore, the
commutator $[A_j,Y]$ is compactly microlocalized and
$\WF([A_j,Y])\subset U\times U$.  However, since $Y=0$ or $Y=I$
microlocally in $U$, we have $[A_j,Y]\in h^\infty\Psi^{-\infty}$, as
needed.

Since $X_\varepsilon=I$ microlocally near $K$ and $\WF(S_j)\subset K$, we
get $(I-X_\varepsilon)S_1,S_2(I-X_\varepsilon)\in h^\infty\Psi^{-\infty}$.
Multiplying~\eqref{e:local-grushin-2} by $X_\varepsilon$, we get for
$\varepsilon$ small enough,
\begin{equation}\label{e:local-grushin-1.1}
(X_\varepsilon A_1)P_1+(X_\varepsilon A_2)P_2+S_1S_2=X_\varepsilon\mod h^\infty\Psi^{-\infty}.
\end{equation}
Next, by Proposition~\ref{l:formal-functional-calculus}
\begin{equation}\label{e:local-grushin-1.2}
\psi_\varepsilon[P_1]P_1+\chi_\varepsilon[P_1]\psi_\varepsilon[P_2]P_2=I-X_\varepsilon\mod h^\infty\Psi^{-\infty}.
\end{equation}
Adding these up, we get
\begin{equation}
(X_\varepsilon A_1+\psi_\varepsilon[P_1])P_1+(X_\varepsilon A_2+\chi_\varepsilon[P_1]\psi_\varepsilon[P_2])P_2
+S_1S_2=I+h^\infty\Psi^{-\infty}.
\end{equation}
The operators $P_1$, $P_2$, $\widetilde A_1=X_\varepsilon
A_1+\psi_\varepsilon[P_1]$, $\widetilde A_2=X_\varepsilon
A_2+\chi_\varepsilon[P_1]\psi_\varepsilon[P_2]$, $S_1$, $S_2$ satisfy the
assumptions of Proposition~\ref{l:global-grushin}. Applying it, we get
the desired spectral result.
\end{proof}

\subsection{Proof of Proposition~\ref{l:global-grushin} in the general case}
\label{s:prelim-grushin-3}

In this subsection, we prove Proposition~\ref{l:global-grushin} for
the general case of $n\geq 2$ operators.  For simplicity, we assume
that $k_1=\dots=k_n=0$; that is, each $P_j$ lies in $\Psi^0(M)$.  (If
this is not the case, one needs to replace $L^2(M)$ below with certain
semiclassical Sobolev spaces.)

Let $V$ be the space of all exterior forms on $\mathbb C^n$; we can
represent it as $V_{\Even}\oplus V_{\Odd}$, where
$$
\begin{gathered}
V_{\Even}=\bigoplus_{j\geq 0}\Lambda^{2j} \mathbb C^n,\
V_{\Odd}=\bigoplus_{j\geq 0}\Lambda^{2j+1} \mathbb C^n
\end{gathered}
$$
are the vector spaces of the even and odd degree forms, respectively.
Note that $V_{\Even}$ and $V_{\Odd}$ have the same dimension.  Define
the spaces
$$
L^2_{\Even}=L^2(M)\otimes V_{\Even},\
L^2_{\Odd}=L^2(M)\otimes V_{\Odd},\
L^2_V=L^2(M)\otimes V.
$$
We call elements of $L^2_V$ forms. They posess properties similar to
those of differential forms; beware though that they are not
differential forms in our case. We will use the families of operators
$(A_j)$ and $(P_j)$ to define the operators
$$
d_P,d^*_A:L^2_V\to L^2_V,
$$
given by the formulas
$$
\begin{gathered}
d_P(u\otimes v)=\sum_{j=1}^n (P_j u)\otimes (e_j\wedge v),\\
d^*_A(u\otimes v)=\sum_{j=1}^n (A_j u)\otimes (i_{e_j} v);\\
u\in L^2(M),\
v\in V.
\end{gathered}
$$
Here $e_1,\dots,e_n$ is the canonical basis of $\mathbb C^n$.  The
notation $i_{e_j}$ is used for the interior product by $e_j$; this is
the adjoint of the operator $v\mapsto e_j\wedge v$ with respect to the
inner product on $V$ induced by the canonical bilinear inner product
on $\mathbb C^n$.  Note that $d_P$ and $d^*_A$ map even forms to
odd and vice versa.

A direct calculation shows that under the assumptions (G1)--(G6),
\begin{equation}\label{e:forms-eq-1}
(d_P+d^*_A)^2=I-S_1S_2\otimes I_V+O_{\Psi^{-\infty}}(h^\infty).
\end{equation}
Here $I_V$ is the identity operator on $V$, while $I$ is the identity
operator on $L^2_V$. Moreover, since the operators $P_1,\dots,P_n$
commute exactly, we have
\begin{equation}\label{e:forms-eq-2}
d_P^2=0.
\end{equation}
For $\lambda=(\lambda_1,\dots,\lambda_n)\in\mathbb C^n$, define the
operator
$$
\begin{gathered}
T(\lambda)=\begin{pmatrix}(d_{P-\lambda}+d^*_A)|_{L^2_{\Even}}&S_1\otimes I_V\\
S_2\otimes I_V&0
\end{pmatrix}:\mathcal H_1\to \mathcal H_2,\\
\mathcal H_1=L^2_{\Even}\oplus V_{\Odd},\
\mathcal H_2=L^2_{\Odd}\oplus V_{\Even}.
\end{gathered}
$$
Here $d_{P-\lambda}$ is defined using the operators
$P_1-\lambda_1,\dots,P_n-\lambda_n$ in place of $P_1,\dots,P_n$.  It
follows from~\eqref{e:forms-eq-1} that for
$$
Q=\begin{pmatrix} (d_P+d^*_A)|_{L^2_{\Odd}}&S_1\otimes I_V\\
S_2\otimes I_V&0\end{pmatrix}:\mathcal H_2\to \mathcal H_1,
$$
we have $QT(0)=I+O_{\mathcal H_1\to \mathcal H_1}(h^\infty)$,
$T(0)Q=I+O_{\mathcal H_2\to \mathcal H_2}(h^\infty)$.  Moreover, it
follows from (G1) and (G3) that $\|Q\|_{\mathcal H_2\to \mathcal
H_1}=O(h^{-r})$.  Therefore, for $|\lambda|\leq\delta h^r$ and
$h$ and $\delta>0$ small enough, the operator $T(\lambda)$ is
invertible, with $\|T(\lambda)^{-1}\|_{\mathcal H_2\to \mathcal
H_1}=O(h^{-r})$.

Assume that $|\lambda|\leq \delta h^r$ and let $\mathbf 1\in
V_{\Even}$ be the basic zero-form on $\mathbb C^n$. Put
$(\alpha(\lambda),v(\lambda))=T(\lambda)^{-1}(0,\mathbf 1)$, where
$\alpha(\lambda)\in L^2_{\Even}$, $v(\lambda)\in V_{\Odd}$; then
$(\alpha(\lambda),v(\lambda))$ is the unique solution to the system
\begin{equation}\label{e:forms-eq-3}
\begin{gathered}
(d_{P-\lambda}+d^*_A)\alpha(\lambda)+S_1(1)\otimes v(\lambda)=0,\\
(S_2\otimes I_V)\alpha(\lambda)=\mathbf 1.
\end{gathered}
\end{equation}
We further write $v(\lambda)=f(\lambda)+w(\lambda)$, where
$f(\lambda)$ is a 1-form and $w(\lambda)$ is a sum of forms of degree
3 or more. Note that both $f$ and $w$ are holomorphic functions of
$\lambda$, with $f(\lambda)\in\mathbb C^n$.

We claim that $\lambda$ is a joint eigenvalue of $P_1,\dots,P_n$ if
and only if $f(\lambda)=0$. First of all, if $u$ is a joint
eigenfunction, then $T(\lambda)(u\otimes \mathbf 1,0)=c(0,\mathbf 1)$
for some scalar $c\neq 0$; therefore, $f(\lambda)=0$ and the joint
eigenspace is one dimensional.

Now, assume that $f(\lambda)=0$. We will prove that the solution
to~\eqref{e:forms-eq-3} satisfies $\alpha(\lambda)=u\otimes \mathbf 1$ for some
$u\in L^2(M)$; it follows immediately that
$(P_1-\lambda_1)u=\dots=(P_n-\lambda_n)u=0$.  Let $\alpha=u\otimes
1+\beta$, where $\beta$ is a sum of forms of degree 2 or higher. Then
by~\eqref{e:forms-eq-2},
\begin{equation}\label{e:forms-eq-4}
(d_{P-\lambda}+d^*_A)^2(u\otimes 1)\in L^2(M)\otimes \mathbf 1.
\end{equation}
Next, we get from~\eqref{e:forms-eq-3}
$$
(S_2\otimes I_V)(d_{P-\lambda}+d^*_A)\alpha+(1+O(h^\infty))v=0.
$$
The components of this equation corresponding to odd forms of degree 3
or higher depend only on $\beta$ and $w$; therefore, for $h$ small
enough, $w=W\beta$ for some operator $W$ of norm $O(h^{-r})$.  Since
$f=0$, we get $v=W\beta$; therefore, by~\eqref{e:forms-eq-4}
and~\eqref{e:forms-eq-3} multiplied by $d_{P-\lambda}+d^*_A$,
$$
(d_{P-\lambda}+d^*_A)^2\beta+(d_{P-\lambda}+d^*_A)(S_1(1)\otimes W\beta)\in L^2(M)\otimes\mathbf 1.
$$
Taking the components of this equation corresponding to forms
of even degree 2 or higher and recalling~\eqref{e:forms-eq-1},
we get
$$
((I-S_1S_2)\otimes I_V+O(\delta+h^\infty))\beta=0.
$$
However, $(S_2\otimes I_V)\beta=0$ by~\eqref{e:forms-eq-3}; therefore,
$$
(I+O(\delta+h^\infty))\beta=0.
$$
It follows that $\beta=0$ and the claim is proven.

It remains to show that the equation $f(\lambda)=0$ has exactly one
solution in the disk of radius $\delta h^r$. For that, we write
$QT(\lambda)=I-R(\lambda)$,
$$
T(\lambda)^{-1}=(I+R(\lambda)+(I-R(\lambda))^{-1}R(\lambda)^2)Q.
$$
We have $Q(\lambda)(0,\mathbf 1)=(S_1(1)\otimes \mathbf 1,0)$
and
$$
R(\lambda)=\begin{pmatrix}(d_P+d^*_A)d_\lambda&0\\
(S_2 \otimes I_V)d_\lambda&0\end{pmatrix}+O(h^\infty).
$$
Here $d_\lambda$ is constructed using $\lambda_1,\dots,\lambda_n$ in
place of $P_1,\dots,P_n$.  Now, we use that $R(\lambda)^2Q(0,\mathbf
1)=O_{\mathcal H_1}(h^\infty)$ to conclude that
$f(\lambda)=\lambda-g(\lambda;h)$ with $g=O(h^\infty)$; it then
remains to use the contraction mapping principle.


\section{Numerical results}\label{s:numerics}

\subsection{Overview}\label{s:numerics-intro}

This section describes a procedure for computing the quantization
symbol $\mathcal F(m,l,k)$ from Theorem~\ref{l:theorem-asymptotics} to
an arbitrarily large order in the case
\begin{equation}\label{e:numerics-l-prime}
l'=l-|k|=O(1).
\end{equation}
The reason for the restriction $l'=O(1)$ is because then we can use
bottom of the well asymptotics for eigenvalues of the angular
operator; otherwise, we would have to deal with nondegenerate
trajectories, quantization conditions for which are harder to compute
numerically; see for example~\cite{cdv}.

We first use the equation~\eqref{e:final-q-eq}; once we get rid of the
semiclassical parameter $h$ (remembering that the original problem was
$h$-independent), the number $\omega=\mathcal F(m,l,k)$ is the
solution to the equation
\begin{equation}\label{e:numerics-q-eq}
\mathcal G^r(m,\omega,k)=\mathcal G^\theta(l',\omega,k).
\end{equation}
Here $\mathcal G^r,\mathcal G^\theta$ are the non-semiclassical
analogues of~$\mathcal F^r,\mathcal F^\theta$;
namely, \eqref{e:radial-qc} and~\eqref{e:angular-qc} take the form
$$
\begin{gathered}
\lambda = \mathcal G^r(m,\omega,k)\sim\sum_{j\geq 0} \mathcal G^r_j(m,\omega,k),\\
\lambda = \mathcal G^\theta(l',\omega,k)\sim\sum_{j\geq 0} \mathcal G^\theta_j(l',\omega,k),
\end{gathered}
$$
respectively. The functions $\mathcal G^r_j,\mathcal G^\theta_j$ are
homogeneous of degree $2-j$ in the following sense:
\begin{equation}\label{e:numerics-homogeneous}
\begin{gathered}
\mathcal G^r_j(m,M_s\omega,sk)=s^{2-j} \mathcal G^r_j(m,\omega,k),\
\mathcal G^\theta_j(l',M_s\omega,sk)=s^{2-j} \mathcal G^\theta_j(l',\omega,k),\
s>0.
\end{gathered}
\end{equation}
Here $M_s\omega=s\Real\omega+i\Imag\omega$; the lack of dilation in
the imaginary part of $\omega$ reflects the fact that it is very close
to the real axis.

\begin{figure}
\includegraphics[width=17.5cm]{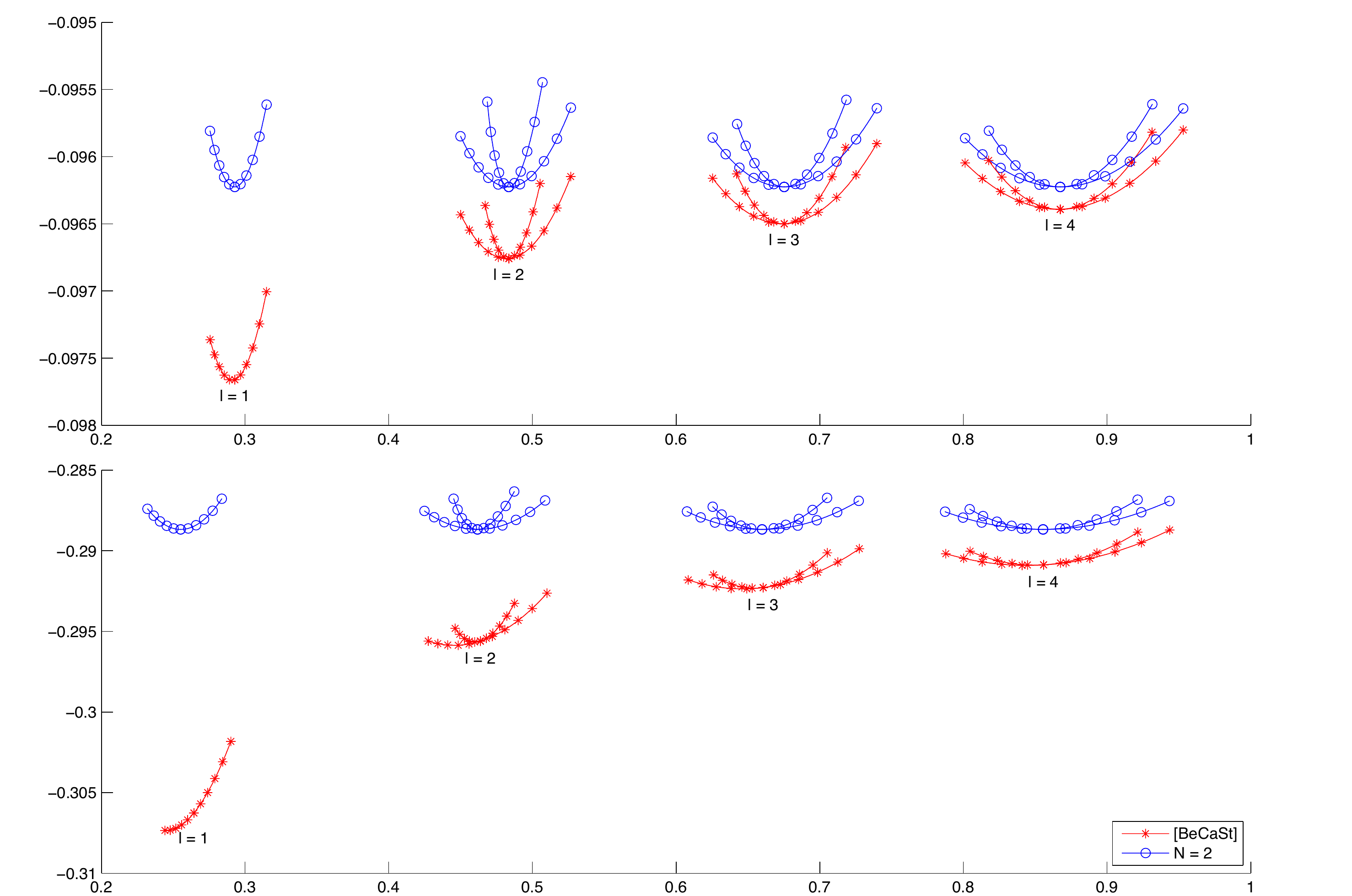}
\caption{Comparison of order 2 approximation to QNMs with
the data of~\cite{b-c-s}. Here $l=1,\dots,4$, $k=-l,-l+1,l-1,l$
(left to right), and $m=0$ (top) and $1$ (bottom).}
\label{f:numerics-1}
\end{figure}

We will describe how to compute $\mathcal G^r_j,\mathcal G^\theta_j$
for an arbitrary value of $j$ in Section~\ref{s:numerics-specific}.
The method is based on a quantization condition for barrier-top
resonances, studied in Section~\ref{s:radial-barrier}; their
computation is explained in Section~\ref{s:numerics-general} and a
MATLAB implementation and data files for several first QNMs can be
found online at \url{http://math.berkeley.edu/~dyatlov/qnmskds}.  We
explain why the presented method gives the quantization conditions of
Propositions~\ref{l:radial} and~\ref{l:angular}, but we do not provide
a rigorous proof.

\begin{figure}
\includegraphics[width=17cm]{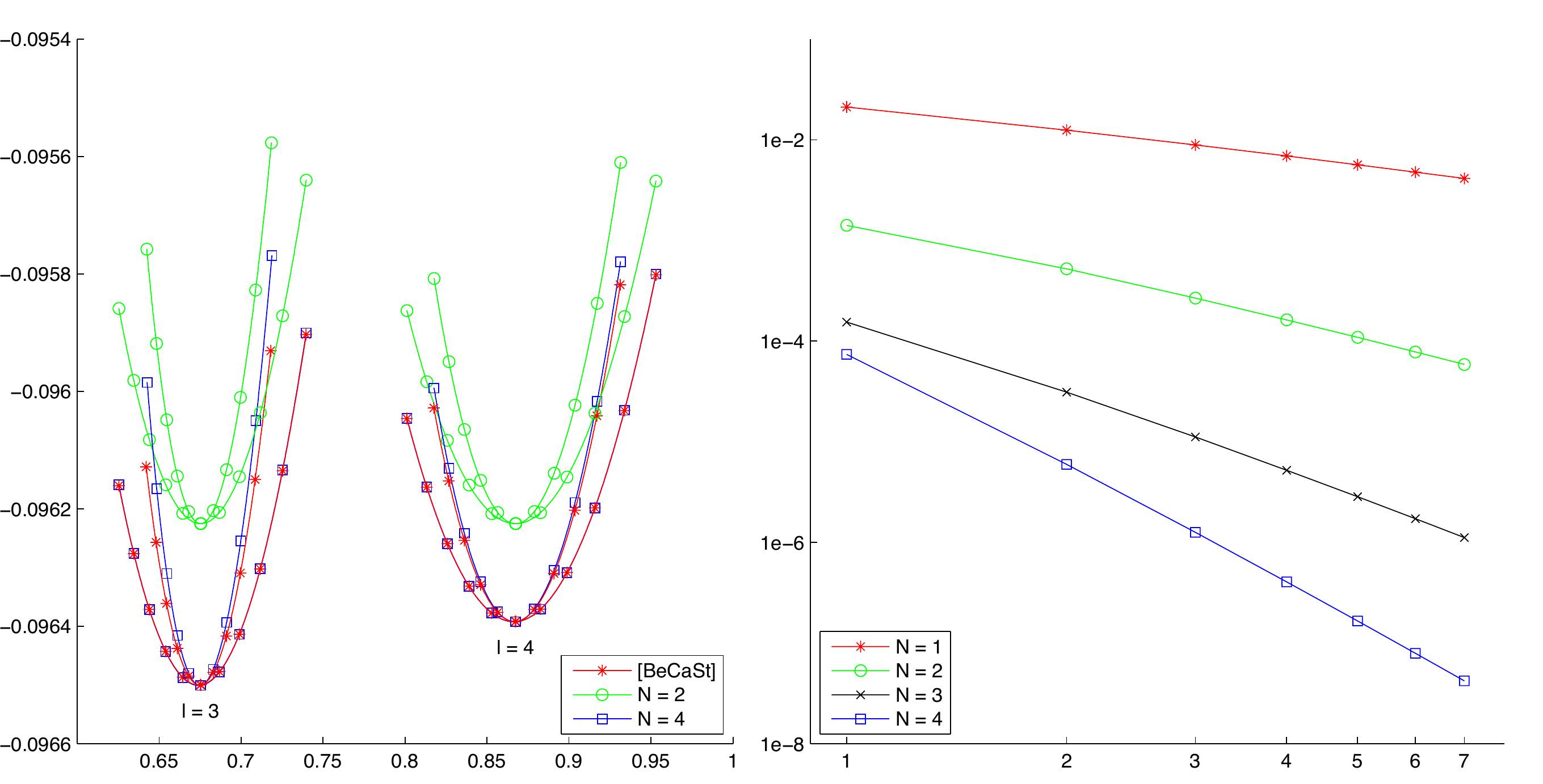}
\caption{Left: comparison of order 2 and 4 approximations
to QNMs with the data of~\cite{b-c-s}. Here $l=3, 4$, $l'=0,1$,
and $m=0$.  Right: log-log plot of the error of order 1--4
approximations to QNMs, as compared to~\cite{b-c-s}. Here
$a=0.1$, $k=l$, $m=0$, and $l$, plotted on the $x$ axis, ranges from
$1$ to $7$.}
\label{f:numerics-2}
\end{figure}

We now compare the pseudopoles given by quantization conditions to
QNMs for the Kerr metric\footnote{The results of the present paper do
  not apply to the Kerr case $\Lambda=0$, due to lack of control on
  the scattering resolvent at the asymptotically flat spatial
  infinity. However, the resonances described by~\eqref{e:our-q} are
  generated by trapping, which is located in a compact set; therefore,
  we can still make sense of the quantization condition and compute
  approximate QNMs.} computed by the authors of~\cite{b-c-s} using
Leaver's continued fraction method~--- see~\cite[Section~4.6]{b-c-s}
for an overview of the method and~\cite[Appendix~E]{b-c-c}
and~\cite[Section~IV]{b-k} for more details. The QNM data for the case
of scalar perturbations, studied in this paper, computed using Leaver's
method can be found online at
\url{http://www.phy.olemiss.edu/~berti/qnms.html}.

Figure~\ref{f:numerics-1} compares the second order approximation to
QNMs (that is, solution to the equation~\eqref{e:numerics-q-eq}
constructed using $\mathcal G^r_j$ and $\mathcal G^\theta_j$ for
$j\leq 2$) to the QNMs of~\cite{b-c-s}. Each branch on the picture
shows the trajectory of the QNM with fixed parameters $m,l,k$ for
$a\in [0,0.25]$; the marked points correspond to
$a=0,0.05,\dots,0.25$.  The branches for same $m,l$ and different $k$
converge to the Schwarzschild QNMs as $a\to 0$.  We see that the
approximation gets better when $l$ increases, but worse if one
increases $m$; this agrees well with the fact that the computed
quantization conditions are expected to work when $l$ is large and $m$
is bounded.

The left part of
Figure~\ref{f:numerics-2} compares the second and fourth order
approximations with the QNMs of~\cite{b-c-s} (with the same values of
$a$ as before); we see that the fourth order approximation is
considerably more accurate than the second order one, and the former
is more accurate for a smaller value of $l'$.  Finally, the right part
of Figure~\ref{f:numerics-2} is a log-log plot of the error of
approximations of degree 1 through 4, as a function of $l$; we see
that the error decreases polynomially in $l$.

\subsection{Barrier-top resonances}\label{s:numerics-general}

Here we study a general spectral problem to which we will reduce both
the radial and the angular problems in the next subsection. Our
computation is based on the following observation: when the
quantization condition of Section~\ref{s:radial-barrier} is satisfied,
the function $u_+$ has the microlocal form~\eqref{e:radial-wkb-form},
with the symbol behaving like $(r-r_0)^m$ near the trapped set. This
can be seen from the proof of Proposition~\ref{l:radial-last}: if
$\beta=-ihm$, then $\tilde u_\pm^+(x)=x^m$ and $B_1\tilde u_\pm^+$ has
to have the form~\eqref{e:radial-wkb-form}. The calculations below
are similar to~\cite[Section~3]{d-s}.

Consider the operator
\begin{equation}\label{e:numerics-general-form}
P_y=D_yA(y)D_y+B(y;\omega,k).
\end{equation}
Here the function $A(y)$ is independent of $\omega,k$, real-valued,
and $A(0)>0$; $B$ is a symbol of order 2:
$$
B(y;\omega,k)\sim\sum_{j\geq 0} B_j(y;\omega,k),
$$ 
with $B_j$ homogeneous of degree $2-j$ in the sense of~\eqref{e:numerics-homogeneous}.
We also require that $B_0$
be real-valued and
$$
B'_0(0;\omega,k)=0,\
B''_0(0;\omega,k)<0.
$$
We will describe an algorithm to find the
quantization condition for eigenvalues $\lambda$ of $P_y$
with eigenfunctions having the outgoing WKB form~\eqref{e:numerics-wkb} near $y=0$;
we will compute $\lambda$ as a symbol of order 2:
$$
\lambda\sim \sum_{j\geq 0}\lambda_j(\omega,k),\
\lambda_j(M_s\omega,sk)=s^{2-j}\lambda(\omega,k),\ s>0.
$$
More precisely, we will show how to inductively compute each
$\lambda_j$.  The principal part $\lambda_0$ is given by the following
barrier-top condition:
\begin{equation}\label{e:numerics-lambda0}
\lambda_0=B_0(0;\omega,k).
\end{equation}
In this case, we have
$$
B_0(y;\omega,k)=\lambda_0(\omega,k)-y^2U_0(y;\omega,k),
$$
where $U_0$ is a smooth function, and $U_0(0)=-V''(0)/2>0$.  Define
the phase function $\psi_0(y;\omega,k)$ such that
$$
\psi'_0(y;\omega,k)=y\sqrt{U_0(y;\omega;k)/A(y)};
$$
note that $\psi_0$ is homogeneous of degree 1. We will look
for eigenfunctions of the WKB form
\begin{equation}\label{e:numerics-wkb}
u(y;\omega,k)=e^{i\psi_0(y;\omega,k)} a(y;\omega,k),
\end{equation}
solving the equation $P_y u=\lambda u$ up to
$O(|\omega|+|k|)^{-\infty}$ error near $y=0$. Here $a$ is a symbol of
order zero:
$$
a(y;\omega,k)\sim\sum a_j(y;\omega,k),
$$
with $a_j$ homogeneous of order $-j$. 

Substituting~\eqref{e:numerics-wkb} into the equation $P_y u=\lambda
u$ and gathering terms with the same degree of homogeneity, we get the
following system of transport equations:
\begin{equation}\label{e:numerics-transport}
\begin{gathered}
(L_0-B_1+\lambda_1)a_j=-L_1a_{j-1}+\sum_{0<l\leq j}(B_{l+1}-\lambda_{l+1}) a_{j-l},\ j\geq 0,\\
L_0=2i\psi'_0 A \partial_y+i(A\psi'_0)'
=2i\sqrt{U_0(y)A(y)}y\partial_y
+i(y\sqrt{U_0(y)A(y)})';\\
L_1=\partial_y A(y) \partial_y,
\end{gathered}
\end{equation}
with the convention $a_{-1}=0$.

Now, consider the space of infinite sequences
$$
\mathbb C^\infty = \{\mathbf a=(\mathbf a^j)_{j=0}^\infty\mid
\mathbf a^j\in \mathbb C\}
$$
and the operator $T:C^\infty(\mathbb R)\to \mathbb C^\infty$
defined by
$$
T(a)=\mathbf a,\
\mathbf a_j=\partial_y^j a(0)/j!.
$$
Let the operators $\mathbf L_j,\mathbf B_j:\mathbb C^\infty\to \mathbb C^\infty$
be defined by the relations
$$
TL_j=\mathbf L_j T,\
TB_j=\mathbf B_j T.
$$
We treat $\mathbf L_j,\mathbf B_j$ as infinite dimensional matrices.
We see that each $\mathbf B_j$ is lower triangular, with elements on
the diagonal given by $B_j(0)$; $(\mathbf L_1)_{jk}=0$ for $j+2<k$. As
for $\mathbf L_0$, due to the factor $y$ in front of the
differentiation it is lower triangular and
$$
(\mathbf L_0)_{jj}=i(2j+1)\sqrt{U_0(0)A(0)}.
$$
One can show that there exists a smooth nonzero function $a_0$ solving
$(L_0-B_1+\lambda_1)a_0=0$ if and only if one of the diagonal elements
of the matrix $\mathbf L_0-\mathbf B_1+\lambda_1$ is zero (the kernel
of this matrix being spanned by $Ta_0$).  Let $m\geq 0$ be the index
of this diagonal element; this will be a parameter of the quantization
condition. We can now find
\begin{equation}\label{e:numerics-lambda1}
\lambda_1=B_1(0)-i(2m+1)\sqrt{U_0(0)A(0)}.
\end{equation}
Now, there exists a nonzero functional $\mathbf f$ on $\mathbb
C^\infty$, such that $\mathbf f(\mathbf a)$ depends only on $\mathbf
a^0,\dots,\mathbf a^m$, and $\mathbf f$ vanishes on the image of
$\mathbf L_0-\mathbf B_1+\lambda_1$.  Moreover, one can show that the
equation $(L_0-B_1+\lambda_1)a=b$ has a smooth solution $a$ if and
only if $\mathbf f(Tb)=0$.

Take $a_0$ to be a nonzero element of the kernel of
$L_0-B_1+\lambda_1$; we normalize it so that $\mathbf f(Ta_0)=0$. Put
$\mathbf a_j=Ta_j$; then the transport equations become
\begin{equation}\label{e:numerics-a}
(\mathbf L_0-\mathbf B_1+\lambda_1)\mathbf a_j=-\mathbf L_1 \mathbf a_{j-1}
+\sum_{0<l\leq j} (\mathbf B_{l+1}-\lambda_{l+1})\mathbf a_{j-l},\ j>0.
\end{equation}
We normalize each $a_j$ so that $\mathbf f(\mathbf a_j)=0$ for $j>0$.
The $j$-th transport equation has a solution if and only if the
$\mathbf f$ kills the right-hand side, which makes it possible to find
\begin{equation}\label{e:numerics-lambda2}
\lambda_{j+1}=\mathbf f\Big(-\mathbf L_1a_{j-1}+\sum_{0<l<j}\mathbf B_{l+1}\mathbf a_{j-l}\Big),\
j>0. 
\end{equation}
Using the equations~\eqref{e:numerics-lambda0},
\eqref{e:numerics-lambda1},
\eqref{e:numerics-lambda2}, and~\eqref{e:numerics-a},
we can find all $\lambda_j$ and $\mathbf a_j$ inductively.

\subsection{Radial and angular quantization conditions}\label{s:numerics-specific}

We start with the radial quantization condition. Consider the original
radial operator
$$
\begin{gathered}
P_r=D_r(\Delta_r D_r)+V_r(r;\omega,k),\\
V_r(r;\omega,k)=-\Delta_r^{-1}(1+\alpha)^2((r^2+a^2)\omega-ak)^2.
\end{gathered}
$$
It has the form~\eqref{e:numerics-general-form}, with
\begin{equation}\label{e:numerics-radial-form}
y=r-r_0,\
A(y)=\Delta_r,\
B(y;\omega,k)=V_r(r;\omega,k).
\end{equation}
Here $r_0$ is the point where $V_r$ achieves its maximal value,
corresponding to the trapped point $x_0$ in
Section~\ref{s:radial-trapping}. Now the previous subsection applies,
with the use of the outgoing microlocalization mentioned in the
beginning of that subsection.  Using~\eqref{e:numerics-lambda0}
and~\eqref{e:numerics-lambda1}, we can compute near $a=0$,
\begin{equation}\label{e:numerics-radial-stuff}
\begin{gathered}
r_0=3M_0-{2ak(1-9\Lambda M_0^2)\over 9 M_0\Real \omega}+O(a^2(|k|^2+|\omega|^2)),\\
\mathcal G^r_0={27M_0^2\over 1-9\Lambda M_0^2}
\bigg(1-{2ak\over 9M_0^2\Real\omega}\bigg)(\Real\omega)^2+O(a^2(|k|^2+|\omega|^2)),\\
\mathcal G^r_0+\mathcal G^r_1=\bigg[
i(m+1/2)+{3\sqrt 3 M_0\omega\over \sqrt{1-9\Lambda M_0^2}}\bigg]^2
+O(1)\text{ for }a=0;
\end{gathered}
\end{equation}
reintroducing the semiclassical parameter, we get the formulas
for $\mathcal F^r$ in Proposition~\ref{l:radial}.
\medskip

Now, we consider the angular problem. Without loss of generality, we
assume that $k>0$. After the change of variables $y=\cos\theta$, the
operator $P_\theta|_{\mathcal D'_k}$ takes the form
$$
P_y=D_y(1-y^2)(1+\alpha y^2)D_y+{(1+\alpha)^2(a\omega (1-y^2)-k)^2\over (1-y^2)(1+\alpha y^2)}.
$$
We are now interested in the bottom of the well asymptotics for the
eigenvalues of $P_y$, with the parameter $l'$
from~\eqref{e:numerics-l-prime} playing the role of the quantization
parameter $m$. The critical point for the principal symbol of the
operator $P_y$ is $(0,0)$. To reduce the bottom of the well problem to
the barrier-top problem, we formally rescale in the complex plane,
introducing the parameter $y'=e^{i\pi/4}y$, so that $(y')^2=iy^2$. We
do not provide a rigorous justification for such an operation; we only
note that the WKB solution of~\eqref{e:numerics-wkb} looks like
$e^{ic(y')^2}a=e^{-cy^2}a$ near $y=0$ for some positive constant $c$;
therefore, it is exponentially decaying away from the origin,
reminding one of the exponentially decaying Gaussians featured in the
bottom of the well asymptotics (see for
example~\cite[Section~3]{d-s} or the
discussion following~\cite[Proposition~4.3]{sb-z}).  There is a similar
calculation of the bottom of the well resonances based on quantum
Birkhoff normal form; see for example~\cite{cdv-g}. The rescaled
operator $P_{y'}=-iP_y$ takes the
form~\eqref{e:numerics-general-form}, with $y'$ taking the place of
$y$ and
\begin{equation}\label{e:numerics-angular-form}
A(y')=(1+i(y')^2)(1-i\alpha(y')^2),\
B(y';\omega,k)=-{i(1+\alpha)^2(a\omega (1+i(y')^2)-k)^2\over (1+i(y')^2)(1-i\alpha(y')^2)}.
\end{equation}
We can now formally apply the results of
Section~\ref{s:numerics-general}; note that, even though $A$ and $B$
are not real-valued, we have
$$
\begin{gathered}
A(0)=1,\
B_0(0)=-i(1+\alpha)^2(a\Real\omega-k)^2,\
B''_0(0)<0.
\end{gathered}
$$
An interesting note is that when $a=0$ and $k>0$, the process
described in Section~\ref{s:numerics-general} gives the spherical
harmonics $\lambda=l(l+1)$ exactly and without the
assumption~\eqref{e:numerics-l-prime}.  In fact, the first three terms
of the asymptotic expansion of $\lambda$ sum to $l(l+1)$ and the
remaining terms are zero.

\medskip

\addcontentsline{toc}{section}{Acknowledgements}

\noindent\textbf{Acknowledgements.} I would like to thank Maciej Zworski
for suggesting this problem, helpful advice, and constant
encouragement.
I would also like to thank Andr\'as Vasy for
sharing with me early versions of~\cite{v} and his
interest in the project, Michael Hitrik for guiding me through the
complexities of~\cite{h-s}, and Kiril Datchev, Hamid Hezari, and Jakub
Kominiarczuk for friendly consultations and references. I am
especially grateful to Emanuele Berti and Vitor Cardoso for providing
the data on quasi-normal modes for the Kerr metric in the case of
scalar perturbations, used in Appendix~\ref{s:numerics-intro}.
Finally, I would like to thank an anonymous referee for carefully
reading the manuscript and many suggestions for improving it. 


\end{document}